\documentclass[pdflatex,sn-nature]{sn-jnl}
\theoremstyle{thmstyleone}%
\newtheorem{theorem}{Theorem}
\newtheorem{proposition}[theorem]{Proposition}%
\newtheorem{assumption}[theorem]{Assumption}%
\newtheorem{lemma}[theorem]{Lemma}%
\newtheorem{corollary}[theorem]{Corollary}%

\theoremstyle{thmstyletwo}%

\theoremstyle{thmstylethree}%

\usepackage{graphicx}%
\usepackage{multirow}%
\usepackage{amsmath,amssymb,amsfonts}%
\usepackage{amsthm}%
\usepackage{mathrsfs}%
\usepackage[title]{appendix}%
\usepackage{textcomp}%
\usepackage{manyfoot}%
\usepackage{booktabs}%
\usepackage{algorithm}%
\usepackage{algorithmicx}%
\usepackage{algpseudocode}%
\usepackage{listings}%
\usepackage{wrapfig}
\usepackage{bookmark}
\usepackage{xfrac}
\usepackage[table]{xcolor}

\usepackage{natbib}

\usepackage{tikz}
\def\checkmark{\tikz\fill[scale=0.4](0,.35) -- (.25,0) -- (1,.7) -- (.25,.15) -- cycle;} 
\newcommand\crossmark[1][]{%
  \tikz[scale=0.4,#1]{
    \fill(0,0)--(0.1,0) .. controls (0.5,0.4) .. (1,0.7)--(0.9,0.7) ..  controls (0.5,0.5) ..(0,0.1) --cycle;
    \fill(1,0.1)--(0.9,0.1) .. controls (0.5,0.3) .. (0,0.7)--(0.1,0.7) .. controls (0.5,0.4) ..(1,0.2) --cycle;
  }%
}

\newcommand{\norm}[1]{\left\lVert#1\right\rVert}
\DeclareMathOperator*{\argmin}{arg\,min}

\newcommand\restr[2]{{
  \left.\kern-\nulldelimiterspace 
  #1 
  \vphantom{\big|} 
  \right|_{#2} 
  }}

\newcommand{\dif}[0]{\mathrm{d}} 

\begin{document}

\rowcolors{2}{gray!10}{white}

\title[GEORCE: A Fast New Control Algorithm for Computing Geodesics]{GEORCE: A Fast New Control Algorithm for Computing Geodesics}


\author*[1]{\fnm{Frederik} \sur{Möbius Rygaard}}\email{fmry@dtu.dk}
\author*[1]{\fnm{Søren} \sur{Hauberg}}\email{sohau@dtu.dk}

\affil*[1]{\orgdiv{DTU Compute}, \orgname{Technical University of Denmark (DTU)}, \orgaddress{\street{Anker Engelundsvej 1}, \city{Kongens Lyngby}, \postcode{2800}, \country{Denmark}}}



\abstract{Computing geodesics for Riemannian manifolds is a difficult task that often relies on numerical approximations. However, these approximations tend to be either numerically unstable, have slow convergence, or scale poorly with manifold dimension and number of grid points. We introduce a new algorithm called \textit{GEORCE} that computes geodesics in a local chart via a transformation into a discrete control problem. We show that \textit{GEORCE} has global convergence and asymptotic quadratic local convergence. In addition, we show that it extends to Finsler manifolds. For both Finslerian and Riemannian manifolds, we thoroughly benchmark GEORCE against several alternative optimization algorithms and show empirically that it has a much faster and more accurate performance for a variety of manifolds, including key manifolds from information theory and manifolds that are learned using generative models.}




\maketitle

\section{Introduction}\label{sec1}
Reliably computing geodesics that connect point pairs is a constant source of frustration to practitioners working with non-trivial Riemannian manifolds. Such manifolds see common use across many fields, including medical image analysis \citep{pennec2019medicalimage, medical_landmark, medical_diffusion_tensor}, protein modeling \citep{Watson2022.12.09.519842, deftelsen_proetin, Shapovalov_protein}, robotics \citep{beikmohammadi2021learningriemannianmanifoldsgeodesic, simeonov2021neuraldescriptorfieldsse3equivariant, Senanayake2018DirectionalGM, feiten_robotics}, generative modeling \citep{arvanitidis2021latentspaceodditycurvature, shao2017riemannian, debortoli2022riemannianscorebasedgenerativemodelling, jo2024generativemodelingmanifoldsmixture, mathieu2020riemanniancontinuousnormalizingflows, moser_flow} and information geometry \citep{arvanitidis2022pullinginformationgeometry, Nielsen_2020, cheng2025categoricalflowmatchingstatistical, myers2025learningassisthumansinferring, davis2024fisherflowmatchinggenerative}.

\emph{Geodesics}, or \emph{locally shortest paths}, are the key computational tools for engaging with data residing on a Riemannian manifold. The lengths of geodesic curves inform us about manifold distances. Likewise, the \emph{exponential} and \emph{logarithm maps}, which generalize vector addition and subtraction, are rooted in geodesics \citep[Chapter 9]{dggm}. Unfortunately, the solution to the boundary-value problem for geodesics are only available in closed form for the simplest manifolds, and numerical approximations are therefore required in practice.

Such approximations are typically achieved by minimizing the \emph{energy functional} over a suitable set of curves connecting the given boundary points and parametrized by $\gamma: [0,1] \rightarrow U$ \citep[Page 194]{do1992riemannian}
\begin{equation} \label{eq:energy}
    \mathcal{E}(\gamma) = \frac{1}{2}\int_{0}^{1}\dot{\gamma}(t)^{\top}G\left(\gamma(t)\right)\dot{\gamma}(t)\,\dif t,
\end{equation}
where $G$ denotes the metric matrix function of the manifold in local coordinates in a chart $\phi: U \rightarrow \mathcal{M}$, where $U$ is an open subset of $\mathbb{R}^{d}$. This energy is commonly minimized with general-purpose gradient-based optimization methods \citep{software:stochman, arvanitidis2021latentspaceodditycurvature, shao2017riemannian} or, in particular when considering the initial value problem and using shooting methods, by solving the system of ordinary differential equations that arise by applying the Euler-Lagrange equations \citep{leapfrog_noakes, tosi2014metricsprobabilisticgeometries}. Unfortunately, both approaches are often numerically brittle, exhibit slow convergence, or scale poorly with the number of grid points and manifold dimensions.

\begin{figure}[t!]
    \centering
    \includegraphics[width=1.0\textwidth]{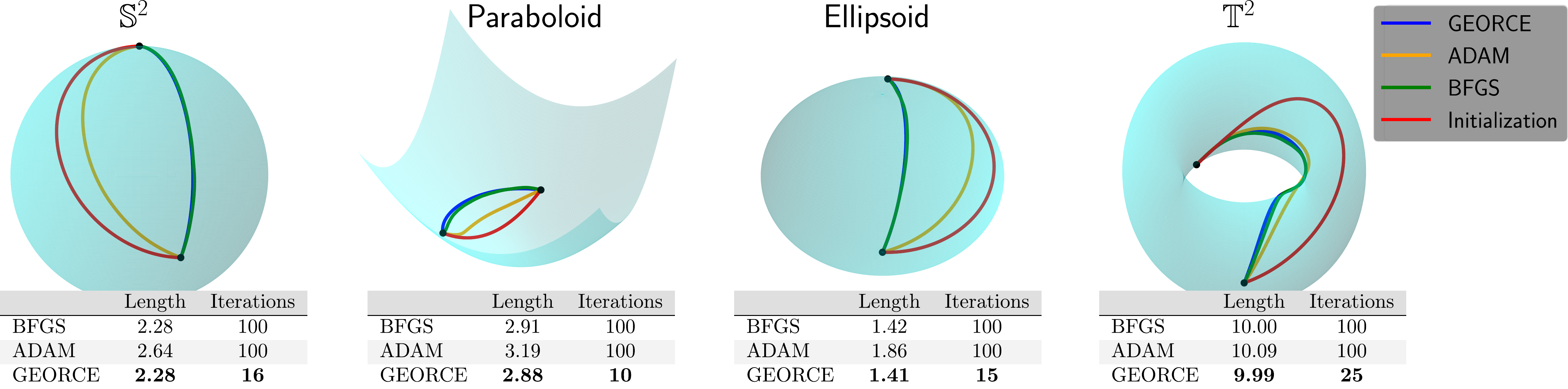}
    \caption{Comparison between \textit{GEORCE} and baseline methods for computing locally length minimizing geodesics between given point pairs on four different manifolds. All algorithms are terminated if the $\ell^{2}$-norm of the gradient of the discretized energy functional \eqref{eq:disc_const_energy} is less than $10^{-4}$, or the number of iterations exceeds $100$. Note that sometimes the respective solutions overlap, making it difficult to distinguish them. Also note, that in the case of the torus the initial curve determines the direction around the hole so that the obtained geodesics follow this direction and are therefore clearly not the globally shortest geodesics between the given start and end point. Experimental details and benchmark data are in Appendix~\ref{ap:experiments}.}
    \label{fig:synthetic_riemannian_geodesics}
    \vspace{-1.em}
\end{figure}

\textbf{In this paper}, we introduce the \textit{GEORCE} algorithm (\textit{GEodesic Optimization algoRithm using Control tEchniques}) that computes numerical solutions to the geodesic boundary value problem via an optimal control strategy. \textit{GEORCE} applies to general Riemannian as well as general Finslerian manifolds. We show that the algorithm has global convergence similar to gradient descent and that it has asymptotic quadratic local convergence similar to the Newton method in the limit. Empirically, we compare \textit{GEORCE} with alternative optimization algorithms on the discretized energy functional and see that \textit{GEORCE} is much faster and attains lower Riemannian length as illustrated in Fig.~\ref{fig:synthetic_riemannian_geodesics}.


\section{Background and related work}
\textbf{A Riemannian manifold}, $(\mathcal{M},g)$, can indirectly be considered as a smooth $n$-dimensional surface in a Euclidean space of sufficiently high dimension, see \citep{nash_embedding, whitney_embedding}. In this way $\mathcal{M}$ inherits its metric $g$ from the ambient space. All the intrinsic geometry of $(\mathcal{M},g)$, including curvature and geodesic distances, are encoded in the metric $g$. The metric can be -- and often is -- constructed independently of the formal embedding alluded to above, namely directly as a symmetric positive definite matrix valued function on the domain of a parametrization of the manifold in question. The metric, $g: T_{x}\mathcal{M} \times T_{x}\mathcal{M} \rightarrow \mathbb{R}$, defines an inner product that smoothly varies over the tangent spaces of $\mathcal{M}$ such that in a local chart $g(v,w)=v^{\top}G(x)w$, where $G$ is the metric matrix function evaluated in $x \in \mathcal{M}$ for $v,w \in T_{x}\mathcal{M}$. The tangent space, $T_{x}\mathcal{M}$, consists of tangents to all curves at $x \in \mathcal{M}$ \citep[Page 7 Definition 2.6]{do1992riemannian}.

\textbf{A geodesic} can then be defined using the inner product in each tangent space. The length of a curve, $\gamma: [0,1] \rightarrow U$ parameterized by $t \in [0,1]$ is given by 
\begin{equation} \label{eq:length}
    \mathcal{L}(\gamma) = \int_{0}^{1} \sqrt{\dot{\gamma}(t)^{\top}G(\gamma(t))\dot{\gamma}(t)} \,\dif t,
\end{equation}
where $\phi: U \rightarrow \mathcal{M}$ is a local chart for an open set $U \subset \mathbb{R}^{d}$. A Riemannian manifold can be equipped with a unique metrically compatible connection, $\nabla$, known as the Levi-Civita Connection \citep[Page 55 Theorem 3.6]{do1992riemannian}, where uniqueness and existence are guaranteed by completeness of the metric \citep[Page 63 Lemma 2.3]{do1992riemannian}. Note that curves are geodesics if and only if they are critical points of the energy functional, i.e., the first variation of the energy functional is zero for all variations with fixed endpoints assuming that the end points are not conjugate points \citep[Page 196 Proposition 2.5]{do1992riemannian}. This also holds in the Finslerian case \citep[Page 81]{ohta2021comparison}.

An equivalent way to define geodesics is to define these as curves, $\gamma$, with zero acceleration using the Levi-Civita connection as the covariant derivative, i.e., $\nabla_{\dot{\gamma}(t)}\dot{\gamma}(t) = 0$. In a local chart, this gives rise to the following (usually non-linear) \textsc{ode}-system \citep[Page 62]{do1992riemannian}
\begin{equation} \label{eq:bvp_ode}
    \frac{\dif^{2} \gamma^{k}}{\dif t^{2}}+\Gamma_{ij}^{k}\frac{\dif \gamma^{i}}{\dif t}\frac{\dif \gamma^{j}}{\dif t} = 0, \quad \gamma(0)=a,\gamma(1)=b,
\end{equation}
written in Einstein notation\footnote{In Einstein notation, upper and lower indices constitute a sum, i.e. $a^{i}b_{i}=\sum_{i}a_{i}b_{i}$}, where $\left\{\Gamma_{ij}^{k}\right\}_{i,j,k}$ denote the Christoffel symbols. The exponential map $\mathrm{Exp}_{x}(v): T_{x}\mathcal{M} \rightarrow \mathcal{M}$ is then defined as $\mathrm{Exp}_{x}(v)=\gamma(1)$, where $\gamma$ is a geodesic with $\gamma(0)=x \in \mathcal{M}$ and initial velocity vector, $v \in T_{x}\mathcal{M}$, i.e. $\dot{\gamma}(0)=v$. Within the injectivity radius of $x$ the exponential map, $\mathrm{Exp}_{x}$, is a diffeomorphism and its inverse is the logarithmic map, $\mathrm{Log}_{x}: \mathcal{M} \rightarrow T_{x}\mathcal{M}$.

Examples of Riemannian manifolds include information geometric spaces, i.e., statistical manifolds, where ``points'' in the manifold correspond to statistical distributions. These can be equipped with the so-called Fisher-Rao metric, see \citep{miyamoto2024closedformexpressionsfisherraodistance, Nielsen_2020}:
\begin{equation} \label{eq:fisher_rao_metric}
    G_{ij}(\theta) = \int_{\mathcal{X}}p(x | \theta)\left(\frac{\partial}{\partial \theta^{i}}\log p(x | \theta)\right)\left(\frac{\partial}{\partial \theta^{j}}\log p(x | \theta)\right) \,\dif\mu(x),
\end{equation}
where $\theta$ corresponds to the parameters of the probability density function $p$ under a measure $\mu$. Informally, the Fisher-Rao metric can be seen as a measure of the amount of information in the data around a parameter $\theta$ \citep{fisher_intuitiv}.

\textbf{The energy functional} in Eq.~\ref{eq:energy} can be discretized and minimized using standard optimization methods. Consider fixed start and end points $x_{0}=a \in U$ and $x_{T}=b \in U$, and let $x_{0:T}:=\{x_{t}\}_{t=0}^{T}$ denote a discretized version of the candidate curve $\gamma$ in $T+1$ points in a local chart $U$. Minimizing the energy functional can be re-written as a discrete constrained optimization problem, i.e., subdivision of $[0,1]$ into $T$ subintervals, in the chosen local chart \citep{arvanitidis2021latentspaceodditycurvature, shao2017riemannian}
\begin{equation} \label{eq:disc_const_energy}
    \begin{split}
        E(x_{0:T}) = \quad \min_{x_{0:T}} \quad &\sum_{t=0}^{T-1} (x_{t+1}-x_{t})^{\top}G(x_{t})(x_{t+1}-x_{t}) \\
        \text{s.t.} \quad &x_{0}=a,x_{T}=b,
    \end{split}
\end{equation}
where $G\left(\cdot\right)$ denotes the metric matrix function in local coordinates, and the tangent vectors of the coordinate expression for $\gamma$ are discretized as $\dot{\gamma}(t)\approx\left(x_{t+1}-x_{t}\right)T$. Note that the constant scaling by $\sfrac{1}{2}$ is omitted since it does not affect the optimal solution. With this formulation, the curve is approximated as a piece-wise linear function. This is not restrictive since $\gamma$ can be approximated into a curve for any given precision $T$ or converted to a smooth function, e.g., using splines \citep{software:stochman}. Note that by approximating the functional energy by forward differences, the computational error is linear in $\sfrac{1}{T}$, that is, in the form $\mathcal{O}\left(\sfrac{1}{T}\right)$. 

\begin{wrapfigure}[15]{r}{0.50\textwidth}
    \vspace{-10mm}
    \centering
    \includegraphics[width=0.50\textwidth]{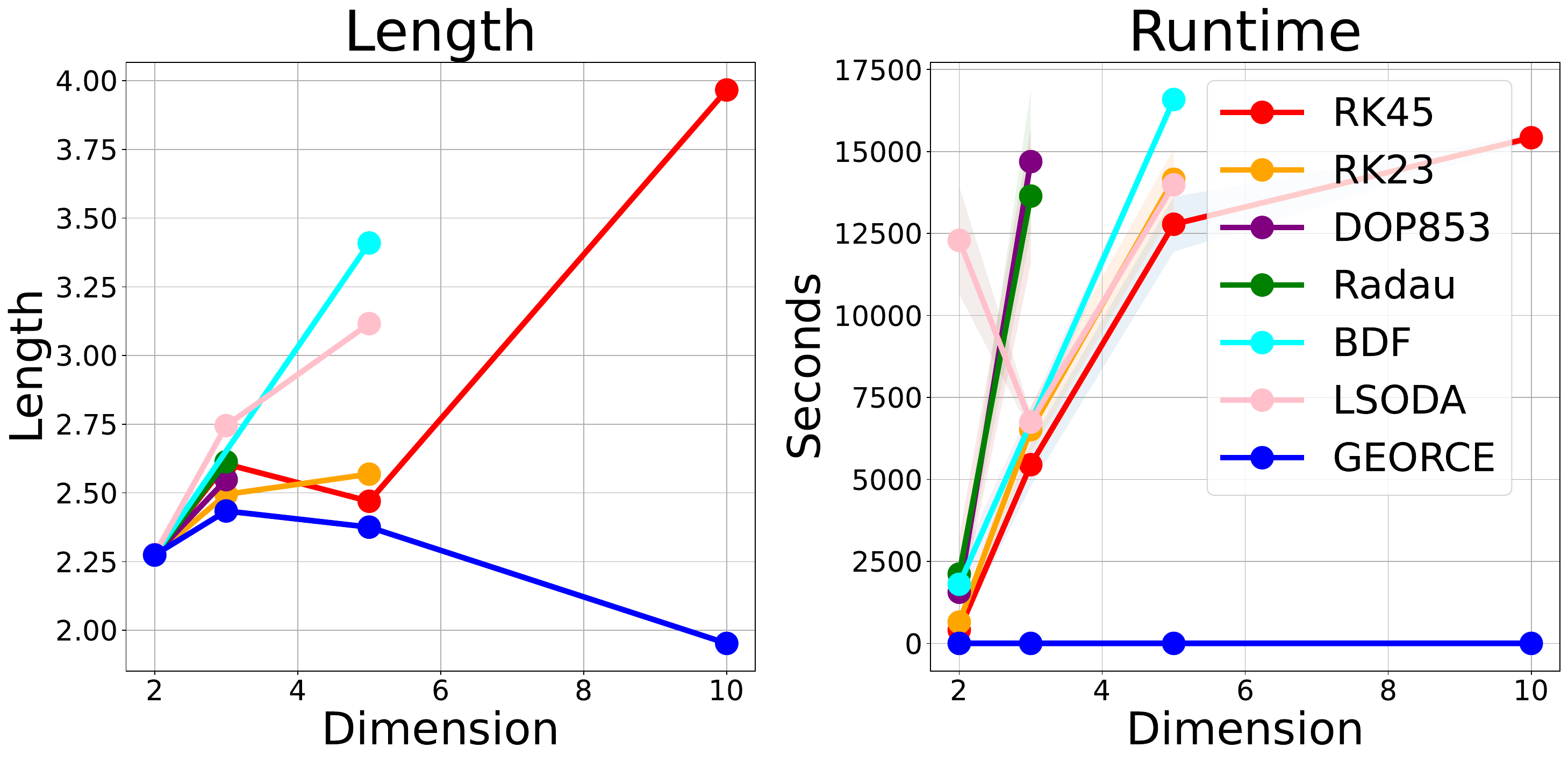}
    \vspace{-3mm}
    \caption{The estimates of \textsc{bvp}-solvers with different integration methods compared to \textit{GEORCE} on $\mathbb{S}^{n}$ for $n=2,3,5,10$. The \textsc{bvp}-solvers minimize the squared error to the boundary condition using \textit{BFGS} \citep{broyden_bfgs, fletcher_bfgs, Goldfarb1970AFO, shanno_bfgs}. All methods are terminated if they take more than 24 hours or use more than 10 GB of memory on a CPU. Details on the experiment is found in Appendix~\ref{ap:hyper_parameters}.}
    \label{fig:bvp_comparison}
    \vspace{-1.65em}
\end{wrapfigure}

Similarly to the energy functional in Eq.~\ref{eq:energy}, it is assumed that a stationary point to the discretized energy functional in Eq.~\ref{eq:disc_const_energy} is either a local or global minimum point. Thus, the number of points, T, in Eq.~\ref{eq:disc_const_energy} needs to be sufficiently high to achieve this. If this is not fulfilled, then there may exist saddle points for the discretized energy functional.

The optimization problem in Eq.~\ref{eq:disc_const_energy} can be solved using standard solvers in the variables $x_{1:(T-1)}$. Gradient descent converges to a local minimum (\emph{global convergence}), but will have local linear convergence, while a Newton method or quasi-Newton method will have local quadratic and super-linear convergence, respectively, but will involve solving a linear equation system and will, in general, not have global convergence \citep[Chapter 8]{luenberger2008linear}.

\textbf{Limitations of existing numerical solvers.} Solving the \textsc{ode}-system in Eq.~\ref{eq:bvp_ode} as a boundary value problem (\textsc{bvp}) can have difficulties converging depending on the choice of initial velocity vector \citep{leapfrog_noakes} and does not scale well in higher dimension. We illustrate this in Fig.~\ref{fig:bvp_comparison}, where we show the estimated length and runtime for $\mathbb{S}^{n}$ with $n=2,3,5,10,20$ for different integration methods compared to our proposed algorithm (\textit{GEORCE}). Different approaches try to circumvent this by solving the \textsc{bvp} on subintervals \citep{leapfrog_noakes, leapfrog_optimal_control} or solving the \textsc{bvp}-problem as a fixed-point problem in a probabilistic manner \citep{hennig2014probabilisticsolutionsdifferentialequations, arvanitidis2019fastrobustshortestpaths}. However, these methods often require access to higher-order information about the metric tensor, e.g in the form of Christoffel symbols.

Directly minimizing the energy in Eq.~\ref{eq:disc_const_energy} can, e.g., be done using off-the-shelf solvers like \textit{ADAM} \citep{kingma2017adam} or \textit{BFGS} \citep{broyden_bfgs, fletcher_bfgs, Goldfarb1970AFO, shanno_bfgs}. However, these methods tend to have slow convergence and scale poorly in the manifold dimension and number of grid points (see Sec.~\ref{sec:experiments}). Alternatively, the geodesic can be estimated by discretizing the smooth manifold by simplicial complexes \citep{crane2020survey}, for which there exist different geodesic solvers \citep{mmp, chen_han, chen_han_updated, vtp, hadi:rss:2021}. Although discretized manifolds appear naturally in computational geometry, the approach is not scalable to discretize high-dimensional smooth manifolds.

\section{GEORCE}
To develop an efficient primal-dual algorithm for computing geodesics, we will first reformulate the discretized energy functional \eqref{eq:disc_const_energy} as a discrete control problem in a local chart. We then derive the necessary conditions for a minimum for the control problem using techniques from optimal control to decompose the discrete control problem into strictly convex subproblems; from this, we deduce an iterative scheme to compute geodesics. Finally, we prove that our proposed algorithm exhibits global convergence and local asymptotic quadratic convergence.

We will assume the following for our algorithm
\begin{itemize}
    \item We assume access to a local chart of the manifold and assume that the entire candidate curve of the geodesic can be represented within the same chart.
    \item All computations are performed within the local coordinate chart $U$, which is defined based on the induced metric of the Riemannian manifold.
    \item For computational efficiency, we will leverage the Euclidean properties of the local chart (i.e., straight-line interpolation within $U$).
    \item We will assume that the end points are not conjugate points, and that the discretization is sufficiently fine such that a critical point of the discretized energy functional in Eq.~\ref{eq:disc_const_energy} is either a local or global minimum point.
\end{itemize}

\subsection{Deriviaton of GEORCE}
As a starting point, we reformulate the problem of finding a minimum of the discrete energy functional \eqref{eq:disc_const_energy} to a control problem in a local chart
\begin{equation} \label{eq:energy_control}
    \begin{split}
        \min_{(x_{1:T-1},u_{0:T-1})} \quad &\left\{\sum_{t=0}^{T-1}u_{t}^{\top}G(x_{t})u_{t}\right\} \\
        \text{s.t.} \quad &x_{t+1} = x_{t}+u_{t}, \quad t=0,\dots,T-1, \\
        &x_{0}=a,x_{T}=b.
    \end{split}
\end{equation}
Here, $x_{0:T}$ denotes the state variables, while $u_{0:(T-1)}$ corresponds to the control variables, i.e., the velocity vectors along the curve (modulo scaling). A key insight into our approach is that rephrasing into an optimal control problem leads to a set of convex subproblems. Using this, we derive the following proposition.
%
%
\begin{proposition} \label{prop:riemann_cond}
    The necessary conditions for a minimum in Eq.~\ref{eq:energy_control} is
    \begin{equation} \label{eq:energy_opt_condtions}
        \begin{split}
            &2G(x_{t})u_{t}+\mu_{t}=0, \quad t=0,\dots, T-1, \\
            &x_{t+1}=x_{t}+u_{t}, \quad t=0,\dots,T-1, \\
            &\restr{\nabla_{y}\left[u_{t}^{\top}G(y)u_{t}\right]}{y=x_{t}}+\mu_{t}=\mu_{t-1}, \quad t=1,\dots,T-1. \\
            &x_{0}=a, x_{T}=b,
        \end{split}
    \end{equation}
    where $\mu_{t} \in \mathbb{R}^{d}$ for $t=0,\dots,T-1$.
\end{proposition}
\begin{proof}
    The Hamiltonian of Eq.~\ref{eq:energy_control} is
    \begin{equation*}
        H_{t}(x_{t},u_{t},\mu_{t})=u_{t}^{\top}G(x_{t})u_{t}+\mu_{t}^{\top}(x_{t}+u_{t}).
    \end{equation*}
    The Hamiltonian $H_{t}(x_{t},u_{t},\mu_{t})$ is strictly convex in $u_{t}$, since $G(x_{t})$ is positive definite. The time-discrete version of Pontryagins maximum principle \citep{kirk2004optimal} then yields the following when applied to our control problem \eqref{eq:energy_control}
    \begin{equation}
        \begin{split}
            &2G(x_{t})u_{t}+\mu_{t}=0, \quad t=0,\dots, T-1, \\
            &x_{t+1}=x_{t}+u_{t}, \quad t=0,\dots,T-1, \\
            &\restr{\nabla_{y}\left[u_{t}^{\top}G(y)u_{t}\right]}{y=x_{t}}+\mu_{t}=\mu_{t-1}, \quad t=1,\dots,T-1. \\
            &x_{0}=a, x_{T}=b.
        \end{split}
    \end{equation}
\end{proof}
Proposition~\ref{prop:riemann_cond} decomposes the minimization problem in Eq.~\ref{eq:energy_control} into a system of equations for each time step $t$. Note that these equations are necessary conditions for the optimum for Eq.~\ref{eq:energy_control} and are sufficient in the case where the objective function in Eq.~\ref{eq:energy_control} is convex in both $x_{t}$ and $u_{t}$ \citep[Page 244]{boyd2004convex}. However, this is generally not the case, but due to the assumption applying a sufficiently fine discretization the solution to Eq.~\ref{eq:energy_opt_condtions} will be a local or global minimum point.

The problem with the necessary conditions in Proposition~\ref{prop:riemann_cond} is that the equations cannot be solved in closed form for a general metric matrix function, $G$. To circumvent this, we will derive an iterative scheme. Consider the variables $x_{0:T}^{(i)}$ and $u_{0:T}^{(i)}$ in iteration $i$. We apply the ``trick'' of fixing the metric matrix function, and the derivative of the inner product in iteration $i$.
\begin{equation} \label{eq:fixed_variables}
    \begin{split}
        &\nu_{t} := \restr{\nabla_{y}\left(u_{t}^{\top}G(y)u_{t}\right)}{y=x_{t}^{(i)},u_{t}=u_{t}^{(i)}}, \quad t=1,\dots,T-1, \\
        &G_{t} := G\left(x_{t}^{(i)}\right), \quad t=0,\dots,T-1,
    \end{split}
\end{equation}
Fixing these in iteration $i$, the system of equations in Eq.~\ref{eq:energy_opt_condtions} reduces to
\begin{equation} \label{eq:energy_zero_point_problem}
    \begin{split}
        &2G_{t}u_{t}+\mu_{t} = 0, \quad t=0,\dots,T-1, \\
        &\nu_{t}+\mu_{t} = \mu_{t-1}, \quad t=1,\dots,T-1, \\
        &\sum_{t=0}^{T-1}u_{t}=b-a, \\
    \end{split}
\end{equation}
where $\nu_{1:(T-1)}$ and $G_{0:(T-1)}$ are fixed \eqref{eq:fixed_variables}. Since $G_{0:(T-1)}$ are positive definite, the modified system of equations in Eq.~\ref{eq:energy_zero_point_problem} can be explicitly solved in iteration $i$ for $u_{0:(T-1)}$ and $\mu_{0:(T-1)}$ as shown in the following proposition.
\begin{proposition} \label{prop:update_scheme}
    The update scheme for $u_{t},\mu_{t}$ and $x_{t}$ is
    \begin{equation} \label{eq:energy_update_schem}
        \begin{split}
            &\mu_{T-1} = \left(\sum_{t=0}^{T-1}G_{t}^{-1}\right)^{-1}\left(2(a-b)-\sum_{t=0}^{T-1}G_{t}^{-1}\sum_{t>j}^{T-1}\nu_{j}\right), \\
            &u_{t} = -\frac{1}{2}G_{t}^{-1}\left(\mu_{T-1}+\sum_{j>t}^{T-1}\nu_{j}\right), \quad t=0,\dots,T-1, \\
            &x_{t+1} = x_{t}+u_{t}, \quad t=0,\dots,T-1, \\
            &x_{0}=a.
        \end{split}
    \end{equation}
\end{proposition}
\begin{proof}
    From the update formula $\nu_{t}+\mu_{t}=\mu_{t-1}$ for $t=1,\dots,T-1$ we deduce that $\mu_{t} = \mu_{T-1}+\sum_{j>t}^{T-1}\nu_{j}$. From the update scheme $2G_{t}u_{t}+\mu_{t}=0$ we deduce that $u_{t}=-\frac{1}{2}G_{t}^{-1}\mu_{t}$, since $G_{t}$ is positive definite and hence has a well-defined inverse. Noting that $\sum_{t=0}^{T-1}u_{t}=b-a$ we can re-write the iterative formulas as
    \begin{equation*}
        \begin{split}
            &\mu_{t} = \mu_{T-1}+\sum_{j>t}^{T-1}\nu_{j}, \quad t=0,\dots,T-1, \\
            &u_{t} = -\frac{1}{2}G_{t}^{-1}\left(\mu_{T-1}+\sum_{j>t}^{T-1}\nu_{j}\right), \quad t=0,\dots,T-1, \\
            &-\frac{1}{2}\sum_{t=0}^{T-1}G_{t}^{-1}\left(\mu_{T-1}+\sum_{j>t+1}^{T-1}\nu_{j}\right)=b-a.
        \end{split}
    \end{equation*}
    The last equation can be used to solve for $\mu_{T-1}$ utilizing that the sum of (inverse) positive definite matrices is positive definite and has a well-defined inverse
    \begin{equation*}
        \mu_{T-1} = \left(\sum_{t=0}^{T-1}G_{t}^{-1}\right)^{-1}\left(2(a-b)-\sum_{t=0}^{T-1}G_{t}^{-1}\sum_{t>j}^{T-1}\nu_{j}\right),
    \end{equation*}
    which proves the result.
\end{proof}
\begin{algorithm}[!ht]
    \caption{GEORCE for Riemannian manifolds}
    \label{al:georce}
    \begin{algorithmic}[1]
        \State \textbf{Input}: $\mathrm{tol}$, $T$
        \State \textbf{Output}: Geodesic estimate $x_{0:T}$
        \State Set $x_{t}^{(0)} \leftarrow a+\frac{b-a}{T}t$, $u_{t}^{(0)} \leftarrow \frac{b-a}{T}$ for $t=0.,\dots,T-1$ and $i \leftarrow 0$
        \While{$\norm{\restr{\nabla_{y}E(y)}{y=x_{t}^{(i)}}}_{2} > \mathrm{tol}$}
        \State $G_{t} \leftarrow G\left(x_{t}^{(i)}\right)$ for $t=0,\dots,T-1$
        \State $\nu_{t} \leftarrow \restr{\nabla_{y}\left(u_{t}^{(i)}G\left(y\right)u_{t}^{(i)}\right)}{y=x_{t}^{(i)}}$ for $t=,1\dots,T-1$
        \State $\mu_{T-1} \leftarrow \left(\sum_{t=0}^{T-1}G_{t}^{-1}\right)^{-1}\left(2(a-b)-\sum_{t=0}^{T-1}G_{t}^{-1}\sum_{t>j}^{T-1}\nu_{j}\right)$
        \State $\tilde{u}_{t} \leftarrow -\frac{1}{2}G_{t}^{-1}\left(\mu_{T-1}+\sum_{j>t}^{T-1}\nu_{j}\right)$ for $t=0,\dots,T-1$
        \State Using line search find $\alpha^{*}$ for the following optimization problem with the discrete energy functional $E$
        \begin{equation*}
            \begin{split}
                \alpha^{*} = \argmin_{\alpha}\quad &E\left(\tilde{x}_{0:T}\right) \quad \text{(exact line search)} \\
                \text{s.t.} \quad &\tilde{x}_{t+1}=\tilde{x}_{t}+\alpha \tilde{u}_{t}+(1-\alpha)u_{t}^{(i)}, \quad t=0,\dots,T-1, \\
                &\tilde{x}_{0}=a.
            \end{split}
        \end{equation*}
        \State Set $u_{t}^{(i+1)} \leftarrow \alpha^{*}\tilde{u}_{t}+(1-\alpha^{*})u_{t}^{(i)}$ for $t=0,\dots,T-1$
        \State Set $x_{t+1}^{(i+1)} \leftarrow x_{t}^{(i+1)}+u_{t}^{(i+1)}$ for $t=0,\dots,T-1$
        \State $i \leftarrow i+1$
        \EndWhile
        \State return $x_{t}$ for $t=0,\dots,T-1$
    \end{algorithmic}
\end{algorithm}
Proposition~\ref{prop:update_scheme} immediately leads to an iterative update of $\left(x_{t}^{(i)}, u_{t}^{(i)}\right)$ to estimate the geodesic, which we denote the \textit{GEORCE}-algorithm. \textit{GEORCE} is described in pseudo-code in Algorithm~\ref{al:georce}. Since the computations in \textit{GEORCE} take place in a local chart, it is implicitly assumed that the computations are well-defined within the chart. Let $\phi: U \rightarrow \mathcal{M}$ denote the chart, where $U \subseteq \mathbb{R}^{d}$ is an open set. It is assumed that the initialization of the curve in Algorithm~\ref{al:georce} is defined within the chart, i.e., $x_{t}^{(0)} \in U$ for all $t$. If the updated solution $\tilde{x} \notin U$ in iteration $i+1$, then soft line-search (see below) is applied to determine a point in $U$ on the line between $\left(x^{(i)},u^{(i)}\right)$ and $\left(\tilde{x},\tilde{u}\right)$, since $U$ is an open set. Note also that if the computations of \textit{GEORCE} are close to the boundary of $U$, one can change the chart assuming overlapping charts in the atlas of the manifold. Thus, in practice, it is not a problem for \textit{GEORCE} if $U$ is only a subset of $\mathbb{R}^{d}$. For all computations in this paper, we will not change chart or assume that we can leave the domain $U$ during the computations, but such cases can easily be modified into the algorithm as described above.

As we show below, \textit{GEORCE} exhibits global convergence, which means that the algorithm will converge to a local minimum. If there are multiple local minima, then \textit{GEORCE} will not necessarily find the global minimum. To investigate whether \textit{GEORCE} has found the global minimum, one can apply \textit{GEORCE} for different initial curves to compute locally-length minimizing curves and choose the curve with the shortest length as the global minimum. In Appendix~\ref{ap:init_curve}, we illustrate the effect of different choices of initialization curves and how it affects the solution using \textit{GEORCE}.

\paragraph{Line search and stopping criteria}
We find that an appropriate step size is necessary to ensure global convergence for \textit{GEORCE} (See Proposition~\ref{prop:global_convergence}), and therefore we update $u_{t}^{(i+1)} \leftarrow \alpha u_{t}^{(i+1)}+(1-\alpha)u_{t}^{(i)}$ in iteration $i$. Here, $\alpha$ denotes the step size, which we estimate with line search. One approach is to estimate the optimal step size in each iteration, but, in our experience, this is unnecessarily expensive. We instead use  \textit{soft} line search, which merely ensures that the objective function decreases in each iteration \citep{luenberger2008linear}. We implement this using backtracking, where an initial step size is multiplied by a decay rate, $0<\rho<1$, until some condition is met, e.g., the Wolfe condition \citep{wolfe_condition}, the Armijo condition \citep{armijo1966minimization}, or the curvature condition\footnote{Note that the term ``curvature condition'' is rather misleading in a Riemannian context as it refers to conditions on the derivative of the objective function and not the Riemannian curvature.} \citep{wolfe_condition}. In practice, we find that the Armijo condition \citep{armijo1966minimization} with $\rho=0.5$ is sufficient for most manifolds, although the decay rate can be fine-tuned to increase the performance of \textit{GEORCE}.

We threshold the $\ell^{2}$-norm of the gradient of the discrete energy functional as a stopping criteria for \textit{GEORCE}. In all empirical comparisons, identical stopping criteria are applied to all optimization algorithms.

\begin{table}[h!]
    \resizebox{\textwidth}{!}{
    \begin{tabular}{c|c c c c c}
        \hline
        \textbf{Algorithm} & \textbf{Global convergence} & \textbf{Local convergence} & \textbf{Gradient} & \textbf{Hessian} & \textbf{Complexity} \\
        \hline
        \textbf{Gradient descent} & \checkmark & Linear & \checkmark & \crossmark & $\mathcal{O}\left(Td^{2}\right)$ \\
        \textbf{Quasi-Newton} & \crossmark & Super linear & \checkmark & \crossmark & $\mathcal{O}\left(T^{2}d^{2}\right)$ \\
        \textbf{Newton method} & \crossmark & Quadratic & \checkmark & \checkmark & $\mathcal{O}\left(Td^{3}\right)$ \\
        \textbf{GEORCE} & \checkmark & Asymptotic quadratic &  \checkmark & \crossmark & $\mathcal{O}\left(Td^{3}\right)$ \\
        \hline
    \end{tabular}
    }
    \caption{Convergence, complexity, and use of higher order derivatives for different optimization algorithms in each iteration, where $T$ is the number of grid points and $d$ is the manifold dimension. A checkmark indicates that the algorithm exhibits the property, while a cross indicates that it is not the case.}
    \label{tab:algorithm_comparison}
    \vspace{-3.5em}
\end{table}
\paragraph{Comparison with other optimization methods}
\textit{GEORCE} has complexity $\mathcal{O}\left(Td^{3}\right)$ due to matrix inversions along the discretized curve but scales only linearly in $T$, unlike e.g.\@ quasi-Newton methods (Table~\ref{tab:algorithm_comparison}).
Although the computational complexity in the manifold dimension is of a lower order for quasi-Newton methods and gradient descent, \textit{GEORCE} exhibits faster convergence. Since the Hessian of the discretized energy functional in Eq.~\ref{eq:disc_const_energy} is sparse, the update step in the Newton method can be simplified in complexity compared to the standard Newton method that has complexity $\mathcal{O}\left(T^{3}d^{3}\right)$ (see Appendix~\ref{ap:newton_method}). However, the sparse Newton step requires second-order derivatives, unlike \textit{GEORCE}, which only requires first-order derivatives of the discretized energy. Note that in Table~\ref{tab:algorithm_comparison}, we assume that the iterations for the different methods remain well-defined in the chosen chart.

\subsection{Convergence results}
In this section, we show that \textit{GEORCE} exhibits global convergence, i.e., converges to a local minimum, and that \textit{GEORCE} exhibits local asymptotic quadratic convergence under certain regularity of the discretized energy functional. To simplify the notation, we let $\langle \cdot, \cdot \rangle$, $\norm{\cdot}$ and $\nabla$ denote the Euclidean inner product, norm and gradient in $\mathbb{R}^{d}$, respectively.

\paragraph{Global convergence}
We will prove that \textit{GEORCE} converges to a (local) minimum. To prove this, we will use the fact that the discretized energy functional is smooth and hence locally Lipschitz. This implies that the first order Taylor approximation of the discretized energy functional can be written as
\begin{equation} \label{eq:locally_lipschitz_taylor}
    \Delta E = \langle \nabla E(z_{0}), \Delta z \rangle + \mathcal{O}\left(\Delta z\right)\norm{\Delta z},
\end{equation}
Thus, the term $\mathcal{O}\left(\Delta z\right)\norm{\Delta z}$ can be re-written as $\mathcal{O}\left(\norm{\Delta z}^{2}\right)$ locally.

To prove global convergence for the algorithm \textit{GEORCE} we first show existence and uniqueness by proving two lemmas. Secondly, we prove that in each iteration the discretized energy functional is decreasing, unless a (local) minimum has been found.

For notation set $x^{(i)}=(a,x_{1}^{(i)}, \dots, x_{T-1}^{(i)}, b)$ and $u^{(i)} = (u_{0}^{(i)}, u_{1}^{(i)}, \dots, u_{T-2}^{(i)}, u_{T-1}^{(i)})$ as the solution for iteration $i$ such that $(x^{(i+1)}, u^{(i+1)})$ is the solution using the update scheme in proposition~\ref{prop:update_scheme}.

\begin{lemma}\label{lemma:global_conv_feasible}
    Assume that $\left(x^{(i)}, u^{(i)}\right)$ is a feasible solution, then 
    \begin{enumerate}
        \item There exists a unique solution $\left(x^{(i+1)}, u^{(i+1)}\right)$ to the system of equations in Eq.~\ref{eq:energy_update_schem} based on $\left(x^{(i)}, u^{(i)}\right)$.
        \item All linear combinations $(1-\alpha)\left(x^{(i)},u^{(i)}\right)+\alpha\left(x^{(i+1)},u^{(i+1)}\right)$ for $0 < \alpha \leq 1$ are feasible solutions.
    \end{enumerate}
\end{lemma}
\begin{proof}
    Since the matrices, $\left\{G\left(x_{t}^{(i)}\right)\right\}_{t=0}^{T}$ are positive definite, then the matrices are regular with unique inverse, and so are the sum of the inverse matrices. This means that there exists an unique solution to the update scheme in proposition~\ref{prop:update_scheme} proving 1.

    Initially, assume that $\left(x^{(i+1)},u^{(i+1)}\right)$ is well-defined in the local chart.

    Since the solution $\left(x^{(i+1)}, u^{(i+1)}\right)$ is also a feasible solution for the update scheme in proposition~\ref{prop:update_scheme}, then the linear combination of two feasible solutions will also be a feasible solution, i.e.
    \begin{equation*}
        \sum_{t=0}^{T-1}u_{t}^{(j)}=(b-a).
    \end{equation*}
    The new solution based on the linear combination gives
    \begin{equation*}
        \begin{split}
            &(1-\alpha)\sum_{t=0}^{T-1}u_{t}^{(i)}+\alpha\sum_{t=0}^{T-1}u_{t}^{(i+1)}=(1-\alpha)(b-a)+\alpha(b-a)=(b-a) \\
            &(1-\alpha)x^{(i)}+\alpha x^{(i+1)}  = (1-\alpha)\left(\alpha, x_{1}^{(i)}, \dots, x_{T-1}^{(i)},b\right) +\alpha \left(a, x_{1}^{(i+1)}, \dots, x_{T-1}^{(i+1)},b\right) \\
            &= \left(a, (1-\alpha)x_{1}^{(i)}+\alpha x_{1}^{(i+1)},\dots,(1-\alpha)x_{T-1}^{(i)}+\alpha x_{T-1}^{(i+1)},b\right),
        \end{split}
    \end{equation*}
    where for any $t=0,\dots,T-1$
    \begin{equation*}
        \begin{split}
            (1-\alpha)x_{t}^{(i)}+\alpha x_{t}^{(i+1)} &= (1-\alpha)\left(a+\sum_{j=0}^{t-1}u_{j}^{(i)}\right)+\alpha\left(a+\sum_{j=0}^{t-1}u_{j}^{(i+1)}\right) \\
            &= a+\sum_{j=0}^{t-1}\left((1-\alpha)u_{j}^{(i)}+\alpha u_{j}^{(i+1)}\right).
        \end{split}
    \end{equation*}
    This shows that the linear combination in the state is feasible as it produces feasible state variables in terms of start and end point. Furthermore, each state variable is feasible, if they were determined from the linear combination of the control vectors, which proves that the linear combinations of the state and control variables are also feasible solutions.

    If $\left(x^{(i+1)},u^{(i+1)}\right)$ is not well-defined in the local chart, then the line-search in \textit{GEORCE} will be able to determine a point belonging to the manifold and that will be on the line between $\left(x^{(i+1)},u^{(i+1)}\right)$ and $\left(x^{(i)},u^{(i)}\right)$ as the local chart is defined on an open set and $x^{(i)}$ is well-defined in the local chart. The new solution is then a feasible solution following the argument above.
\end{proof}
\begin{lemma}\label{lemma:global_taylor}
    Let $\left\{x_{t}^{(i)}, u_{t}^{(i)}\right\}_{t=0}^{T}$ denote the feasible solution after iteration $i$ in \textit{GEORCE}. If $\left\{x_{t}^{(i)}, u_{t}^{(i)}\right\}_{t=0}^{T}$ is not a (local) minimum point, then the feasible solution from iteration $\left\{x_{t}^{(i+1)}, u_{t}^{(i+1)}\right\}_{t=0}^{T}$ will decrease the objective function in the sense that there exists an $\eta>0$ such that for all $\alpha$ with $0<\alpha\leq\eta\leq 1$, then
    \begin{equation*}
        E\left(x^{(i)}+\alpha \left(x^{(i+1)}-x^{(i)}\right), u^{(i)}+\alpha \left(u^{(i+1)}-u^{(i)}\right)\right) < E\left(x^{(i)}, u^{(i)}\right),
    \end{equation*}
    where $x^{(i)}=\left(a,x_{1}^{(i)}, \dots, x_{T-1}^{(i)}, b\right)$ and $u^{(i)} = \left(u_{0}^{(i)}, u_{1}^{(i)}, \dots, u_{T-2}^{(i)}, u_{T-1}^{(i)}\right)$.
\end{lemma}
\begin{proof}
    $E(x,u)$ is a smooth function. The first order Taylor approximation of the energy function $E(x,u)$ in Eq.~\ref{eq:disc_const_energy} in $(x,u)$ is
    \begin{equation*}
        \begin{split}
            \Delta E(x,u) &= \sum_{t=1}^{T-1}\left(\left\langle\restr{\nabla_{x_{t}}E(x_{t},u_{t})}{(x_{t},u_{t})=\left(x_{t}^{(i)}, u_{t}^{(i)}\right)}, \Delta x_{t}\right\rangle +\mathcal{O}(\Delta x_{t})\norm{\Delta x_{t}}\right) \\
            &+\sum_{t=0}^{T-1}\left(\left\langle \restr{\nabla_{u_{t}}E(x_{t},u_{t})}{(x_{t},u_{t})=\left(x_{t}^{(i)}, u_{t}^{(i)}\right)}, \Delta u_{t}\right\rangle+\mathcal{O}(\Delta u_{t})\norm{\Delta u_{t}}\right),
        \end{split}
    \end{equation*}
    where $\Delta u_{t} := u_{t}^{(i+1)}-u_{t}^{(i)}$ and $\Delta x_{t} := x_{t}^{(i+1)}-x_{t}^{(i)}$. By the optimally conditions in Eq.~\ref{eq:energy_opt_condtions} we have that
    \begin{equation} \label{eq:f_cond}
        \begin{split}
            \restr{\nabla_{x_{t}}E(x,u)}{(x,u)=\left(x_{t}^{(i)},u_{t}^{(i)}\right)} &= \mu_{t-1}-\mu_{t}, \quad t=1,\dots,T-1, \\
            \restr{\nabla_{u_{t}}E(x,u)}{(x_{t},u_{t})=\left(x_{t}^{(i)}, u_{t}^{(i+1)}\right)} &= -\mu_{t}, \quad t=0,\dots,T-1.
        \end{split}
    \end{equation}
    Inserting the optimality conditions in the Taylor expansion we have that
    \begin{equation*}
        \begin{split}
            \Delta E(x,u) &= \sum_{t=1}^{T-1}\left(\langle \mu_{t-1}-\mu_{t}, \Delta x_{t}\rangle +\mathcal{O}(\Delta x_{t})\norm{\Delta x_{t}}\right) \\
            &+\sum_{t=0}^{T-1}\left(\left\langle \restr{\nabla_{u_{t}}E(x,u)}{(x,u)=\left(x_{t}^{(i)},u_{t}^{(i)}\right)}, \Delta u_{t}\right\rangle+\mathcal{O}(\Delta u_{t})\norm{\Delta u_{t}}\right).
        \end{split}
    \end{equation*}
    Since $\Delta x_{t} = \sum_{j=0}^{t-1}\Delta u_{j}$, then
    \begin{equation*}
        \begin{split}
            \Delta E(x,u) &= \sum_{t=1}^{T-1}\left(\langle \mu_{t-1}-\mu_{t}, \sum_{j=0}^{t-1}\Delta u_{j}\rangle +\mathcal{O}\left(\sum_{j=0}^{t-1}\Delta u_{j}\right)\norm{\sum_{j=0}^{t-1}\Delta u_{j}}\right) \\
            &+\sum_{t=0}^{T-1}\left(\left\langle \restr{\nabla_{u_{t}}E(x,u)}{(x,u)=\left(x_{t}^{(i)},u_{t}^{(i)}\right)}, \Delta u_{t}\right\rangle+\mathcal{O}(\Delta u_{t})\norm{\Delta u_{t}}\right).
        \end{split}
    \end{equation*}
    Re-arranging the terms we get
    \begin{equation*}
        \begin{split}
            \Delta E(x,u) &= \sum_{t=0}^{T-2}\left\langle \mu_{t}, \sum_{j=0}^{t}\Delta u_{j}\right\rangle \\
            &-\sum_{t=1}^{T-1}\left(\left\langle \mu_{t}, \sum_{j=0}^{t-1}\Delta u_{j} \right\rangle+\mathcal{O}\left(\sum_{j=0}^{t-1}\Delta u_{j}\right)\norm{\sum_{j=0}^{t-1}\Delta u_{j}}\right) \\
            &+\sum_{t=0}^{T-1}\left(\left\langle \restr{\nabla_{u_{t}}E(x_{t},u_{t})}{(x_{t},u_{t})=\left(x_{t}^{(i)}, u_{t}^{(i)}\right)}, \Delta u_{t}\right\rangle+\mathcal{O}\left(\Delta u_{t}\right)\norm{u_{t}}\right) \\
            &= \langle \mu_{0}, \Delta u_{0} \rangle + \sum_{t=1}^{T-2}\langle \mu_{t}, \Delta u_{t}\rangle - \left\langle \mu_{T-1}, \sum_{j=0}^{T-1}\Delta u_{j}\right\rangle+\langle \mu_{T-1}, \Delta u_{T-1}\rangle \\
            &+ \sum_{t=0}^{T-1}\left(\left\langle \restr{\nabla_{u_{t}}E(x_{t},u_{t})}{(x_{t},u_{t})=\left(x_{t}^{(i)}, u_{t}^{(i)}\right)}, \Delta u_{t}\right\rangle+\mathcal{O}\left(\Delta u_{t}\right)\norm{\Delta u_{t}}\right)
        \end{split}
    \end{equation*}
    Since $\Delta u_{j} = u_{j}^{(i+1)}-u_{j}^{(i)}$ for $j=0,\dots,T-1$ and $\Delta x_{j}=x_{j}^{(i+1)}-x_{j}^{(i)}$ for $j=0,\dots,T$, and since the solution is feasible we have that $\sum_{j=0}^{T-1}u_{j}^{(i)}=b-a$ for any $i$, which implies that $\sum_{j=0}^{T-1}\Delta u_{j}=0$. This reduces the above expression to
    \begin{equation*}
        \begin{split}
            \Delta E(x,u) &= \sum_{t=0}^{T-1}\left(\langle \mu_{t}, \Delta u_{t}\rangle+\mathcal{O}\left(\sum_{j=0}^{t-1}\Delta u_{j}\right)\norm{\sum_{j=0}^{t-1}\Delta u_{j}}\right) \\
            &+ \sum_{t=0}^{T-1}\left(\left\langle \restr{\nabla_{u_{t}}E(x_{t},u_{t})}{(x_{t},u_{t})=\left(x_{t}^{(i)}, u_{t}^{(i)}\right)}, \Delta u_{t}\right\rangle +\mathcal{O}\left(\Delta u_{t}\right)\norm{\Delta u_{t}}\right) \\
            &= \sum_{t=0}^{T-1}\left\langle \mu_{t}+\restr{\nabla_{u_{t}}E(x,u)}{(x,u)=\left(x_{t}^{(i)},u_{t}^{(i)}\right)}, \Delta u_{t}\right\rangle \\
            &+\sum_{t=0}^{T-1}\mathcal{O}\left(\sum_{j=0}^{t-1}\Delta u_{j}\right)\norm{\sum_{j=0}^{t-1}\Delta u_{j}}+\sum_{t=0}^{T-1}\mathcal{O}\left(\Delta u_{t}\right)\norm{u_{t}}.
        \end{split}
    \end{equation*}
    By Eq.~\ref{eq:f_cond} we have
    \begin{equation}
        \begin{split}
            \mu_{t}+\restr{\nabla_{u_{t}}E(x_{t},u_{t})}{(x_{t},u_{t})=\left(x_{t}^{(i)}, u_{t}^{(i)}\right)} &= -\restr{\nabla_{u_{t}}E(x_{t},u_{t})}{(x_{t},u_{t})=\left(x_{t}^{(i)}, u_{t}^{(i+1)}\right)} \\
            &+\restr{\nabla_{u_{t}}E(x,u)}{(x,u)=\left(x_{t}^{(i)},u_{t}^{(i)}\right)} \\
            &= -2 G\left(x_{t}^{(i)}\right)\left(u_{t}^{(i+1)}-u_{t}^{(i)}\right),
        \end{split}
        \label{eq:finsler_changed}
    \end{equation}
    which implies that
    \begin{equation} \label{eq:energy_taylor}
        \begin{split}
            \Delta E(x,u) &= \sum_{t=0}^{T-1}\left\langle -2G\left(x_{t}^{(i)}\right)\Delta u_{t}, \Delta u_{t}\right\rangle \\
            &+\sum_{t=0}^{T-1}\left(\mathcal{O}\left(\sum_{j=0}^{t-1}\Delta u_{j}\right)\norm{\sum_{j=0}^{t-1}\Delta u_{j}}+\mathcal{O}\left(\Delta u_{t}\right)\norm{\Delta u_{t}}\right).
        \end{split}
    \end{equation}
    Since $G\left(x_{t}^{(i)}\right)$ is positive definite, the first summation is negative assuming that at least one non-zero vector $\{\Delta u_{t}\}_{t=0}^{T-1}$, which is the case, since we have assumed that $\left(x_{t}^{(i)},u_{t}^{(i)}\right)$ is not a (local) minimum. Now scale all vectors, $\left\{\Delta u_{t}\right\}_{t=0}^{T-1}$ with a scalar $0<\alpha<1$. Note that the term, $-2G\left(x_{t}^{(i)}\right)\Delta u_{t}$, is unaffected by this, since it equals $\mu_{t}+\nabla_{u_{t}}E(x,u)\left(x_{t}^{(i)}, u_{t}^{(i)}\right)$, which is independent of $\alpha$. The first order Taylor order approximation is
    \begin{equation*}
        \begin{split}
            \Delta E(x,u) &= \sum_{t=0}^{T-1}\left\langle -2 G\left(x_{t}^{(i)}\right)\Delta u_{t}, \alpha \Delta u_{t}\right\rangle \\
            &+\sum_{t=0}^{T-1}\left(\mathcal{O}\left(\sum_{j=0}^{t-1}\alpha \Delta u_{j}\right)\norm{\sum_{j=0}^{t-1}\alpha \Delta u_{j}}+\mathcal{O}\left(\alpha \Delta u_{t}\right)\norm{\alpha \Delta u_{t}}\right).
        \end{split}
    \end{equation*}
    In the limit it now follows that there exists an $\eta>0$ such that $\frac{\Delta E(x,u)}{\alpha}<0$ for some $\alpha$ with $0 < \alpha \leq \eta \leq 1$, which implies that $\Delta E(x,u) < 0$.
\end{proof}
\begin{proposition} \label{prop:global_convergence}
    Let $E^{(i)}$ be the value of the discretized energy functional for the solution after iteration $i$ (with line search) in \textit{GEORCE}. If the starting point $\left(x^{(0)}, u^{(0)}\right)$ is feasible, then the series $\left\{E^{(i)}\right\}_{i>0}$ will converge to a (local) minimum. 
\end{proposition}
\begin{proof}
    If the starting point $\left(x^{(0)}, u^{(0)}\right)$ in \textit{GEORCE} is feasible, then the points $\left(x^{(i)}, u^{(i)}\right)$ after each iteration (with line-search) in \textit{GEORCE} will also be feasible solutions by lemma~\ref{lemma:global_conv_feasible}.

    The discretized energy functional is a positive function, i.e. a lower bound is at least 0. Assume that the series $\left\{E^{(i)}\right\}_{i>0}$ is not converging to a (local) minimum. From lemma~\ref{lemma:global_taylor} it follows that $\left\{E^{(i)}\right\}_{i>0}$ is decreasing for increasing $i$, and since $E^{(i)}$ has a lower bound, the series $\left\{E^{(i)}\right\}_{i}$ can not be non-converging. Assume now that the convergence value for $\left\{E^{(i)}\right\}_{i}$, denoted $\hat{E}$, is not a (local) minimum, and denote the convergence point $(\hat{x}, \hat{u})$. According to lemma~\ref{lemma:global_taylor} \textit{GEORCE} will from the solution $(\hat{x}, \hat{u})$ in the following iteration produce a new point $(\tilde{x},\tilde{u})$ such that $E(\tilde{x},\tilde{u})<\hat{E}$, which is contradiction as the series $\left\{E^{(i)}\right\}_{i}$ is assumed to be converging to $\hat{E}$. The series, $\left\{E^{(i)}\right\}_{i}$, will therefore converge, and the convergence point is a (local) minimum.
\end{proof}

It is not possible for \textit{GEORCE} to jump between two local minimum points. To see this, assume that there are two local minimum points that have the same discretized energy, and assume that \textit{GEORCE} is jumping between the two points, i.e., there is convergence to a local minimum in terms of the discretized energy functional according to Proposition~\ref{prop:global_convergence}, but not to a single local minimum point. In the \textit{GEORCE} iteration, the line search will try to improve the discretized energy functional by finding a point closer to one of the two minimum points. Either this leads to a new point with improved discretized energy, or \textit{GEORCE} stops concluding that a local minimum point has been found since no improvements in discretized energy were possible by line-search in accordance with Lemma~\ref{lemma:global_taylor}. But this contradicts the assumption that \textit{GEORCE} jumps between two minimum points with the same discretized energy. Thus, \textit{GEORCE} converges to only one local minimum point.

\paragraph{Local convergence}
We will show that \textit{GEORCE} under certain regularity of the discretized energy functional asymptotically exhibits local convergence in the sense that
\begin{equation*}
    \begin{split}
        &\exists \epsilon>0 \, \exists C>0: \quad \norm{z^{(i+1)}-z^{(i)}} \leq f(T,\norm{z^{(i)}-z^{*}}), \\
        &\text{where }\lim_{T \rightarrow \infty} f\left(T,\norm{z^{(i)}-z^{*}}\right) = C \norm{z^{(i)}-z^{*}}^{2},
    \end{split}
\end{equation*}
We will assume the following regularity of the discretized energy functional
\begin{assumption}[Local convergence] \label{assum:quad_conv_assumptions}
    We assume the following regarding the discretized energy functional.
    \begin{itemize}
        \item We assume that the discretized energy functional $E(z)$ is locally strictly convex in the (local) minimum point $z^{*}=\left(x^{*},u^{*}\right)$ in the sense that
        \begin{equation*}
            \exists \epsilon>0: \, \forall z \in B_{\epsilon}\left(z^{*}\right), z \neq z^{*}: \, \forall \alpha ]0,1[: \, E\left(\left(1-\alpha\right)z+ \alpha z^{*}\right) < \left(1-\alpha\right)E(z)+\alpha E\left(z^{*}\right),
        \end{equation*}
        where $B_{\epsilon}=\left\{z \, |\, \norm{z-z^{*}} < \epsilon\right\}$.
        \item Assume that the discretized energy functional $E(z)$ is smooth, and consider the first order Taylor approximation of the discretized energy functional
        \begin{equation*}
            \Delta E = \langle \nabla E(z_{0}), \Delta z \rangle + \mathcal{O}\left(\Delta z\right)\norm{\Delta z},
        \end{equation*}
        Then, the term $\mathcal{O}\left(\Delta z\right)\norm{\Delta z}$ can be re-written as $\mathcal{O}\left(\norm{\Delta z}^{2}\right)$ locally.
    \end{itemize}
\end{assumption}
\begin{wrapfigure}[15]{r}{0.50\textwidth}
    \vspace{-7mm}
    \centering
    \includegraphics[width=0.50\textwidth]{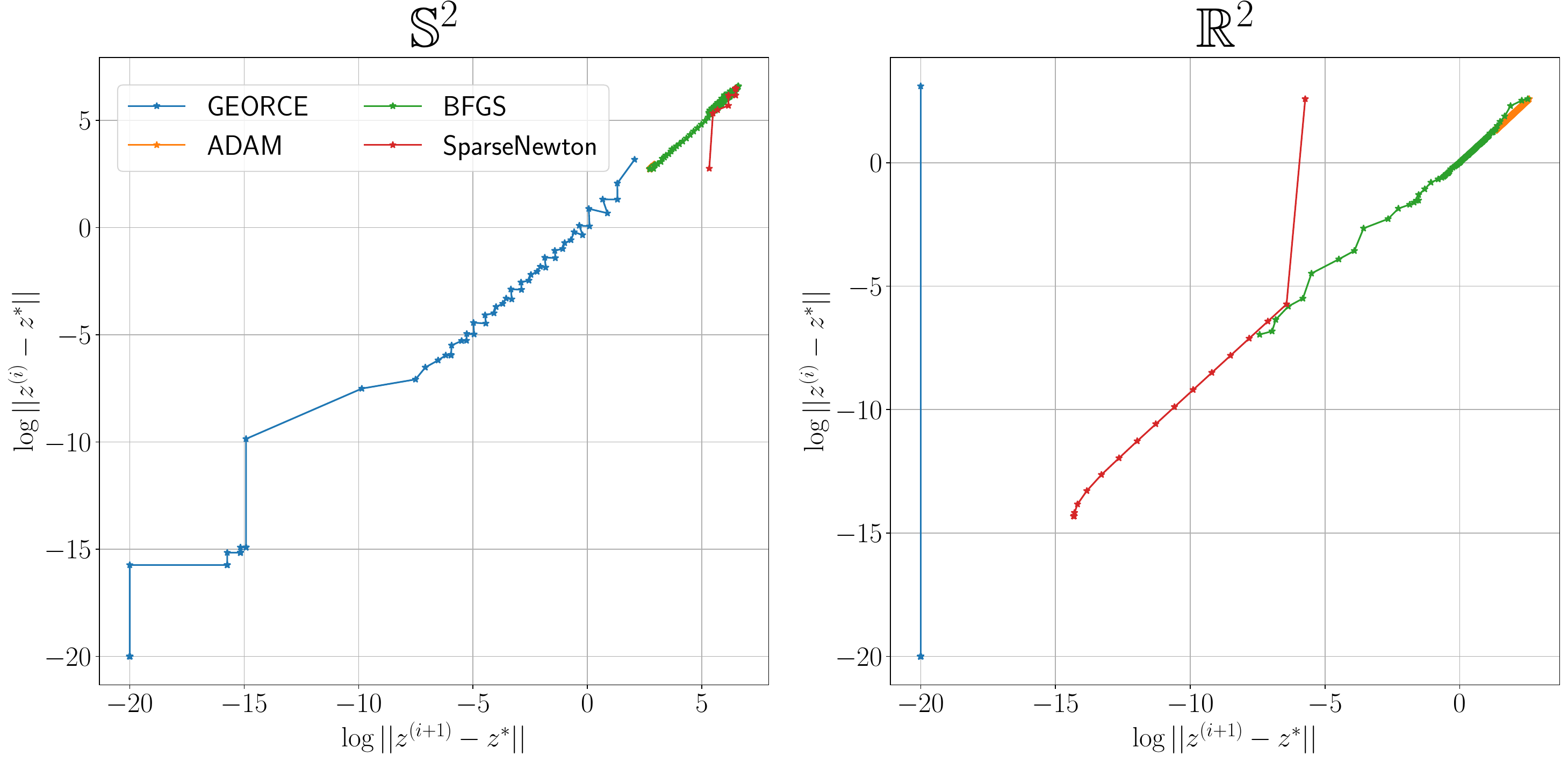}
    \vspace{-2mm}
    \caption{Illustration of the convergence of \textit{GEORCE} and baseline algorithms. We plot $\log \norm{z^{(i)}-z^{*}}$ over $\log \norm{z^{(i+1)}-z^{*}}$ for each iteration $i$ for the different methods. We use the convergence point of \textit{GEORCE} as $z^{*}$. If the norms are $0$, we set the value to $-20$. All initial curves are set to the same random deformation of a straight line similar to Appendix~\ref{ap:init_curve}.}
    \vspace{-10em}
    \label{fig:quadratic_convergence}
\end{wrapfigure}
Intuitively, it is expected that \textit{GEORCE} exhibits local asymptotic quadratic convergence. \textit{GEORCE} solves iteratively quadratic minimization problems by fixing the metric matrix function. Since \textit{GEORCE} exhibits global convergence, the metric matrix function in each iteration of \textit{GEORCE} will converge to the metric matrix function at the local minimum. Close to the local minimum, the metric matrix function will be approximately constant, since it is smooth and hence continuous. At that point, \textit{GEORCE} is simply solving a quadratic optimization problem, which is solved in one iteration similar to the Newton method, i.e,. has asymptotic quadratic convergence. It is therefore expected that when \textit{GEORCE} is sufficiently close to the minimum a large ``dip'' will occur. This is especially clear for the Euclidean space, where $G=I$ is independent of $x$. Therefore, \textit{GEORCE} will find the exact solution after one iteration, as it simply solves a quadratic minimization problem. It will be proved that a \textit{GEORCE} step without line search ($\alpha=1$) is equivalent to minimizing a quadratic polynomial similar to a Newton step, and that the difference between the two can be made sufficiently small for sufficiently high $T$ under certain regularity. We illustrate the convergence of \textit{GEORCE} and the baseline algorithms for $\mathbb{S}^{2}$ and $\mathbb{R}^{2}$ in Fig.~\ref{fig:quadratic_convergence}. 

\rowcolors{2}{white!10}{white}
\begin{lemma} \label{lemma:base_lemma}
    Let $z_{t}:=(x_{t},u_{t})$ and $\norm{u_{t}}^{2}_{x_{t}} = u_{t}^{\top}G(x_{t})u_{t}$, then the solution to the following optimal control problem is identical to a \textit{GEORCE} step without line-search, i.e. $\alpha=1$.
    \begin{equation} \label{eq:georce_op}
        \begin{split}
            \min_{\left(x_{t},u_{t}\right)} &\sum_{t=0}^{T-1}\norm{u_{t}^{(i)}}^{2}_{x_{t}^{(i)}} + \left(\nabla\norm{u_{t}^{(i)}}_{x_{t}^{(i)}}^{2}\right)^{\top}
            \begin{pmatrix}
                x_{t}-x_{t}^{(i)} \\
                u_{t}-u_{t}^{(i)}
            \end{pmatrix} \\
            \quad &+ \frac{1}{2}
            \begin{pmatrix}
                x_{t}-x_{t}^{(i)} & u_{t}-u_{t}^{(i)}
            \end{pmatrix}
            Q(x_{t}^{(i)},u_{t}^{(i)})
            \begin{pmatrix}
                x_{t}-x_{t}^{(i)} \\
                u_{t}-u_{t}^{(i)}
            \end{pmatrix} \\
            \text{s.t.} \quad &x_{t+1} = x_{t} + u_{t}, \quad t=0,\dots,T-1, \\
            &x_{0}=a, \, x_{T}=b, 
        \end{split}
    \end{equation}
    where
    \begin{equation*}
        Q(x_{t},u_{t}) = 
        \begin{pmatrix}
            \pmb{0} & \pmb{0} \\
            \pmb{0} & 2G(x_{t})
        \end{pmatrix}
    \end{equation*}
\end{lemma}
\begin{proof}
    The Hamiltonian function for Eq.~\ref{eq:georce_op} is
    \begin{equation*}
        \begin{split}
            H_{t}(x_{t},u_{t},\mu_{t}) &= \norm{u_{t}^{(i)}}^{2}_{x_{t}^{(i)}} + \left(\nabla \norm{u_{t}^{(i)}}^{2}_{x_{t}^{(i)}}\right)^{\top}
            \begin{pmatrix}
                x_{t}-x_{t}^{(i)} \\
                u_{t}-u_{t}^{(i)}
            \end{pmatrix} \\
            &+ \frac{1}{2}
            \begin{pmatrix}
                x_{t}-x_{t}^{(i)} & u_{t}-u_{t}^{(i)}
            \end{pmatrix}
            Q(x_{t}^{(i)},u_{t}^{(i)})
            \begin{pmatrix}
                x_{t}-x_{t}^{(i)} \\
                u_{t}-u_{t}^{(i)}
            \end{pmatrix}
            + \mu_{t}^{\top}(x_{t}+u_{t})
        \end{split}
    \end{equation*}
    Applying Pontryagins' maximum principle to Eq.~\ref{eq:georce_op} gives
    \begin{equation*}
        \begin{split}
            \min_{u_{t}} \quad &H_{t}(x_{t},u_{t},\mu_{t}) \\
            \text{s.t.} \quad &x_{t+1}=x_{t}+u_{t}, \quad t=0,\dots,T-1, \\
            &\nabla_{x_{t}}H_{t}(x_{t},u_{t},\mu_{t})=\mu_{t-1}, \quad t=1,\dots,T-1.
        \end{split}
    \end{equation*}
    As $H_{t}(x_{t},u_{t},\mu_{t})$ is strictly convex in $u_{t}$, then the optimal control $u_{t}^{*}$ is the stationary point to the Hamiltonian function, i.e.
    \begin{equation*}
        \begin{split}
            &\nabla_{u_{t}}H_{t}(x_{t},u_{t},\mu_{t}) = \nabla_{u_{t}}E(x_{t}^{(i)},u_{t}^{(i)}) + 2G(x_{t}^{(i)})(u_{t}-u_{t}^{(i)}) + \mu_{t} = 0 \\
            &\Leftrightarrow 2G(x_{t}^{(i)})u_{s}^{(i)} + \mu_{t} = 0, \quad t=0,\dots,T-1,
        \end{split}
    \end{equation*}
    since $\nabla_{u_{t}^{(i)}}\norm{u_{t}^{(i)}}^{2}_{2} = 2G(x_{t}^{(i)})u_{t}^{(i)}$. This is identical to the first equation in the equations for \textit{GEORCE}. The state equation is unchanged, and the co-state equation for the problem becomes
    \begin{equation*}
        \nabla_{x_{t}}H_{t}(x_{t},u_{t},\mu_{t}) = \nabla_{x_{t}}\norm{u_{t}}^{2}_{x_{t}} + \mu_{t} = \mu_{t-1}, \quad t=1,\dots,T-1.
    \end{equation*}
    which is identical to the co-state equation of \textit{GEORCE}.
\end{proof}

\begin{lemma} \label{lemma:k_mat}
    The solution to \textit{GEORCE} with no line-search ($\alpha=1$) can be formulated as a linear system of equations, i.e.
    \begin{equation*}
        K(T,z^{(i)})
        \begin{pmatrix}
            z \\
            \mu
        \end{pmatrix}
        = v,
    \end{equation*}
    where $K$ is a regular matrix and $z,\mu$ are vectors containing $(x_{t},u_{t})$ and $\mu_{t}$, respecitvely. 
\end{lemma}
\begin{proof}
    From Pontraygin's maximum principle the derived equations for \textit{GEORCE} are linear in $z_{t}$ and $\mu_{t}$ and can be expressed as a matrix-vector product. The vector $v$ includes the boundary values $a$ and $b$.

    Since $G(x)$ is strictly positive definite given $x$, then $(z,\mu)$ is unique, and therefore $K$ is a regular matrix, which is unique as proven in the update scheme of \textit{GEORCE}.
\end{proof}

\begin{corollary} \label{cor:minimum_cond}
    Consider the minimization problem
    \begin{equation*}
        \begin{split}
            \min_{z} \quad &E(z), \\
            \text{s.t.} \quad &Az = v,
        \end{split}
    \end{equation*}
    where $E$ is the discrete energy functional, $A$ is a constant matrix related to the boundary conditions and state equations and $v$ is a constant vector including the boundary values. Then at the minimum point $z^{*}$, the following condition is fulfilled
    \begin{equation*}
        \begin{split}
            &\exists\mu: \, \nabla E = -A^{\top}\mu, \\
            &Az =v
        \end{split}
    \end{equation*}
\end{corollary}
\begin{proof}
    The Lagrangian of the optimization problem is
    \begin{equation*}
        L(z,\mu) = E(z) + \mu^{\top}(Az-v).
    \end{equation*}
    Since $E$ is smooth, then the Karush-Kuhn-Tucker (KKT) conditions for a minimum point gives
    \begin{equation*}
        \begin{split}
            &\nabla E + A^{\top}\mu = 0, \\
            &Az = v.
        \end{split}
    \end{equation*}
\end{proof}

\begin{lemma} \label{lemma:convergence}
    Assume that for a local minimum solutions to the discrete energy functional satisfy
    \begin{equation*}
        \lim_{T \rightarrow \infty} \norm{u_{t}^{*}} = 0, \quad t=0,\dots,T-1,
    \end{equation*}
    where $\norm{\cdot}$ denotes the $\ell^{2}$-norm. Further assume that $G(\cdot)$ is a smooth matrix function. Then
    \begin{equation*}
        \lim_{T \rightarrow \infty} \norm{\nabla^{2}\norm{u_{t}^{*}}_{x_{t}^{*}}^{2} - Q(x_{t}^{*},u_{t}^{*})}_{\infty} = 0, \quad t=0,\dots,T-1,
    \end{equation*}
    where $\nabla^{2}$ denotes the Hessian, while $\norm{\cdot}_{\infty}$ denotes the $\ell^{\infty}$ norm.
\end{lemma}
\begin{proof}
    Since the discrete energy functional is smooth, consider the following derivatives for the Hessian of the discrete energy functional
    \begin{equation*}
        \begin{split}
            \norm{\nabla^{2}_{x_{t}^{*},x_{t}^{*}}\norm{u_{t}^{*}}_{x_{t}^{*}}^{2}}_{\infty} &= \norm{\nabla^{2}_{x_{t}^{*},x_{t}^{*}}(u_{t}^{\top}G(x_{t})u_{t})}_{\infty} \\
            &\leq \norm{\nabla^{2}_{x_{t}^{*},x_{t}^{*}}G(x_{t}^{*})}_{\infty}\norm{u_{t}^{*}}_{\infty} \\
            &\rightarrow 0 \quad \text{for }T \rightarrow \infty.
        \end{split}
    \end{equation*}
    \begin{equation*}
        \begin{split}
            \norm{\nabla^{2}_{x_{t},u_{t}}\norm{u_{t}}_{x_{t}}^{2}}_{\infty} &= \norm{\nabla_{x_{t}^{*}}(G(x_{t}^{*})u_{t}^{*})}_{\infty} \\
            &\leq \norm{G(x_{t}^{*})}_{\infty} \norm{u_{t}^{*}}_{\infty} \\
            &\rightarrow 0 \quad \text{for }T\rightarrow \infty.
        \end{split}
    \end{equation*}
    \begin{equation*}
        \nabla^{2}_{u_{t}^{*},u_{t}^{*}} \norm{u_{t}^{*}}_{x_{t}^{*}}^{2} = 2G(x_{t}^{*}).
    \end{equation*}
    Since
    \begin{equation*}
        \nabla^{2}\norm{u_{t}^{*}}_{x_{t}^{*}}^{2} - Q(x_{t}^{*},u_{t}^{*}) = 
        \begin{pmatrix}
            \nabla^{2}_{x_{t}^{*},x_{t}^{*}}\norm{u_{t}^{*}}_{x_{t}^{*}}^{2} & \nabla^{2}_{x_{t}^{*},u_{t}^{*}}\norm{u_{t}^{*}}_{x_{t}^{*}}^{2} \\
            \nabla^{2}_{u_{t}^{*},x_{t}^{*}}\norm{u_{t}^{*}}_{x_{t}^{*}}^{2} & \pmb{0}
        \end{pmatrix}
    \end{equation*}
    the result follows immediately.
\end{proof}
\begin{theorem} \label{thm:asymp_conv}
    Assume the following
    \begin{itemize}
        \item $G(\cdot)$ is smooth and thus locally Lipschitz continuous in the local minimum.
        \item Assume that 
        \begin{equation} \label{eq:assum_k_mat}
            \exists \epsilon>0\,\exists r>0 \, \forall T: \, \norm{K^{-1}(T,z^{i})}_{\infty} \leq r \quad \text{for} \quad z^{i} \in B_{\epsilon}(z^{*}).
        \end{equation}
        \item Assume that for local minimum solutions to the discrete energy functional that
        \begin{equation} \label{eq:u_convergence}
            \lim_{T \rightarrow \infty} \norm{u_{t}^{*}} = 0, \quad t=0,\dots,T-1.
        \end{equation}
    \end{itemize}
    Let $E$ denote the discrete functional. The Hessian of the objective function for the \textit{GEORCE}-optimization problem in Eq.~\ref{eq:georce_op} is denoted $\hat{Q}$.

    Then, \textit{GEORCE} exhibits asymptotic quadratic convergence, i.e.
    \begin{equation*}
        \begin{split}
            &\exists \epsilon>0 \, \exists C>0: \quad \norm{z^{(i+1)}-z^{(i)}} \leq f(T,\norm{z^{(i)}-z^{*}}), \\
            &\text{where }\lim_{T \rightarrow \infty} f\left(T,\norm{z^{(i)}-z^{*}}\right) = C \norm{z^{(i)}-z^{*}}^{2},
        \end{split}
    \end{equation*}
    for some smooth function $f$.
\end{theorem}

\begin{proof}
    Since $E(z)$ is a smooth function, consider the first order Taylor approximation for the gradient to $E(z)$ in the local minimum point $z^{*}$
    \begin{equation*}
        \nabla E(z^{(i)}) = \nabla E(z^{*}) + \nabla^{2}E(z^{*})(z^{i}-z^{*}) + \mathcal{O}\left(\norm{z^{(i)-z^{*}}}^{2}\right),
    \end{equation*}
    where $z^{(i)} \in B_{\epsilon}(z^{*})$. Consider a \textit{GEORCE} step from $z^{(i)}$, then the KKT conditions in Corrollary~\ref{cor:minimum_cond} on the equivalent optimization problem in Eq.~\ref{eq:georce_op} gives the following conditions for the solution $z^{(i+1)}$ from \textit{GEORCE}

    \begin{figure}[t!]
    \centering
    \includegraphics[width=1.0\textwidth]{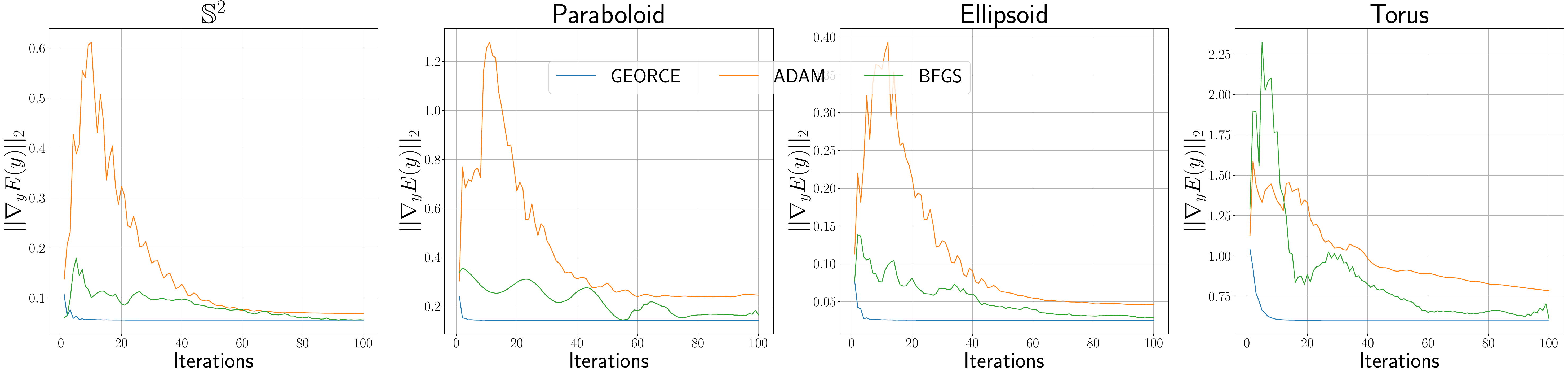}
    \caption{The figure shows the $\ell^{2}$-norm of the gradient of Eq.~\ref{eq:disc_const_energy} for each iteration for \textit{GEORCE} and alternative methods applied to construct geodesics for $100$ iterations for four different manifolds corresponding to Fig.~\ref{fig:synthetic_riemannian_geodesics}. We see that the $\ell^{2}$-norm of the gradient of Eq.~\ref{eq:disc_const_energy} for \textit{GEORCE} converges considerably faster than for alternative algorithms.}
    \label{fig:synthetic_riemannian_gradients}
    \vspace{-1.em}
\end{figure}

    \begin{equation*}
        \nabla E(z^{(i)}) + \hat{Q}(z^{(i)})(z^{(i+1)}-z^{i}) = -A^{\top}\mu^{(i+1)}.
    \end{equation*}
    Inserting the first order Taylor approximation into the KKT condition for the \textit{GEORCE}-step, a rearranging of terms gives the following equation
    \begin{equation*}
        \hat{Q}(z^{(i)})(z^{(i+1)}-z^{(i)}) + A^{\top}\mu^{(i+1)} = -\nabla E(z^{*}) - \nabla^{2}E(z^{*})(z^{(i)}-z^{*}) - \mathcal{O}\left(\norm{z^{(i+1)}-z^{*}}^{2}\right).
    \end{equation*}
    Let $e^{(i)} = z^{i}-z^{*}$. Utilizing the optimality condition in Corolarry~\ref{cor:minimum_cond} for $\nabla E(z^{*}) = -A^{\top}\mu^{*}$ and that $z^{(i+1)}-z^{(i)}=e^{(i+1)}-e^{(i)}$ and rearranging terms gives the following equation
    \begin{equation*}
        \hat{Q}(z^{(i)})e^{(i+1)} + A^{\top}(\mu^{(i+1)}-\mu^{(*)}) = \left(\hat{Q}(z^{(i)})-\nabla^{2}E(z^{*})\right)e^{(i)} + \mathcal{O}\left(\norm{e^{(i)}}^{2}\right).
    \end{equation*}
    The error term related to the Hessian can be decomposed as follows
    \begin{equation*}
        \left(\hat{Q}(z^{(i)})-\nabla^{2}E(z^{*})\right)e^{(i)} = \left(\hat{Q}(z^{(i)})-\hat{Q}(z^{*})\right)e^{(i)} + \left(\hat{Q}(z^{*}) - \nabla^{2}E(z^{*})\right)e^{(i)}.
    \end{equation*}
    As $\hat{Q}(z)$ is assumed Lipschitz continuous $\left(\hat{Q}(z^{(i)})-\hat{Q}(z^{*})\right)e^{(i)} = \mathcal{O}\left(\norm{e^{(i)}}^{2}\right)$, and the second term becomes an error term $\Delta (T)^{*} e^{(i)} = \left(\hat{Q}(z^{*}) - \nabla^{2}E(z^{*})\right)e^{i}$.
    
    As both $z^{(i+1)}$ and $z^{*}$ are feasible solutions to the optmization problem $Ae^{(i+1)}=0$, the previous equation can now be formulated as
    \begin{equation} \label{eq:dummy_eq_for_proof}
        \begin{pmatrix}
            \hat{Q}(z^{i}) & A^{\top} \\
            A & 0
        \end{pmatrix}
        \begin{pmatrix}
            e^{(i+1)} \\
            \mu^{(i+1)} - \mu^{(*)}
        \end{pmatrix}
        =
        \begin{pmatrix}
            \Delta(T)e^{(i)} + \mathcal{O}\left(\norm{e^{(i)}}^{2}\right) \\
            0
        \end{pmatrix}
    \end{equation}
    The matrix on the left hand side above is equivalent to $K(T,z^{(i)})$ in Lemma~\ref{lemma:k_mat} and is therefore regular. Utilizing the assumption in Eq.~\ref{eq:assum_k_mat} together with Eq.~\ref{eq:dummy_eq_for_proof}, then
    \begin{equation*}
        \exists M>0, \, \forall T>0: \, \norm{e^{(i+1)}} \leq M \left(\Delta(T)\norm{e^{(i)}} + \mathcal{O}\left(\norm{e^{(i)}}^{2}\right)\right).
    \end{equation*}
    Utilizing the assumption in Eq.~\ref{eq:u_convergence} and Lemma~\ref{lemma:convergence}, we have that
    \begin{equation*}
        \lim_{T \rightarrow \infty}\norm{\Delta (T)} = \lim_{T \rightarrow \infty}\norm{\hat{Q}(z^{*})-\nabla^{2}E(z^{*})}_{\infty} = 0.
    \end{equation*}
    Since Hessians of each term of the energy functional fulfill in line with Lemma~\ref{lemma:convergence} that
    \begin{equation*}
        \lim_{T \rightarrow \infty} \norm{\nabla^{2}E(x_{t}^{*},u_{t}^{*}) - Q(x_{t}^{*},u_{t}^{*})}_{\infty} = 0, \quad t=0,\dots,T-1.
    \end{equation*}
    This implies that in the limit for $T \rightarrow \infty$
    \begin{equation*}
        \norm{e^{(i+1)}} \leq \mathcal{O}\left(\norm{e^{(i)}}^{2}\right),
    \end{equation*}
    which implies that
    \begin{equation*}
        \exists C>0: \, \norm{e^{(i+1)}} \leq C \norm{e^{(i)}}^{2}.
    \end{equation*}
\end{proof}

\begin{wrapfigure}[15]{r}{0.50\textwidth}
    \vspace{-.5em}
    \centering
    \includegraphics[width=0.50\textwidth]{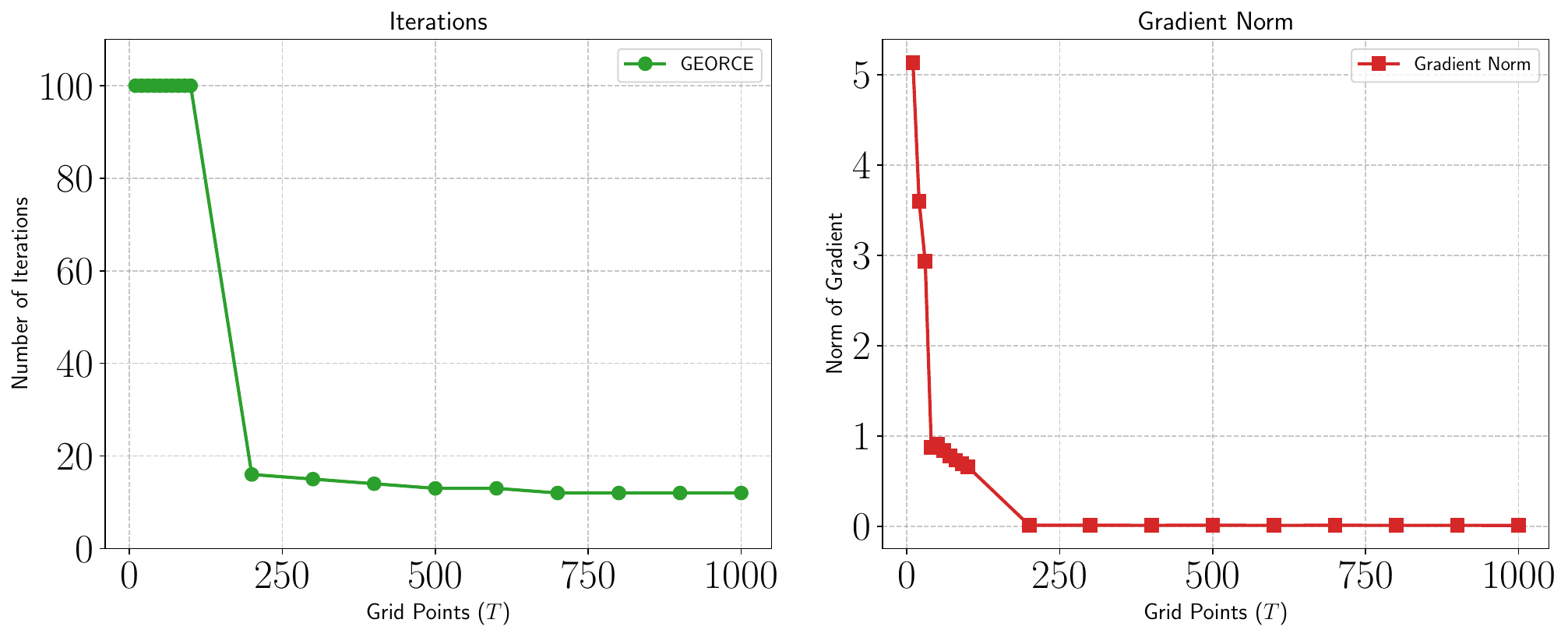}
    \caption{The asymptotic convergence of \textit{GEORCE}, where we show the number of iterations and the norm of the gradient as a function of the grid size. The result is for computing the geodesic for two points on $\mathbb{S}^{2}$. We use a tolerance of $10^{-2}$ for convergence, and we scale the discrete energy functional by the number of grid points, so that the tolerance is invariant to the number of grid points. The maximum number of iterations is set to $100$.}
    \label{fig:asymptotic_convergence}
    \vspace{-2.0em}
\end{wrapfigure}

The proof of Theorem~\ref{thm:asymp_conv} shows that in the limit of $T$ approaching infinity, \textit{GEORCE} has locally quadratic convergence, and it also shows that for increasing $T$, the error term controlling the "linear" convergence component becomes smaller. This means that the quadratic term will have more weights, i.e., the convergence speed will gradually improve by increasing $T$. We illustrate the convergence of \textit{GEORCE} for different values of grid size in Fig.~\ref{fig:asymptotic_convergence}, where we use the gradient of the discretized energy functional scaled by $T$ to make it invariant to the grid size.

Fig.~\ref{fig:synthetic_riemannian_gradients} shows the convergence by plotting the $\ell^{2}$-norm of the gradient of the discretized energy functional in each iteration for four different manifolds. Note that by formulating the geodesic boundary problem as a control problem, we estimate the discretized tangent vectors along the curve $\{u_{t}\}_{t=0}^{T}$. Modulo scaling $u_{0}$ will in the limit of a sufficiently fine grid correspond to the logarithmic map, $\mathrm{Log}_{a}b$.

\rowcolors{2}{gray!10}{white}

\section{Extension to Finslerian geometry}
A Finsler geometry can be seen as a generalization of Riemannian geometry, where the tangent spaces are not equipped with an quadratic form metric but with a Finsler metric instead. A Finsler metric, $F: \mathcal{M} \times T\mathcal{M} \rightarrow \mathbb{R}_{+}$, where $\mathcal{M}$ is a differentiable manifold, is a Minkowski norm in the sense that $F$ is $C^{\infty}$ on $T\mathcal{M}\setminus \{0\}$, $F(x,cv)=cF(x,v)$ for all $v \in T\mathcal{M}$, $x \in \mathcal{M}$ and $c>0$ and for any $v \in T\mathcal{M}$, then $\frac{\partial^{2}F}{\partial v^{i}\partial v^{j}}(x,v)$ is positive definite \citep[Chapter 2]{ohta2021comparison}. 
\begin{wrapfigure}[15]{r}{0.50\textwidth}
    \vspace{-6mm}
    \centering
    \includegraphics[width=0.50\textwidth]{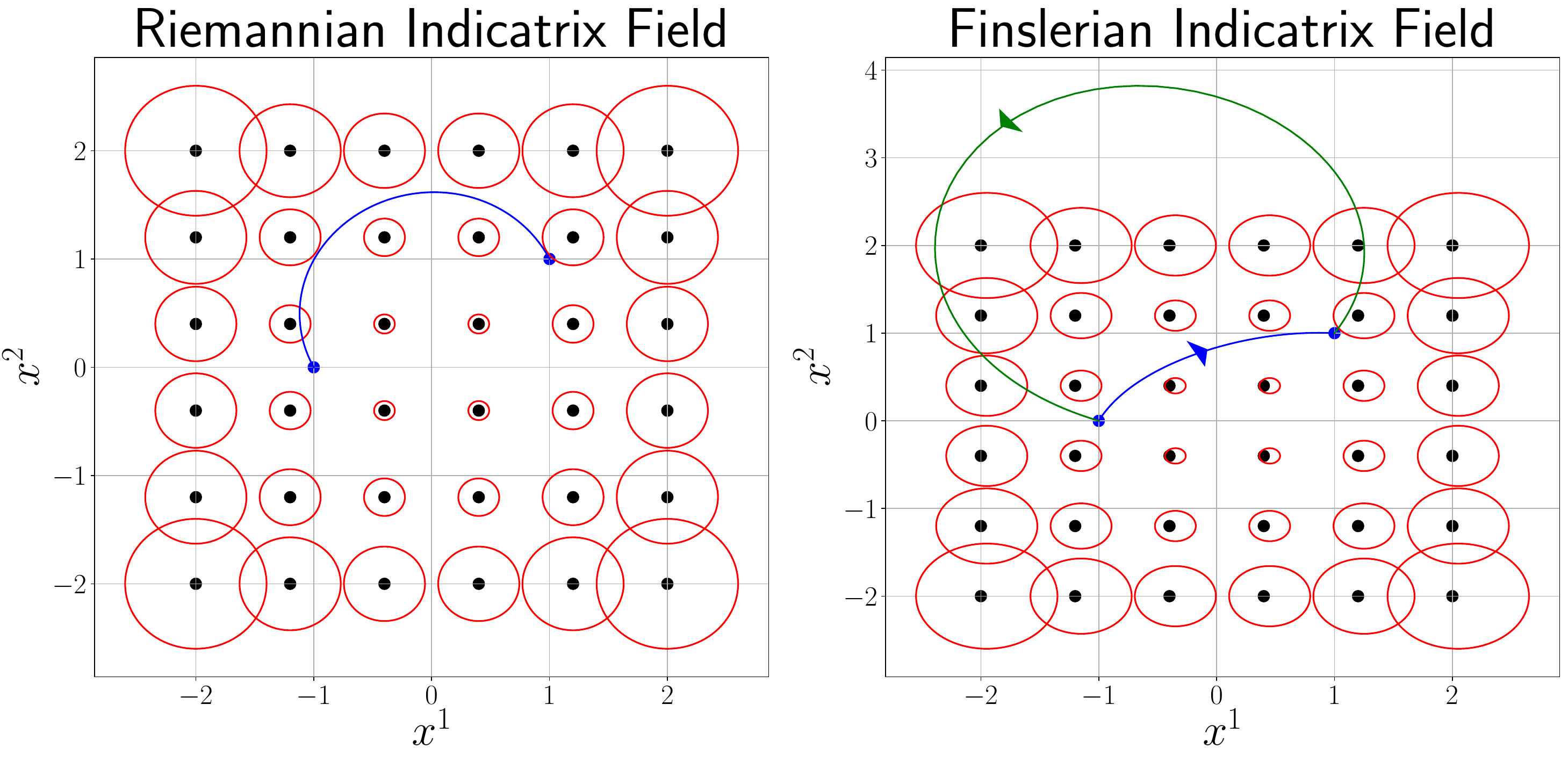}
    \vspace{-5mm}
    \caption{The red circles shows the indicatrix field, $\{v \in T_{x}\mathcal{M} \,|\,F(x,v)=1\}$, in the chart of the sphere used in the Riemannian case in Fig.~\ref{fig:synthetic_riemannian_geodesics} and its corresponding center is marked with black dot. The left figure is the Riemannian case, while the right figure is the Finslerian case corresponding to a force field making the indicatrix-field become non-centered ellipsis for the geodesics (blue and green), where the arrow indicates the direction of the curve.}
    \vspace{0.0em}
    \label{fig:finsler_indicatrix_field}
\end{wrapfigure}
Geodesics for a Finsler manifold can be found equivalently to the Riemannian case by minimizing the energy functional, $\int_{0}^{1}F^{2}\left(\gamma(t), \dot{\gamma}(t)\right) \,\dif t$, on a suitable set of curves connecting the two points thereby approximating a geodesic $\gamma$, \citep[Chapter 3]{ohta2021comparison}. In this way, the indicatrix field for Finslerian manifolds consists of non-centered ellipses unlike the Riemannian case, meaning that the distance is not symmetric for Finslerian manifolds. We illustrate this in Fig.~\ref{fig:finsler_indicatrix_field} for the local chart for the sphere used in Fig.~\ref{fig:synthetic_riemannian_geodesics}, where we shift the indicatrix fields by a wind-field on the form $(0,0.75)$.

\textit{GEORCE} is similarly applicable to Finsler geometry. First, we note that the energy functional can be written as \citep[Page 20]{ohta2021comparison}
\begin{equation} \label{eq:finsler_fundamental_energy}
    \frac{1}{2}\int_{0}^{1}F^{2}\left(\gamma(t), \dot{\gamma}(t)\right) \,\dif t = \int_{0}^{1}\dot{\gamma}(t)^{\top}G\left(\gamma(t), \dot{\gamma}(t)\right)\dot{\gamma}(t) \,\dif t,
\end{equation}
where $G(x,v) = \frac{1}{2}\frac{\partial^{2}F^{2}}{\partial v^{i} \partial v^{j}}\left(x,v\right)$ denotes the fundamental tensor in local coordinates\footnote{In literature, the fundamental tensor is often denoted with a lower case $g$. To have the notation consistent with the Riemannian case, we opt to use an upper case $G$ instead.}, which is positive definite \citep[Page 19]{ohta2021comparison}. Unlike the Riemannian case, the fundamental tensor, $G$, is now a function of both the position and velocity, which implies the update formulas in Proposition~\ref{prop:update_scheme}. Similar to Riemannian case, we can formulate the discrete energy functional for the Finsler case as the following discretized energy functional in Eq.~\ref{eq:finsler_fundamental_energy} becomes
\begin{equation}
    \begin{split}
        E_{F}(x_{0:T}) := \quad \min_{x_{0:T}} \quad &\sum_{t=0}^{T-1} (x_{t+1}-x_{t})^{\top}G(x_{t},x_{t+1}-x_{t})(x_{t+1}-x_{t}) \\
        \text{s.t.} \quad &x_{0}=a,x_{T}=b,
    \end{split}
\end{equation}
where the corresponding discrete control problem becomes
\begin{equation} \label{eq:finsler_disc_const_energy}
    \begin{split}
        \min_{(x_{1:T-1},u_{0:T-1})} \quad &\sum_{t=0}^{T-1} u_{t}^{\top}G\left(x_{t}, u_{t}\right)u_{t}, \quad t=0,\dots,T-1,\\
        \text{s.t.} \quad &x_{t+1}=x_{t}+u_{t} \\
        &x_{0}=a,x_{T}=b,
    \end{split}
\end{equation}
Since, the fundamental tensor, $G$, depends on both the state, $x_{0:T}$, and control variable $u_{0:T}$, the necessary conditions in Proposition~\ref{prop:riemann_cond} have to be modified slightly to take this into account. Taking this into account and following the same approach as in the Riemannian case, we get the following update scheme in the Finslerian case (see Appendix~\ref{ap:finsler_update} for details)
\begin{equation} \label{eq:finsler_energy_update_schem}
    \begin{split}
        &\mu_{T-1} = \left(\sum_{t=0}^{T-1}G_{t}^{-1}\right)^{-1}\left(2(a-b)-\sum_{t=0}^{T-1}G_{t}^{-1}\left(\zeta_{t}+\sum_{t>j}^{T-1}\nu_{j}\right)\right), \\
        &u_{t} = -\frac{1}{2}G_{t}^{-1}\left(\mu_{T-1}+\zeta_{t}+\sum_{j>t}^{T-1}\nu_{j}\right), \quad t=0,\dots,T-1, \\
        &x_{t+1} = x_{t}+u_{t}, \quad t=0,\dots,T-1, \\
        &x_{0}=a,
    \end{split}
\end{equation}
where in iteration $i$ we fix
\begin{equation*}
    \begin{split}
        &\nu_{t} := \restr{\nabla_{y}\left(u_{t}^{\top}G\left(y,u_{t}\right)u_{t}\right)}{y=x_{t}^{(i)},u_{t}=u_{t}^{(i)}}, \quad t=1,\dots,T-1, \\
        &\zeta_{t} := \restr{\nabla_{y}\left(u_{t}^{\top}G\left(x_{t},y\right)u_{t}\right)}{x_{t}=x_{t}^{(i)},u_{t}=u_{t}^{(i)}, y=u_{t}^{(i)}}, \quad t=1,\dots,T-1, \\
        &G_{t} := G\left(x_{t}^{(i)}, u_{t}^{(i)}\right), \quad t=0,\dots,T-1.
    \end{split}
\end{equation*}
We see that the only extra term in \textit{GEORCE} for Finsler manifolds is $\{\zeta_{t}\}_{t=0}^{T-1}$. Using the modified update scheme in Eq.~\ref{eq:finsler_energy_update_schem}, we estimate geodesics for Finsler manifolds similar to Algorithm~\ref{al:georce}. For completeness, we show the algorithm in pseudo-code for the Finsler case in Appendix~\ref{ap:georce_al_finsler}.

Using the same technique as in the Riemannian case, we show that \textit{GEORCE} for Finsler manifolds has global convergence. We refer to the details in Appendix~\ref{ap:finsler_proof}.

Note that the proof for asymptotic local quadratic convergence in the Riemannian case is easily extended to the Finslerian case, since the state equations are unchanged and $G:=G(x_{t},u_{t})$ which simply corresponds to define a variable $z_{t}=(x_{t},u_{t})$. 

\clearpage
\section{Experiments} \label{sec:experiments}
\begin{figure}[t!]
    \centering
    \includegraphics[width=1.0\textwidth]{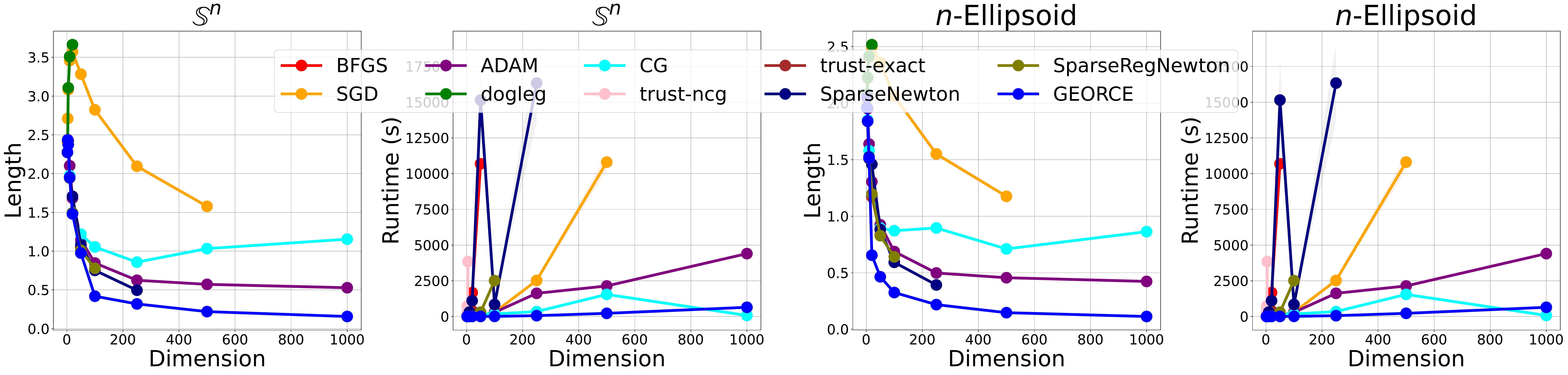}
    \caption{The first two figures display the length and runtime of geodesics for $\mathbb{S}^{n}$ estimated using different algorithms from Table~\ref{tab:riemmannian_comparison_table}. The latter row shows the length and runtime, respectively, for the $n$-Ellipsoid with half axes of $n$ equally spaced points between $0.5$ and $1.0$. The optimization solvers are described in Appendix~\ref{ap:hyper_parameters}.}
    \label{fig:sphere_ellipsoid_runtime}
    \vspace{-1.0em}
\end{figure}
\begin{sidewaystable}
    \begin{tabular*}{\textheight}{@{\extracolsep\fill}lcccccc}
        \toprule%
        & \multicolumn{2}{c}{\textbf{BFGS (T=100)}} & \multicolumn{2}{c}{\textbf{ADAM  (T=100)}} & \multicolumn{2}{c}{\textbf{GEORCE  (T=100)}} \\\cmidrule{2-3}\cmidrule{4-5}\cmidrule{6-7}%
        Riemannian Manifold & Length & Runtime & Length & Runtime & Length & Runtime \\
        \midrule
        $\mathbb{S}^{2}$ & $2.2742$ & $57.4164 \pm 0.0616$ & $2.2772$ & $0.2335 \pm 0.0005$ & $\pmb{2.2741}$ & $\pmb{0.0078} \pm \pmb{ 0.0000 }$ \\ 
        $\mathbb{S}^{3}$ & $\pmb{2.4342}$ & $67.4657 \pm 0.4874$ & $2.4377$ & $0.3475 \pm 0.0010$ & $2.4342$ & $\pmb{0.0635} \pm \pmb{ 0.0001 }$ \\ 
        $\mathbb{S}^{5}$ & $2.3757$ & $77.6975 \pm 0.0741$ & $2.4222$ & $1.9894 \pm 0.0188$ & $\pmb{2.3755}$ & $\pmb{0.2219} \pm \pmb{ 0.0001 }$ \\ 
        $\mathbb{S}^{10}$ & $1.9529$ & $150.0621 \pm 0.8919$ & $2.1016$ & $4.4208 \pm 0.0308$ & $\pmb{1.9523}$ & $\pmb{0.3184} \pm \pmb{ 0.0002 }$ \\ 
        $\mathbb{S}^{20}$ & $1.4810$ & $1068.8645 \pm 17.2001$ & $1.6839$ & $8.7097 \pm 0.0352$ & $\pmb{1.4808}$ & $\pmb{0.6732} \pm \pmb{ 0.0204 }$ \\ 
        $\mathbb{S}^{50}$ & $1.0464$ & $7350.6001 \pm 244.3737$ & $1.1463$ & $61.8926 \pm 0.3034$ & $\pmb{0.9749}$ & $\pmb{4.1016} \pm \pmb{ 0.0201 }$ \\ 
        $\mathbb{S}^{100}$ & $-$ & $-$ & $0.8466$ & $277.4217 \pm 0.2280$ & $\pmb{0.4187}$ & $\pmb{10.9208} \pm \pmb{ 0.0071 }$ \\ 
        $\mathbb{S}^{250}$ & $-$ & $-$ & $0.6244$ & $1424.9031 \pm 22.1592$ & $\pmb{0.3192}$ & $\pmb{78.2937} \pm \pmb{ 0.0991 }$ \\ 
        $\mathbb{S}^{500}$ & $-$ & $-$ & $0.5717$ & $2242.4697 \pm 17.9769$ & $\pmb{0.2198}$ & $\pmb{224.2761} \pm \pmb{ 0.4894 }$ \\ 
        $\mathbb{S}^{1000}$ & $-$ & $-$ & $0.5279$ & $4329.6621 \pm 38.0163$ & $\pmb{0.1573}$ & $\pmb{893.4682} \pm \pmb{ 7.3130 }$ \\ 
        \hline
        $\mathrm{E}\left( 2 \right)$ & $1.9595$ & $56.9650 \pm 0.4106$ & $1.9596$ & $0.2328 \pm 0.0001$ & $\pmb{1.9595}$ & $\pmb{0.0067} \pm \pmb{ 0.0001 }$ \\ 
        $\mathrm{E}\left( 3 \right)$ & $\pmb{2.0499}$ & $94.1769 \pm 0.7044$ & $2.0525$ & $0.3392 \pm 0.0024$ & $2.0501$ & $\pmb{0.0119} \pm \pmb{ 0.0000 }$ \\ 
        $\mathrm{E}\left( 5 \right)$ & $\pmb{1.8403}$ & $194.5403 \pm 1.3293$ & $1.9521$ & $2.2227 \pm 0.0186$ & $1.8410$ & $\pmb{0.4220} \pm \pmb{ 0.0038 }$ \\ 
        $\mathrm{E}\left( 10 \right)$ & $\pmb{1.5114}$ & $427.6278 \pm 4.1044$ & $1.6384$ & $4.5697 \pm 0.0541$ & $1.5139$ & $\pmb{0.4566} \pm \pmb{ 0.0016 }$ \\ 
        $\mathrm{E}\left( 20 \right)$ & $1.1795$ & $1669.2692 \pm 10.2877$ & $1.3049$ & $9.4335 \pm 0.0507$ & $\pmb{0.6548}$ & $\pmb{1.9786} \pm \pmb{ 0.0031 }$ \\ 
        $\mathrm{E}\left( 50 \right)$ & $0.8338$ & $10685.6133 \pm 321.6447$ & $0.9244$ & $73.5682 \pm 0.3306$ & $\pmb{0.4642}$ & $\pmb{2.4208} \pm \pmb{ 0.0222 }$ \\ 
        $\mathrm{E}\left( 100 \right)$ & $-$ & $-$ & $0.6872$ & $289.4196 \pm 2.5224$ & $\pmb{0.3249}$ & $\pmb{7.7128} \pm \pmb{ 0.0714 }$ \\ 
        $\mathrm{E}\left( 250 \right)$ & $-$ & $-$ & $0.4978$ & $1620.6428 \pm 16.9546$ & $\pmb{0.2178}$ & $\pmb{56.8527} \pm \pmb{ 0.1526 }$ \\ 
        $\mathrm{E}\left( 500 \right)$ & $-$ & $-$ & $0.4559$ & $2135.6899 \pm 30.5699$ & $\pmb{0.1467}$ & $\pmb{220.1656} \pm \pmb{ 0.9902 }$ \\ 
        $\mathrm{E}\left( 1000 \right)$ & $-$ & $-$ & $0.4232$ & $4398.9453 \pm 25.8738$ & $\pmb{0.1131}$ & $\pmb{641.7199} \pm \pmb{ 20.9901 }$ \\ 
        \hline
        $\mathbb{T}^{2}$ & $9.9851$ & $92.7400 \pm 3.1267$ & $9.9855$ & $0.1060 \pm 0.0002$ & $\pmb{9.9851}$ & $\pmb{0.0182} \pm \pmb{ 0.0003 }$ \\ 
        \hline
        $\mathbb{H}^{2}$ & $1.2499$ & $75.0079 \pm 0.1710$ & $1.2499$ & $0.0988 \pm 0.0001$ & $\pmb{1.2498}$ & $\pmb{0.0025} \pm \pmb{ 0.0001 }$ \\ 
        \hline
        $\mathcal{P}\left( 2 \right)$ & $\pmb{1.0669}$ & $88.2739 \pm 0.4152$ & $1.0879$ & $0.2611 \pm 0.0004$ & $1.0669$ & $\pmb{0.0141} \pm \pmb{ 0.0001 }$ \\ 
        $\mathcal{P}\left( 3 \right)$ & $\pmb{2.2639}$ & $168.5082 \pm 0.3039$ & $2.2896$ & $2.3780 \pm 0.0046$ & $2.2639$ & $\pmb{0.1442} \pm \pmb{ 0.0002 }$ \\ 
        \hline
        Gaussian Distribution & $\pmb{2.5952}$ & $47.6288 \pm 0.1359$ & $2.5956$ & $0.0167 \pm 0.0001$ & $2.5952$ & $\pmb{0.0015} \pm \pmb{ 0.0000 }$ \\ 
        Fr\'echet Distribution & $\pmb{1.1538}$ & $51.3063 \pm 0.0566$ & $1.1542$ & $0.0185 \pm 0.0001$ & $1.1539$ & $\pmb{0.0009} \pm \pmb{ 0.0000 }$ \\ 
        Cauchy Distribution & $1.6454$ & $36.5031 \pm 0.0437$ & $1.6466$ & $0.0177 \pm 0.0001$ & $\pmb{1.6452}$ & $\pmb{0.0020} \pm \pmb{ 0.0000 }$ \\ 
        Pareto Distribution & $0.8297$ & $39.0465 \pm 0.0471$ & $0.8336$ & $0.0168 \pm 0.0000$ & $\pmb{0.8297}$ & $\pmb{0.0007} \pm \pmb{ 0.0000 }$ \\ 
        \bottomrule
    \end{tabular*}
    \caption{The length of the estimated geodesics for the \textit{BFGS}-algorithm, \textit{ADAM} and \textit{GEORCE} on a CPU. The methods were terminated if the $\ell^{2}$-norm of the gradient was less than $10^{-4}$ or after $1,000$ iterations. $\mathrm{E}(n)$ denotes an Ellipsoid of dimension $n$ with half axes of $n$ equally spaced points between $0.5$ and $1.0$, while $\mathcal{P}(n)$ denotes the space of $n \times n$ symmetric positive definite matrices. When the computational time was longer than $24$ hours, the value is set to $-$.} \label{tab:riemmannian_comparison_table}
\end{sidewaystable}

\begin{sidewaystable}
    \begin{tabular*}{\textheight}{@{\extracolsep\fill}lcccccc}
        \toprule%
        & \multicolumn{2}{c}{\textbf{ADAM (T=100)}} & \multicolumn{2}{c}{\textbf{Sparse Newton  (T=100)}} & \multicolumn{2}{c}{\textbf{GEORCE  (T=100)}} \\\cmidrule{2-3}\cmidrule{4-5}\cmidrule{6-7}%
        Riemannian Manifold & Length & Runtime & Length & Runtime & Length & Runtime \\
        \midrule
        $\mathbb{S}^{2}$ & $2.2755$ & $0.2729 \pm 0.0016$ & $2.2742$ & $0.1417 \pm 0.0041$ & $\pmb{2.2741}$ & $\pmb{0.0163} \pm \pmb{ 0.0003 }$ \\ 
        $\mathbb{S}^{3}$ & $2.4376$ & $0.2890 \pm 0.0005$ & $\pmb{2.4342}$ & $0.2188 \pm 0.0039$ & $2.4342$ & $\pmb{0.0735} \pm \pmb{ 0.0018 }$ \\ 
        $\mathbb{S}^{5}$ & $2.4226$ & $0.3153 \pm 0.0027$ & $\pmb{2.3755}$ & $0.8657 \pm 0.0015$ & $2.3755$ & $\pmb{0.0494} \pm \pmb{ 0.0014 }$ \\ 
        $\mathbb{S}^{10}$ & $2.1031$ & $0.3052 \pm 0.0079$ & $1.9653$ & $117.4589 \pm 0.0140$ & $\pmb{1.9523}$ & $\pmb{0.0351} \pm \pmb{ 0.0001 }$ \\ 
        $\mathbb{S}^{20}$ & $1.6824$ & $0.3401 \pm 0.0117$ & $1.7069$ & $146.3024 \pm 0.0191$ & $\pmb{1.4808}$ & $\pmb{0.0483} \pm \pmb{ 0.0020 }$ \\ 
        $\mathbb{S}^{50}$ & $1.1456$ & $0.5339 \pm 0.0014$ & $1.0359$ & $310.2160 \pm 0.0200$ & $\pmb{0.9749}$ & $\pmb{0.0587} \pm \pmb{ 0.0023 }$ \\ 
        $\mathbb{S}^{100}$ & $0.8475$ & $1.2935 \pm 0.0014$ & $0.7522$ & $6.6464 \pm 0.0016$ & $\pmb{0.4188}$ & $\pmb{0.0717} \pm \pmb{ 0.0022 }$ \\ 
        $\mathbb{S}^{250}$ & $0.6251$ & $3.8386 \pm 0.0018$ & $0.4958$ & $98.6502 \pm 0.0147$ & $\pmb{0.3192}$ & $\pmb{0.3296} \pm \pmb{ 0.0005 }$ \\ 
        $\mathbb{S}^{500}$ & $0.5714$ & $8.2031 \pm 0.0030$ & $0.3787$ & $899.4549 \pm 0.0188$ & $\pmb{0.2198}$ & $\pmb{1.2925} \pm \pmb{ 0.0003 }$ \\ 
        $\mathbb{S}^{1000}$ & $0.5270$ & $19.2234 \pm 0.0712$ & $-$ & $-$ & $\pmb{0.1573}$ & $\pmb{6.8579} \pm \pmb{ 0.0023 }$ \\ 
        \hline
         & $1.9595$ & $0.2735 \pm 0.0074$ & $\pmb{1.9595}$ & $0.2177 \pm 0.0075$ & $1.9595$ & $\pmb{0.0138} \pm \pmb{ 0.0001 }$ \\ 
        $\mathrm{E}\left( 3 \right)$ & $2.0523$ & $0.2839 \pm 0.0056$ & $\pmb{2.0498}$ & $0.3705 \pm 0.0063$ & $2.0501$ & $\pmb{0.0167} \pm \pmb{ 0.0001 }$ \\ 
        $\mathrm{E}\left( 5 \right)$ & $1.9594$ & $0.2948 \pm 0.0116$ & $-$ & $-$ & $\pmb{1.8410}$ & $\pmb{0.0827} \pm \pmb{ 0.0036 }$ \\ 
        $\mathrm{E}\left( 10 \right)$ & $1.6421$ & $0.3097 \pm 0.0133$ & $\pmb{1.4247}$ & $116.6572 \pm 0.0186$ & $1.5139$ & $\pmb{0.0547} \pm \pmb{ 0.0019 }$ \\ 
        $\mathrm{E}\left( 20 \right)$ & $1.3052$ & $0.3875 \pm 0.0058$ & $1.4583$ & $148.6781 \pm 0.0307$ & $\pmb{0.6514}$ & $\pmb{0.1449} \pm \pmb{ 0.0054 }$ \\ 
        $\mathrm{E}\left( 50 \right)$ & $0.9245$ & $0.5704 \pm 0.0013$ & $0.8849$ & $316.9808 \pm 0.0179$ & $\pmb{0.4642}$ & $\pmb{0.0370} \pm \pmb{ 0.0014 }$ \\ 
        $\mathrm{E}\left( 100 \right)$ & $0.6873$ & $1.3063 \pm 0.0020$ & $0.5922$ & $5.2868 \pm 0.0007$ & $\pmb{0.3249}$ & $\pmb{0.0524} \pm \pmb{ 0.0002 }$ \\ 
        $\mathrm{E}\left( 250 \right)$ & $0.4974$ & $4.2084 \pm 0.0029$ & $0.3917$ & $84.6331 \pm 0.0439$ & $\pmb{0.2180}$ & $\pmb{0.2325} \pm \pmb{ 0.0001 }$ \\ 
        $\mathrm{E}\left( 500 \right)$ & $0.4559$ & $7.4772 \pm 0.0024$ & $0.3039$ & $827.7163 \pm 0.1003$ & $\pmb{0.1468}$ & $\pmb{1.2897} \pm \pmb{ 0.0004 }$ \\ 
        $\mathrm{E}\left( 1000 \right)$ & $0.4238$ & $19.8033 \pm 0.0970$ & $-$ & $-$ & $\pmb{0.1129}$ & $\pmb{5.1862} \pm \pmb{ 0.0071 }$ \\ 
        \hline
        $\mathbb{T}^{2}$ & $9.9853$ & $0.2890 \pm 0.0016$ & $9.9851$ & $0.3353 \pm 0.0126$ & $\pmb{9.9851}$ & $\pmb{0.0549} \pm \pmb{ 0.0010 }$ \\ 
        \hline
        $\mathbb{H}^{2}$ & $1.2500$ & $0.3201 \pm 0.0004$ & $1.2498$ & $0.2125 \pm 0.0128$ & $\pmb{1.2498}$ & $\pmb{0.0089} \pm \pmb{ 0.0003 }$ \\ 
        \hline
        $\mathcal{P}\left( 2 \right)$ & $1.0786$ & $0.3801 \pm 0.0047$ & $\pmb{1.0669}$ & $0.2419 \pm 0.0069$ & $1.0669$ & $\pmb{0.0227} \pm \pmb{ 0.0010 }$ \\ 
        $\mathcal{P}\left( 3 \right)$ & $2.3006$ & $0.4115 \pm 0.0009$ & $2.3810$ & $106.2814 \pm 0.0135$ & $\pmb{2.2639}$ & $\pmb{0.0343} \pm \pmb{ 0.0004 }$ \\ 
        \hline
        $\text{Gaussian Distribution}$ & $\pmb{2.5951}$ & $0.1809 \pm 0.0010$ & $2.5952$ & $0.1108 \pm 0.0055$ & $2.5952$ & $\pmb{0.0123} \pm \pmb{ 0.0002 }$ \\ 
        $\text{Fr\'echet Distribution}$ & $\pmb{1.1538}$ & $0.1778 \pm 0.0029$ & $1.1539$ & $0.0813 \pm 0.0044$ & $1.1539$ & $\pmb{0.0055} \pm \pmb{ 0.0003 }$ \\ 
        $\text{Cauchy Distribution}$ & $1.6452$ & $0.1818 \pm 0.0007$ & $1.6453$ & $0.1136 \pm 0.0067$ & $\pmb{1.6452}$ & $\pmb{0.0137} \pm \pmb{ 0.0001 }$ \\ 
        $\text{Pareto Distribution}$ & $0.8314$ & $0.1779 \pm 0.0026$ & $\pmb{0.8297}$ & $0.0758 \pm 0.0036$ & $0.8297$ & $\pmb{0.0043} \pm \pmb{ 0.0002 }$ \\ 
        \bottomrule
    \end{tabular*}
    \caption{The table shows the length of the estimated geodesics for different methods on a GPU. The methods have been stopped if the norm of the gradient was less than $10^{-4}$ or if a maximum of $1,000$ iterations has been reached. When the computational time was longer than $24$ hours, the value is set to $-$.}
    \label{tab:riemannian_runtime1_gpu_T100}
\end{sidewaystable}

\begin{sidewaystable}
    \begin{tabular*}{\textheight}{@{\extracolsep\fill}lcccccc}
        \toprule%
        & \multicolumn{2}{c}{\textbf{ADAM (T=50)}} & \multicolumn{2}{c}{\textbf{Sparse Newton  (T=50)}} & \multicolumn{2}{c}{\textbf{GEORCE  (T=50)}} \\\cmidrule{2-3}\cmidrule{4-5}\cmidrule{6-7}%
        Riemannian Manifold & Length & Runtime & Length & Runtime & Length & Runtime \\
        \midrule
        $\mathbb{S}^{2}$ & $2.2479$ & $0.2683 \pm 0.0082$ & $2.2478$ & $0.1094 \pm 0.0047$ & $\pmb{2.2477}$ & $\pmb{0.0187} \pm \pmb{ 0.0002 }$ \\ 
        $\mathbb{S}^{3}$ & $2.4069$ & $0.2593 \pm 0.0110$ & $\pmb{2.4064}$ & $0.1370 \pm 0.0082$ & $2.4064$ & $\pmb{0.0866} \pm \pmb{ 0.0031 }$ \\ 
        $\mathbb{S}^{5}$ & $2.3494$ & $0.2907 \pm 0.0103$ & $2.3480$ & $0.3153 \pm 0.0063$ & $\pmb{2.3480}$ & $\pmb{0.0622} \pm \pmb{ 0.0019 }$ \\ 
        $\mathbb{S}^{10}$ & $1.9470$ & $0.3026 \pm 0.0014$ & $1.9274$ & $0.8966 \pm 0.0034$ & $\pmb{1.9272}$ & $\pmb{0.0540} \pm \pmb{ 0.0016 }$ \\ 
        $\mathbb{S}^{20}$ & $1.5036$ & $0.3349 \pm 0.0069$ & $1.4610$ & $28.7634 \pm 0.0016$ & $\pmb{1.4566}$ & $\pmb{0.0670} \pm \pmb{ 0.0003 }$ \\ 
        $\mathbb{S}^{50}$ & $1.0320$ & $0.4713 \pm 0.0025$ & $1.0060$ & $39.8257 \pm 0.0010$ & $\pmb{0.5279}$ & $\pmb{0.0438} \pm \pmb{ 0.0002 }$ \\ 
        $\mathbb{S}^{100}$ & $0.7642$ & $0.8173 \pm 0.0009$ & $0.6922$ & $4.2461 \pm 0.0007$ & $\pmb{0.3932}$ & $\pmb{0.0704} \pm \pmb{ 0.0027 }$ \\ 
        $\mathbb{S}^{250}$ & $0.5136$ & $3.2362 \pm 0.0009$ & $0.4589$ & $60.6258 \pm 0.0179$ & $\pmb{0.2748}$ & $\pmb{0.1525} \pm \pmb{ 0.0002 }$ \\ 
        $\mathbb{S}^{500}$ & $0.4253$ & $7.7012 \pm 0.0038$ & $0.3278$ & $706.9603 \pm 0.0293$ & $\pmb{0.1858}$ & $\pmb{0.7340} \pm \pmb{ 0.0006 }$ \\ 
        $\mathbb{S}^{1000}$ & $0.3790$ & $20.5850 \pm 0.0098$ & $-$ & $-$ & $\pmb{0.1626}$ & $\pmb{3.1229} \pm \pmb{ 0.0007 }$ \\ 
        \hline
        $\mathrm{E}\left( 2 \right)$ & $1.9331$ & $0.2723 \pm 0.0013$ & $\pmb{1.9330}$ & $0.1365 \pm 0.0083$ & $1.9331$ & $\pmb{0.0155} \pm \pmb{ 0.0003 }$ \\ 
        $\mathrm{E}\left( 3 \right)$ & $2.0243$ & $0.2852 \pm 0.0017$ & $\pmb{2.0233}$ & $0.1828 \pm 0.0074$ & $2.0235$ & $\pmb{0.0305} \pm \pmb{ 0.0012 }$ \\ 
        $\mathrm{E}\left( 5 \right)$ & $1.8262$ & $0.2996 \pm 0.0079$ & $\pmb{1.8047}$ & $0.6859 \pm 0.0047$ & $1.8054$ & $\pmb{0.1172} \pm \pmb{ 0.0050 }$ \\ 
        $\mathrm{E}\left( 10 \right)$ & $1.5193$ & $0.3094 \pm 0.0086$ & $\pmb{0.9472}$ & $57.6302 \pm 0.0107$ & $1.4788$ & $\pmb{0.0855} \pm \pmb{ 0.0036 }$ \\ 
        $\mathrm{E}\left( 20 \right)$ & $1.2028$ & $0.3379 \pm 0.0049$ & $\pmb{0.7603}$ & $48.7781 \pm 0.0029$ & $1.1366$ & $\pmb{0.0727} \pm \pmb{ 0.0033 }$ \\ 
        $\mathrm{E}\left( 50 \right)$ & $0.8382$ & $0.4596 \pm 0.0092$ & $0.6157$ & $153.9548 \pm 0.0165$ & $\pmb{0.4417}$ & $\pmb{0.0284} \pm \pmb{ 0.0001 }$ \\ 
        $\mathrm{E}\left( 100 \right)$ & $0.6175$ & $0.8219 \pm 0.0036$ & $0.5822$ & $3.9507 \pm 0.0014$ & $\pmb{0.3160}$ & $\pmb{0.0373} \pm \pmb{ 0.0001 }$ \\ 
        $\mathrm{E}\left( 250 \right)$ & $0.4216$ & $2.4940 \pm 0.0028$ & $0.3882$ & $48.5547 \pm 0.0120$ & $\pmb{0.2039}$ & $\pmb{0.1519} \pm \pmb{ 0.0003 }$ \\ 
        $\mathrm{E}\left( 500 \right)$ & $0.3454$ & $5.5967 \pm 0.0145$ & $-$ & $-$ & $\pmb{0.1434}$ & $\pmb{0.7326} \pm \pmb{ 0.0005 }$ \\ 
        $\mathrm{E}\left( 1000 \right)$ & $0.3056$ & $14.1006 \pm 0.0361$ & $-$ & $-$ & $\pmb{0.0996}$ & $\pmb{4.1512} \pm \pmb{ 0.0043 }$ \\ 
        \hline
        $\mathbb{T}^{2}$ & $9.8952$ & $0.2821 \pm 0.0029$ & $\pmb{9.8943}$ & $38.0816 \pm 0.0018$ & $9.8955$ & $\pmb{0.0531} \pm \pmb{ 0.0013 }$ \\ 
        \hline
        $\mathbb{H}^{2}$ & $1.2435$ & $0.3072 \pm 0.0003$ & $\pmb{1.2434}$ & $0.1172 \pm 0.0048$ & $1.2434$ & $\pmb{0.0082} \pm \pmb{ 0.0001 }$ \\ 
        \hline
        $\mathcal{P}\left( 2 \right)$ & $1.0704$ & $0.3571 \pm 0.0072$ & $\pmb{1.0567}$ & $0.1231 \pm 0.0051$ & $1.0567$ & $\pmb{0.0243} \pm \pmb{ 0.0002 }$ \\ 
        $\mathcal{P}\left( 3 \right)$ & $2.2400$ & $0.3924 \pm 0.0057$ & $2.3328$ & $57.1278 \pm 0.0125$ & $\pmb{2.2304}$ & $\pmb{0.0198} \pm \pmb{ 0.0001 }$ \\ 
        \hline
        $\text{Gaussian Distribution}$ & $2.5779$ & $0.1785 \pm 0.0003$ & $\pmb{2.5778}$ & $0.0617 \pm 0.0027$ & $2.5778$ & $\pmb{0.0130} \pm \pmb{ 0.0002 }$ \\ 
        $\text{Fr\'echet Distribution}$ & $1.1471$ & $0.1696 \pm 0.0004$ & $\pmb{1.1457}$ & $0.0442 \pm 0.0016$ & $1.1457$ & $\pmb{0.0063} \pm \pmb{ 0.0002 }$ \\ 
        $\text{Cauchy Distribution}$ & $1.6341$ & $0.0230 \pm 0.0002$ & $1.6340$ & $0.0810 \pm 0.0037$ & $\pmb{1.6340}$ & $\pmb{0.0135} \pm \pmb{ 0.0003 }$ \\ 
        $\text{Pareto Distribution}$ & $0.8235$ & $0.0357 \pm 0.0003$ & $0.8235$ & $0.0439 \pm 0.0014$ & $\pmb{0.8234}$ & $\pmb{0.0052} \pm \pmb{ 0.0003 }$ \\ 
        \bottomrule
    \end{tabular*}
    \caption{The table shows the length of the estimated geodesics for different methods on a GPU. The methods have been stopped if the norm of the gradient was less than $10^{-4}$ or if a maximum of $1,000$ iterations has been reached. When the computational time was longer than $24$ hours, the value is set to $-$.}
    \label{tab:riemannian_runtime1_gpu_T50}
\end{sidewaystable}

\paragraph{Riemannian manifolds}
\begin{wrapfigure}{r}{0.4\textwidth}
    \vspace{-2.em}
    \centering
    \includegraphics[width=0.4\textwidth]{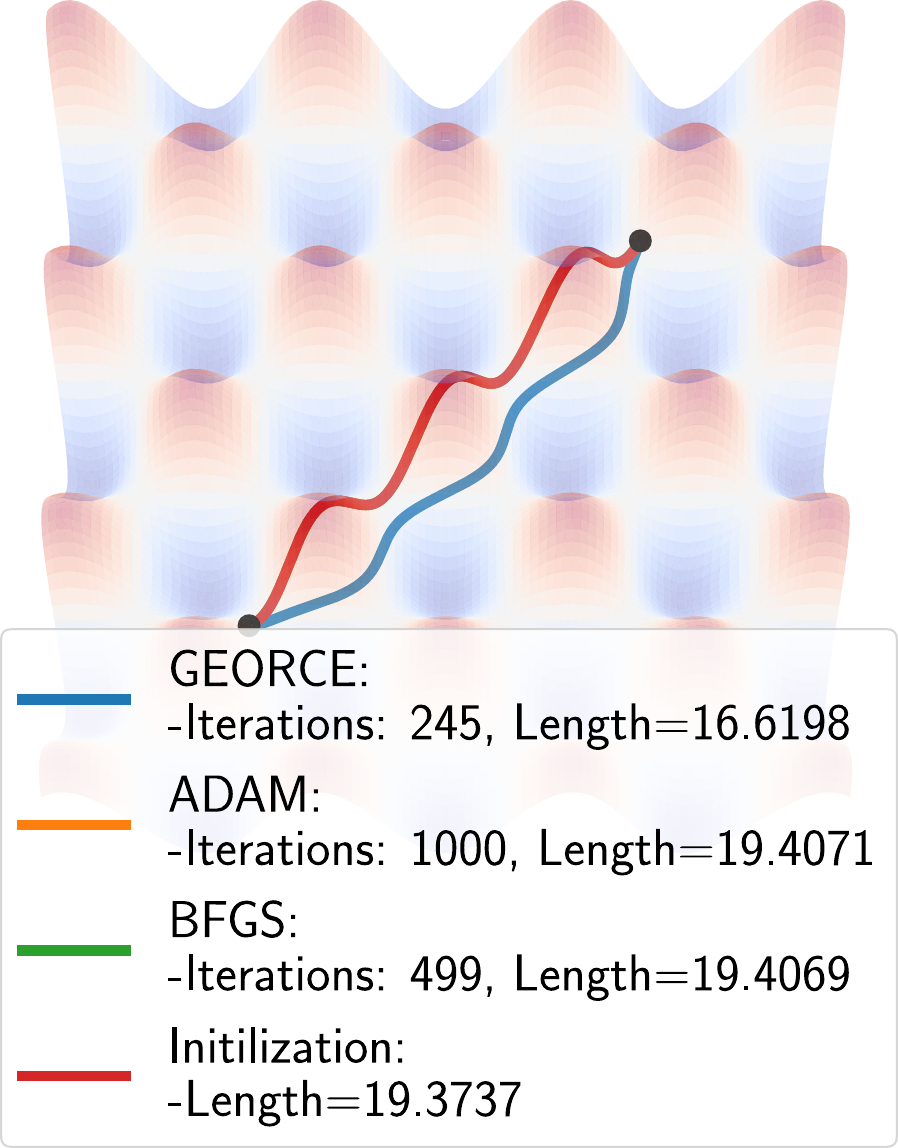}
    \caption{The geodesics for an ``egg-tray'' parameterized by $(x,y,2\cos x \cos y)$, where there are no unique geodesic connecting the points. The \textit{BFGS} estimate, \textit{ADAM} estimate and the initial curve overlap completely.}
    \label{fig:egg_tray}
    \vspace{-1.5em}
\end{wrapfigure}
To compare \textit{GEORCE} with other optimization algorithms, we show the runtime and length of the geodesics computed for different manifolds with varying dimensions in Table~\ref{tab:riemmannian_comparison_table} on a CPU. In Table~\ref{tab:riemannian_runtime1_gpu_T100} and Table~\ref{tab:riemannian_runtime1_gpu_T50}, we show the corresponding results on a GPU for compatible methods for $T=100$ and $T=50$, respectively. Here we compare \textit{GEORCE} with \textit{BFGS} \citep{broyden_bfgs, fletcher_bfgs, Goldfarb1970AFO, shanno_bfgs} and \textit{ADAM} \citep{kingma2017adam}. We terminate the algorithm if the $\ell^{2}$-norm of the gradient of the discretized energy functional in Eq.~\ref{eq:disc_const_energy} is less than $10^{-4}$, or if the number of iterations exceeds $1,000$. A full description of the manifolds and experimental details can be found in Appendix~\ref{ap:manifold_description} and Appendix~\ref{ap:experiments}, respectively. 

We show the estimated geodesics with $T=100$ grid points and summarize the results for the corresponding runtime and length of the computed geodesics for the $n$-sphere and $n$-ellipsoid in Fig.~\ref{fig:sphere_ellipsoid_runtime}. In general, we see that \textit{GEORCE} is considerably faster with a shorter length compared to alternative methods. 

To test \textit{GEORCE} on manifolds with multiple length minimizing geodesics, we consider the geodesic between $(-5,-5)$ and $(5,5)$ on an ``egg-tray'' parameterized by $(x,y,2\cos x \cos y)$ with $T=1,000$ grid points. We show the estimates in Fig.~\ref{fig:egg_tray}, where we see that \textit{GEORCE} obtains a significantly smaller length, where there is also a clear symmetry to solution to \textit{GEORCE} along the other side of the ``egg-tops''.

\begin{figure}[!t]
    \centering
    \includegraphics[width=1.0\textwidth]{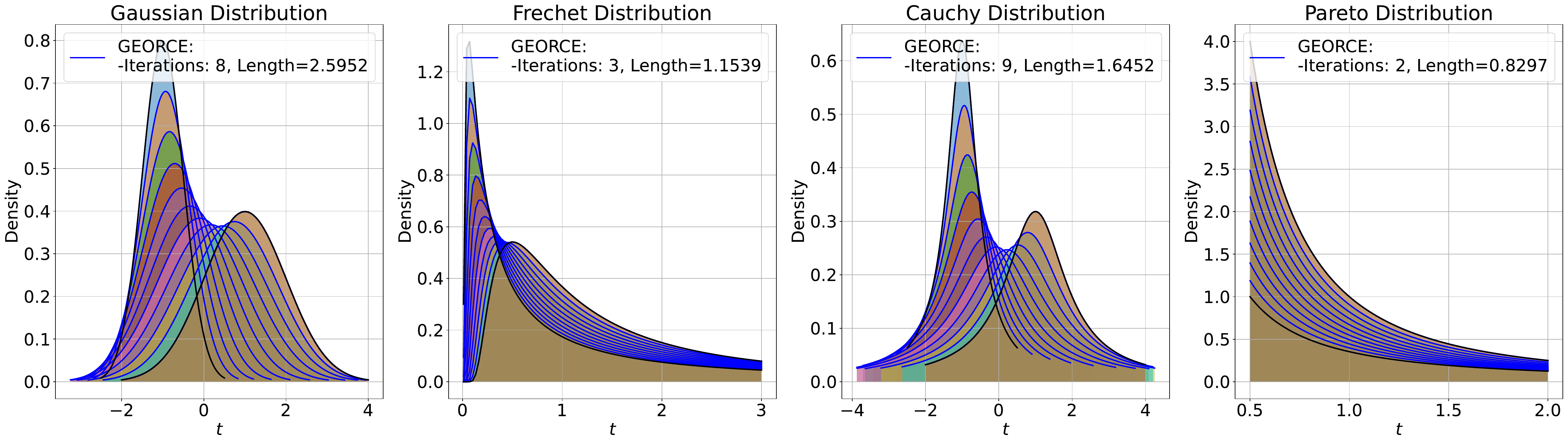}
    \caption{The connecting geodesic between four different distributions equipped with the Fisher-Rao metric using \textit{GEORCE}. From left to right the figures show the connecting geodesic (blue) between two fixed distributions (black) for the Gaussian distribution, Fréchet distribution, Cauchy distribution and the Pareto distribution.}
    \label{fig:information_geometry_geodesics}
    \vspace{-1.0em}
\end{figure}

To illustrate the application of \textit{GEORCE} to more abstract manifolds, we compute geodesics on the statistical manifolds, i.e., between distributions \citep{arvanitidis2022pullinginformationgeometry, Nielsen_2020}. Fig.~\ref{fig:information_geometry_geodesics} shows the estimated geodesic between a Gaussian distribution, a Fréchet distribution, a Cauchy distribution, and a Pareto distribution using \textit{GEORCE} all equipped with the Fisher-Rao metric in Eq.~\ref{eq:fisher_rao_metric}. The metric matrix function for the corresponding statistical manifold is described in \citep{miyamoto2024closedformexpressionsfisherraodistance}. We see that \textit{GEORCE} converges after only a few iterations. The computed geodesics using alternative methods are summarized in Table~\ref{tab:riemmannian_comparison_table}, where we see that \textit{GEORCE} obtains significantly shorter length much faster, also in high-dimension, compared to alternative methods.

\begin{wrapfigure}[19]{r}{0.50\textwidth}
    \vspace{-6.0mm}
    \centering
    \includegraphics[width=0.50\textwidth]{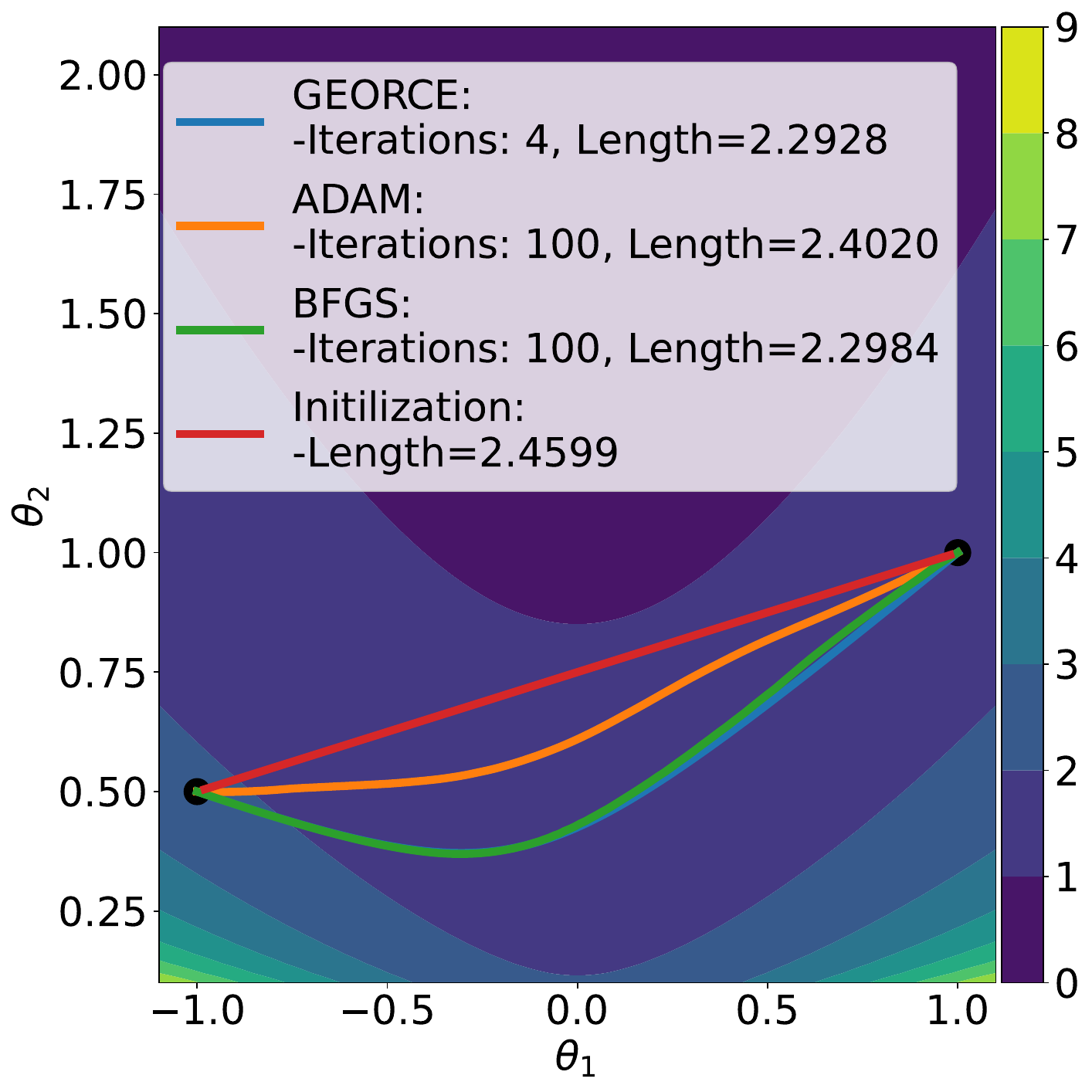}
    \vspace{-3.0mm}
    \caption{The geodesics computed under the Hessian of the log-partition function of a Gaussian distribution. The background color illustrates the log-partition function.}
    \label{fig:bregman_geometry}
\end{wrapfigure}

We further illustrate the application of \textit{GEORCE} to information geometry by considering potential functions encountered in Bregman geometry \citep{Nielsen_2020}. We consider the log-partition function for the Gaussian distribution, which is convex and smooth
\begin{equation*}
    A(\theta_{1},\theta_{2}) = \frac{\theta_{1}^{2}}{2\theta_{2}}+\frac{1}{2}\log\left(\frac{2\pi}{\theta_{2}}\right),
\end{equation*}
where $\theta_{1}=\frac{\mu}{\sigma^{2}}$ and $\theta_{2}=\frac{1}{\sigma^{2}}$ for a Gaussian distribution with parameters $(\mu,\sigma)$ such that the density takes the form of an exponential family $p(x) = h(x)\exp\left(\theta_{1} x - A(\theta_{1},\theta_{2})\right)$ for a suitable $h$. We consider the corresponding Hessian of the log-partition function, which defines a Riemannian metric. Fig.~\ref{fig:bregman_geometry} shows geodesics computed with both \textit{GEORCE} and alternative methods that use at most $100$ iterations. We see that \textit{GEORCE} uses less iterations and obtains the smallest length compared to the alternative methods.

\begin{figure}[t!]
    \centering
    \includegraphics[width=1.0\textwidth]{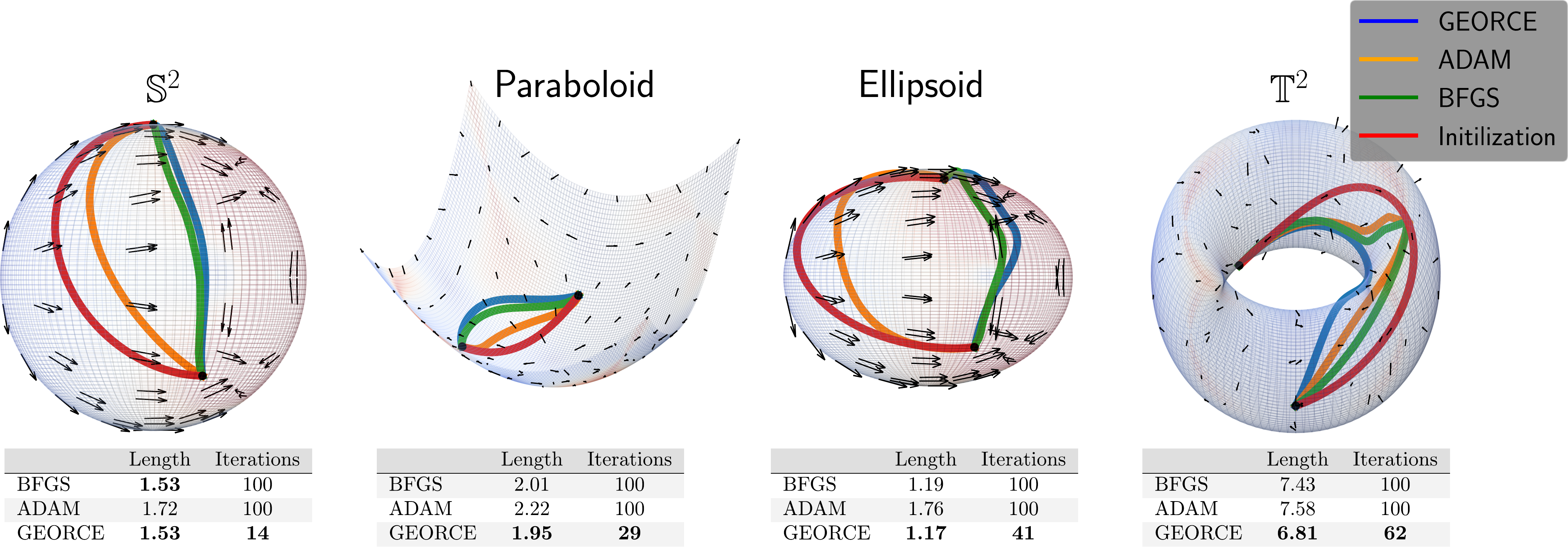}
    \caption{Comparison between \textit{GEORCE} and baseline methods for computing geodesics on four different Finsler manifolds. The Finsler metrics stem from a Riemannian background metric with the force field in Eq.~\ref{eq:generic_forcefield} on them. The length of the force field is displayed in color, while the arrows show the direction of the force field. The algorithms are terminated if the $\ell^{2}$-norm of the gradient of the discretized energy functional is less than $10^{-4}$, or if the number of iterations exceeds $100$. We see that \textit{GEORCE} uses less iterations and obtains the smallest length.}
    \label{fig:synthetic_finsler_geodesics}
    \vspace{-1.em}
\end{figure}

\begin{sidewaystable}
    \begin{tabular*}{\textheight}{@{\extracolsep\fill}lcccccc}
        \toprule%
        & \multicolumn{2}{c}{\textbf{BFGS (T=100)}} & \multicolumn{2}{c}{\textbf{ADAM  (T=100)}} & \multicolumn{2}{c}{\textbf{GEORCE  (T=100)}} \\\cmidrule{2-3}\cmidrule{4-5}\cmidrule{6-7}%
        Finsler Manifold & Length & Runtime & Length & Runtime & Length & Runtime \\
        \midrule
        $\mathbb{S}^{2}$ & $1.6525$ & $244.9129 \pm 0.0680$ & $1.6536$ & $0.4967 \pm 0.0007$ & $\pmb{1.6525}$ & $\pmb{0.0194} \pm \pmb{ 0.0009 }$ \\ 
        $\mathbb{S}^{3}$ & $1.7433$ & $277.8015 \pm 0.5819$ & $1.7456$ & $0.7639 \pm 0.0031$ & $\pmb{1.7432}$ & $\pmb{0.0362} \pm \pmb{ 0.0001 }$ \\ 
        $\mathbb{S}^{5}$ & $1.7032$ & $435.6987 \pm 4.6022$ & $1.7377$ & $4.6711 \pm 0.0047$ & $\pmb{1.7029}$ & $\pmb{0.4556} \pm \pmb{ 0.0021 }$ \\ 
        $\mathbb{S}^{10}$ & $1.4151$ & $812.2809 \pm 13.5968$ & $1.5174$ & $9.7974 \pm 0.0052$ & $\pmb{1.4109}$ & $\pmb{0.7473} \pm \pmb{ 0.0029 }$ \\ 
        $\mathbb{S}^{20}$ & $1.0951$ & $1446.0956 \pm 4.8038$ & $1.2018$ & $18.3489 \pm 0.0747$ & $\pmb{1.0760}$ & $\pmb{1.2856} \pm \pmb{ 0.0016 }$ \\ 
        $\mathbb{S}^{50}$ & $0.8216$ & $5992.7261 \pm 56.5973$ & $0.8128$ & $146.4556 \pm 0.9341$ & $\pmb{0.7285}$ & $\pmb{8.4758} \pm \pmb{ 0.0234 }$ \\ 
        $\mathbb{S}^{100}$ & $-$ & $-$ & $0.6187$ & $498.9705 \pm 0.4032$ & $\pmb{0.2570}$ & $\pmb{15.3786} \pm \pmb{ 0.0092 }$ \\ 
        $\mathbb{S}^{250}$ & $-$ & $-$ & $0.5195$ & $1397.6732 \pm 0.7793$ & $\pmb{0.1622}$ & $\pmb{132.5889} \pm \pmb{ 0.1106 }$ \\ 
        $\mathbb{S}^{500}$ & $-$ & $-$ & $0.4850$ & $1954.3231 \pm 14.1107$ & $\pmb{0.1404}$ & $\pmb{364.5326} \pm \pmb{ 1.8222 }$ \\ 
        $\mathbb{S}^{1000}$ & $-$ & $-$ & $0.3843$ & $5982.0884 \pm 204.1833$ & $\pmb{0.1140}$ & $\pmb{1338.9869} \pm \pmb{ 2.3954 }$ \\ 
        \hline
        $\mathrm{E}\left( 2 \right)$ & $\pmb{1.3616}$ & $230.0809 \pm 1.3967$ & $1.3621$ & $0.5310 \pm 0.0025$ & $1.3618$ & $\pmb{0.0130} \pm \pmb{ 0.0002 }$ \\ 
        $\mathrm{E}\left( 3 \right)$ & $\pmb{1.4001}$ & $416.6281 \pm 5.5354$ & $1.4087$ & $0.8183 \pm 0.0026$ & $1.4023$ & $\pmb{0.1081} \pm \pmb{ 0.0001 }$ \\ 
        $\mathrm{E}\left( 5 \right)$ & $\pmb{1.2951}$ & $638.7185 \pm 2.7818$ & $1.3265$ & $4.3802 \pm 0.0039$ & $1.2956$ & $\pmb{0.4554} \pm \pmb{ 0.0014 }$ \\ 
        $\mathrm{E}\left( 10 \right)$ & $\pmb{1.0733}$ & $741.9274 \pm 12.3382$ & $1.1229$ & $9.7899 \pm 0.0117$ & $1.0804$ & $\pmb{0.5815} \pm \pmb{ 0.0022 }$ \\ 
        $\mathrm{E}\left( 20 \right)$ & $0.8510$ & $1820.3137 \pm 4.0734$ & $0.8980$ & $19.4614 \pm 0.0039$ & $\pmb{0.8121}$ & $\pmb{1.0671} \pm \pmb{ 0.0044 }$ \\ 
        $\mathrm{E}\left( 50 \right)$ & $0.5981$ & $11191.9424 \pm 597.7002$ & $0.6325$ & $146.6409 \pm 0.1668$ & $\pmb{0.2842}$ & $\pmb{4.4377} \pm \pmb{ 0.0451 }$ \\ 
        $\mathrm{E}\left( 100 \right)$ & $-$ & $-$ & $0.4936$ & $491.8223 \pm 0.2346$ & $\pmb{0.2009}$ & $\pmb{14.5367} \pm \pmb{ 0.0092 }$ \\ 
        $\mathrm{E}\left( 250 \right)$ & $-$ & $-$ & $0.4159$ & $1475.5049 \pm 4.1430$ & $\pmb{0.1287}$ & $\pmb{125.4187} \pm \pmb{ 0.0888 }$ \\ 
        $\mathrm{E}\left( 500 \right)$ & $-$ & $-$ & $0.4142$ & $1984.7845 \pm 10.5689$ & $\pmb{0.0967}$ & $\pmb{373.2816} \pm \pmb{ 0.4227 }$ \\ 
        $\mathrm{E}\left( 1000 \right)$ & $-$ & $-$ & $0.3444$ & $5929.4438 \pm 167.1873$ & $\pmb{0.0799}$ & $\pmb{1211.8812} \pm \pmb{ 0.2511 }$ \\ 
        \hline
        $\mathbb{T}^{2}$ & $7.3890$ & $449.4798 \pm 10.3282$ & $7.3585$ & $0.3294 \pm 0.0018$ & $\pmb{6.8067}$ & $\pmb{0.1126} \pm \pmb{ 0.0002 }$ \\ 
        \hline
        $\mathbb{H}^{2}$ & $1.3340$ & $384.3149 \pm 0.3044$ & $1.3360$ & $0.2744 \pm 0.0024$ & $\pmb{1.3337}$ & $\pmb{0.0094} \pm \pmb{ 0.0001 }$ \\ 
        \hline
        $\mathcal{P}\left( 2 \right)$ & $\pmb{0.7291}$ & $414.8355 \pm 3.0837$ & $0.7338$ & $0.6368 \pm 0.0026$ & $0.7292$ & $\pmb{0.0222} \pm \pmb{ 0.0001 }$ \\ 
        $\mathcal{P}\left( 3 \right)$ & $\pmb{1.2856}$ & $536.6575 \pm 0.9323$ & $1.2980$ & $5.2824 \pm 0.0540$ & $1.2856$ & $\pmb{0.2094} \pm \pmb{ 0.0004 }$ \\ 
        \hline
        Gaussian Distribution & $1.6878$ & $163.4399 \pm 0.1859$ & $1.6936$ & $0.1145 \pm 0.0001$ & $\pmb{1.6878}$ & $\pmb{0.0047} \pm \pmb{ 0.0001 }$ \\ 
        Fr\'echet Distribution & $0.5408$ & $123.8418 \pm 0.0499$ & $0.5406$ & $0.0585 \pm 0.0001$ & $\pmb{0.5405}$ & $\pmb{0.0021} \pm \pmb{ 0.0001 }$ \\ 
        Cauchy Distribution & $1.1134$ & $182.3980 \pm 2.2559$ & $\pmb{1.1133}$ & $0.1105 \pm 0.0007$ & $1.1133$ & $\pmb{0.0067} \pm \pmb{ 0.0002 }$ \\ 
        Pareto Distribution & $0.3884$ & $84.8186 \pm 0.0832$ & $0.3885$ & $0.0235 \pm 0.0007$ & $\pmb{0.3883}$ & $\pmb{0.0020} \pm \pmb{ 0.0001 }$ \\ 
        \bottomrule
    \end{tabular*}
    \caption{The length of the estimated geodesics for the \textit{BFGS}-algorithm, \textit{ADAM}-gradient descent and \textit{GEORCE} on a CPU. The algorithms have been stopped if the $\ell^{2}$-norm of the gradient was less than $10^{-4}$ or if a maximum of $1,000$ iterations have been reached. $\mathrm{E}(n)$ denotes an Ellipsoid of dimension $n$ with half axes of $n$ equally spaced points between $0.5$ and $1.0$, while $\mathcal{P}(n)$ denotes the space of $n \times n$ symmetric positive definite matrices. When the computational time was longer than $24$ hours, the value is set to $-$.}
    \label{tab:finsler_comparison_table}
\end{sidewaystable}

\begin{sidewaystable}
    \begin{tabular*}{\textheight}{@{\extracolsep\fill}lcccccc}
        \toprule%
        & \multicolumn{2}{c}{\textbf{ADAM (T=100)}} & \multicolumn{2}{c}{\textbf{Sparse Newton  (T=100)}} & \multicolumn{2}{c}{\textbf{GEORCE  (T=100)}} \\\cmidrule{2-3}\cmidrule{4-5}\cmidrule{6-7}%
        Finsler Manifold & Length & Runtime & Length & Runtime & Length & Runtime \\
        \midrule
        $\mathbb{S}^{2}$ & $1.6530$ & $0.7706 \pm 0.0018$ & $\pmb{1.6525}$ & $0.2512 \pm 0.0089$ & $1.6525$ & $\pmb{0.0372} \pm \pmb{ 0.0004 }$ \\ 
        $\mathbb{S}^{3}$ & $1.7451$ & $0.7362 \pm 0.0036$ & $-$ & $-$ & $\pmb{1.7432}$ & $\pmb{0.0431} \pm \pmb{ 0.0005 }$ \\ 
        $\mathbb{S}^{5}$ & $1.7371$ & $0.9292 \pm 0.0038$ & $-$ & $-$ & $\pmb{1.7029}$ & $\pmb{0.1036} \pm \pmb{ 0.0048 }$ \\ 
        $\mathbb{S}^{10}$ & $1.5174$ & $1.1358 \pm 0.0018$ & $1.4783$ & $208.1472 \pm 0.0176$ & $\pmb{1.4109}$ & $\pmb{0.1000} \pm \pmb{ 0.0033 }$ \\ 
        $\mathbb{S}^{20}$ & $1.2009$ & $1.5876 \pm 0.0028$ & $-$ & $-$ & $\pmb{1.0760}$ & $\pmb{0.1509} \pm \pmb{ 0.0058 }$ \\ 
        $\mathbb{S}^{50}$ & $0.8130$ & $2.0381 \pm 0.0026$ & $-$ & $-$ & $\pmb{0.7283}$ & $\pmb{0.1373} \pm \pmb{ 0.0032 }$ \\ 
        $\mathbb{S}^{100}$ & $0.6146$ & $2.7712 \pm 0.0032$ & $0.6210$ & $972.0195 \pm 0.0307$ & $\pmb{0.2553}$ & $\pmb{0.1106} \pm \pmb{ 0.0028 }$ \\ 
        $\mathbb{S}^{250}$ & $0.5196$ & $3.9803 \pm 0.0019$ & $-$ & $-$ & $\pmb{0.1636}$ & $\pmb{0.5069} \pm \pmb{ 0.0003 }$ \\ 
        $\mathbb{S}^{500}$ & $0.4849$ & $6.6164 \pm 0.0023$ & $-$ & $-$ & $\pmb{0.1417}$ & $\pmb{1.6591} \pm \pmb{ 0.0004 }$ \\ 
        $\mathbb{S}^{1000}$ & $0.3835$ & $25.4638 \pm 0.1448$ & $-$ & $-$ & $\pmb{0.1134}$ & $\pmb{7.9991} \pm \pmb{ 0.0074 }$ \\ 
        \hline
        $\mathrm{E}\left( 2 \right)$ & $1.3634$ & $0.7758 \pm 0.0036$ & $\pmb{1.3616}$ & $0.3735 \pm 0.0010$ & $1.3618$ & $\pmb{0.0282} \pm \pmb{ 0.0009 }$ \\ 
        $\mathrm{E}\left( 3 \right)$ & $1.4085$ & $0.7541 \pm 0.0032$ & $-$ & $-$ & $\pmb{1.4023}$ & $\pmb{0.1273} \pm \pmb{ 0.0047 }$ \\ 
        $\mathrm{E}\left( 5 \right)$ & $1.3305$ & $1.0252 \pm 0.0017$ & $1.5359$ & $170.9661 \pm 0.0296$ & $\pmb{1.2956}$ & $\pmb{0.1135} \pm \pmb{ 0.0039 }$ \\ 
        $\mathrm{E}\left( 10 \right)$ & $1.1235$ & $1.1254 \pm 0.0045$ & $-$ & $-$ & $\pmb{1.0804}$ & $\pmb{0.0825} \pm \pmb{ 0.0031 }$ \\ 
        $\mathrm{E}\left( 20 \right)$ & $0.8987$ & $1.6039 \pm 0.0035$ & $-$ & $-$ & $\pmb{0.8121}$ & $\pmb{0.1282} \pm \pmb{ 0.0054 }$ \\ 
        $\mathrm{E}\left( 50 \right)$ & $0.6332$ & $2.0424 \pm 0.0034$ & $-$ & $-$ & $\pmb{0.2845}$ & $\pmb{0.0770} \pm \pmb{ 0.0004 }$ \\ 
        $\mathrm{E}\left( 100 \right)$ & $0.4953$ & $2.6704 \pm 0.0031$ & $0.4577$ & $470.9805 \pm 0.1448$ & $\pmb{0.2018}$ & $\pmb{0.1127} \pm \pmb{ 0.0023 }$ \\ 
        $\mathrm{E}\left( 250 \right)$ & $0.4179$ & $4.0843 \pm 0.0009$ & $-$ & $-$ & $\pmb{0.1287}$ & $\pmb{0.5239} \pm \pmb{ 0.0047 }$ \\ 
        $\mathrm{E}\left( 500 \right)$ & $0.4139$ & $6.9391 \pm 0.0073$ & $0.2580$ & $1435.3881 \pm 0.0832$ & $\pmb{0.0968}$ & $\pmb{1.7109} \pm \pmb{ 0.0010 }$ \\ 
        $\mathrm{E}\left( 1000 \right)$ & $0.3442$ & $25.9660 \pm 0.1038$ & $-$ & $-$ & $\pmb{0.0798}$ & $\pmb{7.9988} \pm \pmb{ 0.0079 }$ \\ 
        \hline
        $\mathbb{T}^{2}$ & $7.4054$ & $0.8017 \pm 0.0045$ & $9.1903$ & $152.0036 \pm 0.0158$ & $\pmb{6.8067}$ & $\pmb{0.3145} \pm \pmb{ 0.0050 }$ \\ 
        \hline
        $\mathbb{H}^{2}$ & $1.3341$ & $0.7949 \pm 0.0072$ & $1.3338$ & $0.4562 \pm 0.0004$ & $\pmb{1.3337}$ & $\pmb{0.0256} \pm \pmb{ 0.0006 }$ \\ 
        \hline
        $\mathcal{P}\left( 2 \right)$ & $0.7334$ & $0.9344 \pm 0.0013$ & $\pmb{0.7291}$ & $0.4309 \pm 0.0014$ & $0.7292$ & $\pmb{0.0286} \pm \pmb{ 0.0005 }$ \\ 
        $\mathcal{P}\left( 3 \right)$ & $1.2987$ & $1.0812 \pm 0.0006$ & $1.2856$ & $0.6885 \pm 0.0009$ & $\pmb{1.2856}$ & $\pmb{0.0365} \pm \pmb{ 0.0005 }$ \\ 
        \hline
        $\text{Gaussian Distribution}$ & $1.6879$ & $0.4962 \pm 0.0034$ & $1.6878$ & $0.4148 \pm 0.0055$ & $\pmb{1.6878}$ & $\pmb{0.0255} \pm \pmb{ 0.0004 }$ \\ 
        $\text{Fr\'echet Distribution}$ & $0.5406$ & $0.2307 \pm 0.0069$ & $\pmb{0.5405}$ & $0.2060 \pm 0.0042$ & $0.5405$ & $\pmb{0.0088} \pm \pmb{ 0.0003 }$ \\ 
        $\text{Cauchy Distribution}$ & $1.1140$ & $0.4994 \pm 0.0017$ & $\pmb{1.1133}$ & $0.5167 \pm 0.0042$ & $1.1133$ & $\pmb{0.0367} \pm \pmb{ 0.0017 }$ \\ 
        $\text{Pareto Distribution}$ & $0.3885$ & $0.1014 \pm 0.0038$ & $0.3884$ & $0.1042 \pm 0.0031$ & $\pmb{0.3883}$ & $\pmb{0.0096} \pm \pmb{ 0.0004 }$ \\ 
        \bottomrule
    \end{tabular*}
    \caption{The table shows the length of the estimated geodesics for different methods on a GPU. The methods have been stopped if the norm of the gradient was less than $10^{-4}$ or if a maximum of $1,000$ iterations has been reached. When the computational time was longer than $24$ hours, the value is set to $-$.}
    \label{tab:finsler_runtime1_gpu_T100}
\end{sidewaystable}

\begin{sidewaystable}
    \begin{tabular*}{\textheight}{@{\extracolsep\fill}lcccccc}
        \toprule%
        & \multicolumn{2}{c}{\textbf{ADAM (T=50)}} & \multicolumn{2}{c}{\textbf{Sparse Newton  (T=50)}} & \multicolumn{2}{c}{\textbf{GEORCE  (T=50)}} \\\cmidrule{2-3}\cmidrule{4-5}\cmidrule{6-7}%
        Finsler Manifold & Length & Runtime & Length & Runtime & Length & Runtime \\
        \midrule
        $\mathbb{S}^{2}$ & $\pmb{1.6383}$ & $0.7833 \pm 0.0056$ & $1.6385$ & $0.2064 \pm 0.0073$ & $1.6385$ & $\pmb{0.0457} \pm \pmb{ 0.0004 }$ \\ 
        $\mathbb{S}^{3}$ & $1.7300$ & $0.7587 \pm 0.0049$ & $-$ & $-$ & $\pmb{1.7297}$ & $\pmb{0.0499} \pm \pmb{ 0.0003 }$ \\ 
        $\mathbb{S}^{5}$ & $1.6875$ & $0.9909 \pm 0.0022$ & $-$ & $-$ & $\pmb{1.6868}$ & $\pmb{0.1432} \pm \pmb{ 0.0050 }$ \\ 
        $\mathbb{S}^{10}$ & $1.4013$ & $1.1472 \pm 0.0024$ & $-$ & $-$ & $\pmb{1.3908}$ & $\pmb{0.1747} \pm \pmb{ 0.0067 }$ \\ 
        $\mathbb{S}^{20}$ & $1.0829$ & $1.5583 \pm 0.0016$ & $-$ & $-$ & $\pmb{1.0560}$ & $\pmb{0.2663} \pm \pmb{ 0.0005 }$ \\ 
        $\mathbb{S}^{50}$ & $0.7427$ & $1.4442 \pm 0.0038$ & $-$ & $-$ & $\pmb{0.3508}$ & $\pmb{0.0594} \pm \pmb{ 0.0003 }$ \\ 
        $\mathbb{S}^{100}$ & $0.5511$ & $2.5558 \pm 0.0020$ & $-$ & $-$ & $\pmb{0.2626}$ & $\pmb{0.0850} \pm \pmb{ 0.0024 }$ \\ 
        $\mathbb{S}^{250}$ & $0.4094$ & $3.8872 \pm 0.0020$ & $-$ & $-$ & $\pmb{0.1686}$ & $\pmb{0.2660} \pm \pmb{ 0.0042 }$ \\ 
        $\mathbb{S}^{500}$ & $0.3531$ & $7.5214 \pm 0.0022$ & $0.2861$ & $1388.8278 \pm 0.0965$ & $\pmb{0.1193}$ & $\pmb{1.1879} \pm \pmb{ 0.0003 }$ \\ 
        $\mathbb{S}^{1000}$ & $0.3269$ & $13.1209 \pm 0.0379$ & $-$ & $-$ & $\pmb{0.0828}$ & $\pmb{6.0238} \pm \pmb{ 0.0039 }$ \\ 
        \hline
        $\mathrm{E}\left( 2 \right)$ & $1.3460$ & $0.1327 \pm 0.0060$ & $\pmb{1.3460}$ & $0.2847 \pm 0.0043$ & $1.3461$ & $\pmb{0.0332} \pm \pmb{ 0.0011 }$ \\ 
        $\mathrm{E}\left( 3 \right)$ & $1.3870$ & $0.7621 \pm 0.0034$ & $-$ & $-$ & $\pmb{1.3826}$ & $\pmb{0.2156} \pm \pmb{ 0.0006 }$ \\ 
        $\mathrm{E}\left( 5 \right)$ & $1.2780$ & $0.9718 \pm 0.0022$ & $-$ & $-$ & $\pmb{1.2757}$ & $\pmb{0.1654} \pm \pmb{ 0.0050 }$ \\ 
        $\mathrm{E}\left( 10 \right)$ & $1.0665$ & $1.1072 \pm 0.0025$ & $-$ & $-$ & $\pmb{1.0126}$ & $\pmb{0.2215} \pm \pmb{ 0.0062 }$ \\ 
        $\mathrm{E}\left( 20 \right)$ & $0.8405$ & $1.5702 \pm 0.0031$ & $-$ & $-$ & $\pmb{0.7946}$ & $\pmb{0.1937} \pm \pmb{ 0.0061 }$ \\ 
        $\mathrm{E}\left( 50 \right)$ & $\pmb{0.5797}$ & $1.4336 \pm 0.0029$ & $-$ & $-$ & $-$ & $-$ \\ 
        $\mathrm{E}\left( 100 \right)$ & $0.4334$ & $2.0750 \pm 0.0026$ & $-$ & $-$ & $\pmb{0.2046}$ & $\pmb{0.0830} \pm \pmb{ 0.0004 }$ \\ 
        $\mathrm{E}\left( 250 \right)$ & $0.3234$ & $2.8610 \pm 0.0021$ & $-$ & $-$ & $\pmb{0.1365}$ & $\pmb{0.2644} \pm \pmb{ 0.0029 }$ \\ 
        $\mathrm{E}\left( 500 \right)$ & $0.2862$ & $5.1735 \pm 0.0122$ & $0.2428$ & $905.0696 \pm 0.0365$ & $\pmb{0.0940}$ & $\pmb{0.9444} \pm \pmb{ 0.0011 }$ \\ 
        $\mathrm{E}\left( 1000 \right)$ & $0.2723$ & $10.4951 \pm 0.0347$ & $-$ & $-$ & $\pmb{0.0893}$ & $\pmb{4.4688} \pm \pmb{ 0.0060 }$ \\ 
        \hline
         & $\pmb{7.2683}$ & $\pmb{0.7530} \pm \pmb{ 0.0015 }$ & $9.0475$ & $78.7132 \pm 0.0222$ & $9.0475$ & $13.7393 \pm 0.0093$ \\ 
        \hline
        $\mathbb{H}^{2}$ & $1.3328$ & $0.7454 \pm 0.0052$ & $1.3316$ & $0.2303 \pm 0.0058$ & $\pmb{1.3316}$ & $\pmb{0.0317} \pm \pmb{ 0.0003 }$ \\ 
        \hline
        $\mathcal{P}\left( 2 \right)$ & $0.7249$ & $0.5927 \pm 0.0008$ & $\pmb{0.7249}$ & $0.2189 \pm 0.0051$ & $0.7249$ & $\pmb{0.0331} \pm \pmb{ 0.0011 }$ \\ 
        $\mathcal{P}\left( 3 \right)$ & $1.2710$ & $1.0184 \pm 0.0008$ & $1.2707$ & $0.4325 \pm 0.0020$ & $\pmb{1.2707}$ & $\pmb{0.0434} \pm \pmb{ 0.0006 }$ \\ 
        \hline
        $\text{Gaussian Distribution}$ & $\pmb{1.6796}$ & $0.4337 \pm 0.0111$ & $1.6806$ & $0.2081 \pm 0.0047$ & $1.6806$ & $\pmb{0.0282} \pm \pmb{ 0.0017 }$ \\ 
        $\text{Fr\'echet Distribution}$ & $0.5370$ & $0.0829 \pm 0.0032$ & $0.5370$ & $0.1059 \pm 0.0042$ & $\pmb{0.5370}$ & $\pmb{0.0111} \pm \pmb{ 0.0003 }$ \\ 
        $\text{Cauchy Distribution}$ & $1.1085$ & $0.0626 \pm 0.0010$ & $\pmb{1.1084}$ & $0.2270 \pm 0.0081$ & $1.1084$ & $\pmb{0.0410} \pm \pmb{ 0.0015 }$ \\ 
        $\text{Pareto Distribution}$ & $0.3857$ & $0.0752 \pm 0.0025$ & $\pmb{0.3856}$ & $0.1054 \pm 0.0028$ & $0.3857$ & $\pmb{0.0082} \pm \pmb{ 0.0003 }$ \\ 
        \bottomrule
    \end{tabular*}
    \caption{The table shows the length of the estimated geodesics for different methods on a GPU. The methods have been stopped if the norm of the gradient was less than $10^{-4}$ or if a maximum of $1,000$ iterations has been reached. When the computational time was longer than $24$ hours, the value is set to $-$.}
    \label{tab:finsler_runtime1_gpu_T50}
\end{sidewaystable}

\paragraph{Finsler manifolds} To illustrate our method for different Finsler manifolds and dimensions, we consider a microswimmer on a Riemannian manifold $m$ that moves with constant self-propulsion velocity $v$ with Riemannian length, $v_{0}:=||v||_{g}$ \citep{Piro_2021} and we assume that this movement is affected by a generic force field, $f$, of the form
\begin{equation}
    f(x) = \frac{\sin x \odot \cos x}{(\cos x)^{\top} G(x) \cos x},
    \label{eq:generic_forcefield}
\end{equation}
where $\cos x$ and $\sin x$ denote the element-wise function for each entry of the local coordinates $x$, while $G$ denotes the corresponding Riemannian background metric such that the force field has $G$-length less than or equal to 1. If $||f(x)||_{g} < v_{0}$, then the optimal travel time of the microswimmer corresponds to a geodesic for a Finsler manifold of the form \citep{Piro_2021}
\begin{equation*}
    F = \sqrt{a_{ij}\dot{r}^{i}\dot{r}^{j}}+b_{i}\dot{r}^{i},
\end{equation*}
where $a_{ij} = g_{ij} \lambda + f_{i}f_{j}\lambda^{2}$, $b_{i} = -f_{i}\lambda$, $b_{i} = -f_{i}\lambda$, $f_{i} = g_{ij}f^{j}$ and $\lambda^{-1} = v_{0}^{2}-g_{ij}f^{i}f^{j}$ \citep{Piro_2021}, while $r(t)$ denotes the position of the microswimmer at time $t$. We refer to \cite{Piro_2021} and Appendix~\ref{ap:manifold_description} for details. We consider the same background metrics in Table~\ref{tab:riemmannian_comparison_table} all equipped with the generic force field in Eq.~\ref{eq:generic_forcefield} and compare \textit{GEORCE} to \textit{BFGS} \citep{broyden_bfgs, fletcher_bfgs, Goldfarb1970AFO, shanno_bfgs} and \textit{ADAM} \citep{kingma2017adam}. Fig.~\ref{fig:synthetic_finsler_geodesics} shows the estimated geodesics for four different Riemannian manifolds equipped with the generic force field in Eq.~\ref{eq:generic_forcefield}. We show the estimated geodesics for $T=100$ grid points in Table~\ref{tab:finsler_comparison_table}. In Table~\ref{tab:finsler_runtime1_gpu_T100} and Table~\ref{tab:finsler_runtime1_gpu_T50} we show the corresponding results on a GPU for compatible methods for $T=100$ and $T=50$. 

\paragraph{Applications to generative models} To illustrate the application of \textit{GEORCE} to learned manifolds, we compute geodesics for generative models and compare them to similar methods. A variational autoencoder (VAE) is a generative model that learns the data distribution, $\mathcal{X} \subset \mathbb{R}^{D}$, and a latent representation of data, $\mathcal{Z} \subset \mathbb{R}^{d}$, with $d < D$ by maximizing a lower bound of the log-likelihood of the data with respect to a variational family of distributions \citep{kingma2022autoencodingvariationalbayes}

%
%
%
%

For latent space generative models, the \emph{decoder} can be seen as a smooth $d$ dimensional immersion, $f: \mathcal{Z} \rightarrow \mathcal{X}$, which defines a Riemannian geometry with the pull-back metric $G(z)=J_{f}^{\top}(z)J_{f}(z)$, where $J_{f}$ denotes the Jacobian of $z$ \citep{arvanitidis2021latentspaceodditycurvature, shao2017riemannian}. 

We train a VAE for the CelebA dataset \citep{liu2015faceattributes}, which consists of $202{,}599$ facial images of size $64 \times 64 \times 3$, where the architecture is described in Appendix~\ref{ap:vae_architecture}. Fig.~\ref{fig:celeba_32} illustrates the geodesics constructed between two images reconstructed using the variational autoencoder. We see that \textit{GEORCE} obtains a smaller length than comparable methods for $T=100$ grid points. We refer to Appendix~\ref{ap:vae_experiments} for runtime results for the VAE, where we see that \textit{GEORCE} is significantly faster than alternative methods. Note that the tolerance for convergence is $10^{-3}$, since it was observed that it was not possible for any method to obtain an $\ell^{2}$-norm of the gradient less than $10^{-4}$. This is most likely due to potential instabilities for the learned manifold. We summarize the results in Table~\ref{tab:vae_table} on a CPU. 

\begin{figure}[t!]
    \centering
    \includegraphics[width=1.0\textwidth]{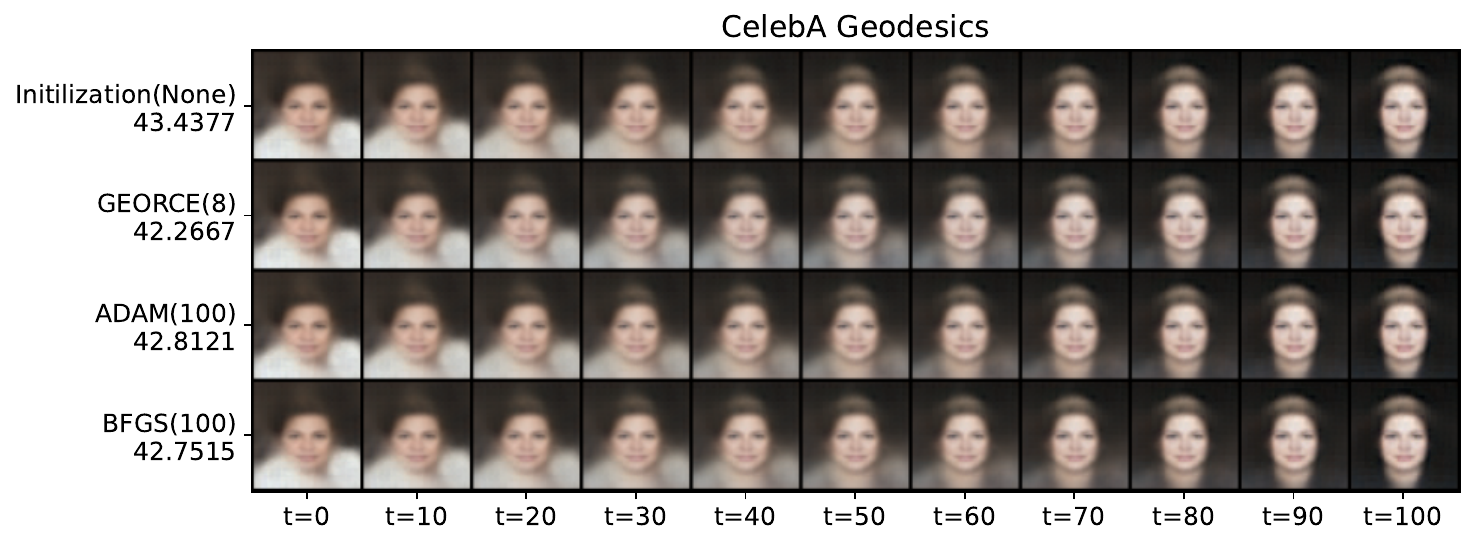}
    \caption{Geodesics for a variational autoencoder \citep{kingma2022autoencodingvariationalbayes} trained on the CelebA dataset \citep{liu2015faceattributes} using \textit{GEORCE} and alternative methods. All algorithms are terminated if the $\ell^{2}$-norm of the gradient of the discretized energy functional in Eq.~\ref{eq:disc_const_energy} is less than $10^{-3}$, or if the number of iterations exceeds $100$. The algorithms are written to the left with number of iterations in parenthesis and length to the right. The first-axis shows the grid number for the image sequence. The details of the manifolds and experiments can be found Appendix~\ref{ap:vae_architecture}}
    \label{fig:celeba_32}
    \vspace{-1.0em}
\end{figure}

\begin{sidewaystable}
    \begin{tabular*}{\textheight}{@{\extracolsep\fill}lcccccc}
        \toprule%
        & \multicolumn{2}{c}{\textbf{BFGS (T=100)}} & \multicolumn{2}{c}{\textbf{ADAM  (T=100)}} & \multicolumn{2}{c}{\textbf{GEORCE  (T=100)}} \\\cmidrule{2-3}\cmidrule{4-5}\cmidrule{6-7}%
        Riemannian Manifold & Length & Runtime & Length & Runtime & Length & Runtime \\
        \midrule
        VAE MNIST & $10.3075$ & $1407.5743 \pm 6.4430$ & $10.3312$ & $645.5201 \pm 2.6648$ & $\pmb{10.3074}$ & $\pmb{67.6149} \pm \pmb{ 0.0304 }$ \\ 
        VAE SVHN & $-$ & $-$ & $19.9855$ & $13325.2285 \pm 123.4978$ & $\pmb{19.9346}$ & $\pmb{882.4120} \pm \pmb{ 1.8839 }$ \\  
        VAE CelebA & $-$ & $-$ & $-$ & $-$ & $\pmb{42.2666}$ & $\pmb{2222.5259} \pm \pmb{ 8.2988 }$ \\ 
        \bottomrule
    \end{tabular*}
    \caption{The table shows the length of the estimated geodesics for the \textit{BFGS}-algorithm, \textit{ADAM} and \textit{GEORCE} on a CPU. The methods were terminated if the $\ell^{2}$-norm of the gradient was less than $10^{-3}$ or $1,000$ iterations have been reached. When the computational time was longer than $24$ hour, the value is set to $-$.}
    \label{tab:vae_table}
\end{sidewaystable}
\clearpage

\section{Conclusion}
In this paper, we have introduced the algorithm \textit{GEORCE} for estimating geodesics between any given points on Riemannian and Finsler manifolds by formulating the discretized energy functional as a control problem. We have proved that \textit{GEORCE} has global convergence and asymptotic quadratic local convergence. \textit{GEORCE} scales cubicly in the manifold dimension and only linearly in the number of grid points. Empirically, we have shown that \textit{GEORCE} exhibits superior convergence in terms of runtime and accuracy compared to other geodesic optimization algorithms for Riemannian and Finsler manifolds across both dimensions and manifolds.

\textbf{Limitations:} \textit{GEORCE} exhibits both global convergence and local asymptotic quadratic convergence requiring few iterations. However, \textit{GEORCE} requires evaluating the inverse matrix of the metric matrix function in each iteration, which can be numerically expensive. In addition, \textit{GEORCE} uses line search to determine the step size in each iteration, and if it is numerically expensive to evaluate the discretized energy functional, then this can potentially increase the runtime of \textit{GEORCE}. Further research should therefore investigate efficient approximations of the inverse metric matrix function as well as possible adaptive updating schemes for the step size to increase computational efficiency for high-dimensional problems.

\backmatter

\bmhead{Supplementary information}

The supplementary information in the paper consists of the appendix for proofs, details on experiments and manifolds as well as additional experiments.

\bmhead{Acknowledgements}
Many thanks to Prof.\@ Steen Markvorsen and Dr.\@ Andreas A. Bock for comments and corrections in the process of writing this paper. Especially, a thanks to Prof.\@ Steen Markvorsen for hinting at the possible extension to Finsler geometry. 
SH is supported by a research grant (42062) from VILLUM FONDEN. This project received funding from the European Research Council (ERC) under the European Union’s Horizon research and innovation programme (grant agreement 101125993). The work was partly funded by the Novo Nordisk Foundation through the Center for Basic Machine Learning Research in Life Science (NNF20OC0062606).

\section*{Declarations}

\begin{itemize}
\item Funding: Research grant 42062 from VILLUM FONDEN;  funding from the European Research Council (ERC) under the European Union’s Horizon research and innovation programme (grant agreement 101125993); funding from the Novo Nordisk Foundation through the Center for Basic Machine Learning Research in Life Science (NNF20OC0062606).
\item Conflict of interest/Competing interests: NA
\item Ethics approval and consent to participate: NA
\item Consent for publication: Currently under review.
\item Data availability: All data used are described in the appendix.
\item Materials availability: NA
\item Code availability: The code can be found at \url{https://github.com/FrederikMR/georce}
\item Author contribution: FMR derived the algorithm and implemented the method as well as benchmarks. SH supervised the project with comments and corrections. The paper was written jointly by SH and FMR. All authors reviewed the manuscript.
\end{itemize}
\clearpage

\bibliography{sn-article}

\clearpage
\begin{appendices}
    \section{Proofs and derivations} \label{ap:proofs}
    \subsection{Proof of update scheme for Finslerian geometry}
    \label{ap:finsler_update}
    Consider the following control problem in the Finslerian case.
\begin{equation} 
    \begin{split}
        \min_{x_{0:T}} \quad &\sum_{t=0}^{T-1} u_{t}^{\top}G\left(x_{t}, u_{t}\right)u_{t},\\
        \text{s.t.} \quad &x_{t+1}=x_{t}+u_{t}, \quad t=0,\dots,T-1, \\
        &x_{0}=a,x_{T}=b,
    \end{split}
\end{equation}
By similar computation as in Proposition.~\ref{prop:riemann_cond} with $G$ is now dependent on both $x_{t}$ and $u_{t}$, we get the following equation system to be solved in iteration $i+1$ derived from Pontryagin's maximum principle by fixing variables
\begin{equation} \label{eq:finsler_energy_zero_point_problem}
    \begin{split}
    &\nu_{t}+\mu_{t} = \mu_{t-1}, \quad t=1,\dots,T-1, \\
    &2G_{t}u_{t}+\zeta_{t}+\mu_{t} = 0, \quad t=0,\dots,T-1, \\
    &\sum_{t=0}^{T-1}u_{t}=b-a, \\
    &\nu_{t} := \restr{\nabla_{y}\left(u_{t}^{\top}G\left(y,u_{t}\right)u_{t}\right)}{y=x_{t}^{(i)},u_{t}=u_{t}^{(i)}}, \quad t=1,\dots,T-1, \\
    &\zeta_{t} := \restr{\nabla_{y}\left(u_{t}^{\top}G\left(x_{t},y\right)u_{t}\right)}{x_{t}=x_{t}^{(i)},u_{t}=u_{t}^{(i)}, y=u_{t}^{(i)}}, \quad t=1,\dots,T-1, \\
    &G_{t} := G\left(x_{t}^{(i)}, u_{t}^{(i)}\right), \quad t=0,\dots,T-1,
    \end{split}
\end{equation}
The related state variables (before line search) are found by the state equation, $x_{t+1} = x_{t}+u_{t}$ for $t=0,\dots,T-1$ with $x_{0}=a$. Using the same approach as in the proof of Proposition~\ref{prop:update_scheme}, then Eq.~\ref{eq:finsler_energy_zero_point_problem} is equivalent to
\begin{equation}
    \begin{split}
        &\mu_{T-1} = \left(\sum_{t=0}^{T-1}G_{t}^{-1}\right)^{-1}\left(2(a-b)-\sum_{t=0}^{T-1}G_{t}^{-1}\left(\zeta_{t}+\sum_{t>j}^{T-1}\nu_{j}\right)\right), \\
        &u_{t} = -\frac{1}{2}G_{t}^{-1}\left(\mu_{T-1}+\zeta_{t}+\sum_{j>t}^{T-1}\nu_{j}\right), \quad t=0,\dots,T-1, \\
        &x_{t+1} = x_{t}+u_{t}, \quad t=0,\dots,T-1, \\
        &x_{0}=a,
    \end{split}
\end{equation}
where it is exploited that the fundamental tensor is positive definite.
    \subsection{Extension to Finslerian manifolds} \label{ap:finsler_proof}
    The proof for global convergence hold similar for \textit{GEORCE} in the Finsler case. Let $E_{F}$ denote the corresponding discretized Finsler energy functional. For the Finsler case the optimally condition becomes
\begin{equation}
    \begin{split}
        \restr{\nabla_{x_{t}}E_{F}(x_{t},u_{t0})}{(x_{t},u_{t})=\left(x_{t}^{(i)}, u_{t}^{(i)}\right)} &= \mu_{T-1}-\mu_{t}, \quad t=1,\dots,T-1, \\
        \restr{\nabla_{u_{t}}E_{F}(x_{t},u_{t})}{(x_{t},u_{t})=\left(x_{t}^{(i)}, u_{t}^{(i+1)}\right)} &= 2G\left(x_{t}^{(i)},u_{t}^{(i)}\right)u_{t}^{(i+1)}+\zeta_{t}=-\mu_{t}, \\
        &\quad t=0,\dots,T-1
    \end{split}
    \label{eq:finsler_optimality}
\end{equation}
The proof of global convergence is completely similar for the Riemannian and Finsler case except that Eq.~\ref{eq:finsler_changed} in the Finsler case becomes
\begin{equation}
    \begin{split}
        \mu_{t} + \restr{\nabla_{u_{t}}E_{F}(x_{t},u_{t})}{(x_{t},u_{t})=\left(x_{t}^{(i)},u_{t}^{(i)}\right)} &= -\restr{\nabla_{u_{t}}E_{F}(x_{t},u_{t})}{(x_{t},u_{t})=\left(x_{t}^{(i+1)},u_{t}^{(i+1)}\right)} \\
        &+\restr{\nabla_{u_{t}}E_{F}(x_{t},u_{t})}{(x_{t},u_{t})=\left(x_{t}^{(i+1)},u_{t}^{(i)}\right)} \\
        &= -\left(G\left(x_{t}^{(i)},u_{t}^{(i)}\right)u_{t}^{(i+1)}+\zeta_{t}\right) \\
        &+\left(2G\left(x_{t}^{(i)},u_{t}^{(i)}\right)+\zeta_{t}\right) \\
        &= -2G\left(x_{t}^{(i)},u_{t}^{(i)}\right)\left(u_{t}^{(i+1)}-u_{t}^{(i)}\right),
    \end{split}
    \label{eq:finsler_update}
\end{equation}
by using the properties of \textit{GEORCE} in the Finsler case in Eq.~\ref{eq:finsler_optimality}. Eq.~\ref{eq:finsler_update} is completely similar to Eq.~\ref{eq:finsler_changed} in the Riemannian case, and from that the global convergence proof is completely similar to the Riemannian case.
    \section{Algorithms}
    \subsection{GEORCE for Finsler manifolds} \label{ap:georce_al_finsler}
    In Algorithm~\ref{al:georcef} we show \textit{GEORCE} for Finsler manifolds in pseudo-code using the update scheme in Eq.~\ref{eq:finsler_energy_update_schem}.

\begin{algorithm}[hbt]
    \caption{GEORCE for Finsler Manifolds}
    \label{al:georcef}
    \begin{algorithmic}[1]
        \State \textbf{Input}: $\mathrm{tol}$, $T$
        \State \textbf{Output}: Geodesic estimate $x_{0:T}$
        \State Set $x_{t}^{(0)}\leftarrow a+\frac{b-a}{T}t$, $u_{t}^{(0)}\leftarrow \frac{b-a}{T}$ for $t=0.,\dots,T-1$  and $i \leftarrow 0$
        \While{$\norm{\restr{\nabla_{y}E(y)}{y=x_{t}^{(i)}}}_{2} > \mathrm{tol}$}
        \State $G_{t} \leftarrow G\left(x_{t}^{(i)}, u_{t}^{(i)}\right)$ for $t=0,\dots,T-1$
        \State $\nu_{t} \leftarrow \restr{\nabla_{y}\left(u_{t}^{(i)}G\left(y, u_{t}^{(i)}\right)u_{t}^{(i)}\right)}{y=x_{t}^{(i)}}$ for $t=1,\dots,T-1$
        \State $\zeta_{t} \leftarrow \restr{\nabla_{y}\left(u_{t}^{(i)}G\left(x_{t}^{(i)}, y\right)u_{t}^{(i)}\right)}{y=u_{t}^{(i)}}$ for $t=1,\dots,T-1$
        \State $\mu_{T-1} \leftarrow \left(\sum_{t=0}^{T-1}G_{t}^{-1}\right)^{-1}\left(2(a-b)-\sum_{t=0}^{T-1}G_{t}^{-1}\left(\zeta_{t}+\sum_{t>j}^{T-1}\nu_{j}\right)\right)$
        \State $u_{t} \leftarrow -\frac{1}{2}G_{t}^{-1}\left(\mu_{T-1}+\zeta_{t}+\sum_{j>t}^{T-1}\nu_{j}\right)$ for $t=0,\dots,T-1$
        \State $x_{t+1} \leftarrow x_{t}+u_{t}$ for $t=0,\dots,T-1$
        \State Using line search find $\alpha^{*}$ for the following optimization problem with the discrete energy functional $E_{F}$
        \begin{equation*}
            \begin{split}
                \alpha^{*} = \argmin_{\alpha}\quad &E_{F}\left(x_{0:T}\right) \quad \text{(exact line-search)} \\
                \text{s.t.} \quad &x_{t+1}=x_{t}+\alpha u_{t}+(1-\alpha)u_{t}^{(i)}, \quad t=0,\dots,T-1, \\
                &x_{0}=a.
            \end{split}
        \end{equation*}
        \State Set $u_{t}^{(i+1)} \leftarrow \alpha^{*}u_{t}+(1-\alpha^{*})u_{t}^{(i)}$ for $t=0,\dots,T-1$
        \State Set $x_{t+1}^{(i+1)}\leftarrow x_{t}^{(i+1)}+u_{t}^{(i+1)}$ for $t=0,\dots,T-1$
        \State $i \leftarrow i+1$
        \EndWhile
        \State return $x_{t}$ for $t=0,\dots,T-1$
    \end{algorithmic}
\end{algorithm}
    \subsection{Sparse Newton Method} \label{ap:newton_method}
    In this section we derive a sparse Newton method for minimizing Eq.~\ref{eq:disc_const_energy}. We consider the Riemannian case, but note that it is easily generalized to the Finslerian case. The Hessian of the Eq.~\ref{eq:disc_const_energy} we get that
\begin{equation*}
    \begin{split}
        \partial^{2}_{x_{t}x_{t}}E\left(x_{0:T}\right) &= \partial^{2}_{x_{t}x_{t}}\restr{u_{t}^{\top}G(x_{t})u_{t}}{u_{t}=x_{t+1}-x_{t}} + 2G(x_{t-1}) - 2G(x_{t}) \\
        &+ 4 \nabla_{x_{t-1}}\restr{G(x_{t-1})u_{t}}{u_{t}=x_{t}-x_{t-1}}-4\nabla_{x_{t}}\restr{G(x_{t})u_{t}}{u_{t}=x_{t+1}-x_{t}}, \quad t=1,\dots,T-1, \\
        \partial^{2}_{x_{t}x_{t+1}}E\left(x_{0:T}\right) &= 2 \nabla_{x_{t}}\restr{G(x_{t})u_{t}}{u_{t}=x_{t+1}-x_{t}}, \quad t=1,\dots,T-2, \\
        \partial^{2}_{x_{t+1}x_{t}}E\left(x_{0:T}\right) &= 2 \restr{\left(G(x_{t})u_{t}\right)^{\top}}{u_{t}=x_{t+1}-x_{t}}, \quad t=1,\dots,T-2.
    \end{split}
\end{equation*}
where $\partial^{2}_{x_{i},x_{j}}$ denotes the second order derivative. From this we see that the Hessian can be written as
\begin{equation} \label{eq:newton_block_structure}
    H =
    \begin{pmatrix}
        H_{11} & H_{12} & 0 & 0 & 0 & 0 & 0 &\dots & 0 \\
        H_{21} & H_{22} & H_{23} & 0 & 0 & 0 & 0 & \dots & 0 \\
        0 & H_{23} & H_{33} & H_{34} & 0 & 0 & 0 & \dots & 0 \\
        0 & 0 & H_{43} & H_{33} & H_{45} & 0 & 0 & \dots & 0 \\
        0 & 0 & 0 & H_{54} & H_{55} & H_{56} & 0 & \dots & 0 \\
        0 & 0 & 0 & 0 & H_{65} & H_{66} & H_{67} & \dots & 0 \\
        \vdots & \ddots & \ddots & \ddots & \ddots & \ddots & \ddots & \ddots & \vdots\\ 
        0 & 0 & 0 & 0 & 0 & 0 & \dots &H_{(T-2)(T-2)} & H_{(T-1)(T-2)} \\
        0 & 0 & 0 & 0 & 0 & 0 & \dots & H_{(T-2)(T-1)} & H_{(T-1)(T-1)} \\
    \end{pmatrix},
\end{equation}
where we define
\begin{equation*}
    H_{ij} = \partial^{2}_{x_{i}x_{j}} E(x_{0:T}),
\end{equation*}
A Newton step has the following form at iteration $j$.
\begin{equation} \label{eq:newton_step}
    s^{(j)} = -H^{-1}\nabla_{x_{1:(T-1)}}E(x_{0:T}),
\end{equation}
To simplify notation, define $s_{i} := -\nabla_{x_{i}}E(x_{0:T})$. Using the block structure in Eq.~\ref{eq:newton_block_structure}, we can recursively solve Eq.~\ref{eq:newton_step} similar to the Thomas Algorithm \citep{thomas_al}.
\begin{equation} \label{eq:newton_solve_block}
    \begin{split}
        s_{i} &\leftarrow H_{ii}^{-1}s_{i}, \\
        P &\leftarrow H_{ii+1}^{\top}, \quad i+1 < T\\
        H_{ii+1} &\leftarrow H_{ii}^{-1}H_{ii+1}, \quad i+1 < T\\
        H_{i+1,i+1} &\leftarrow H_{i+1,i+1} - H_{i+1,i}H_{i,i+1} = H_{i+1,i+1}-PH_{i,i+1}, \quad i+1 < T\\
        s_{i+1} &\leftarrow s_{i+1} - H_{i+1,i}s_{i} = s_{i+1} - Ps_{i}, \quad i+1 < T.
    \end{split}
\end{equation}
such that the Newton step is updated for $i=T-2,\dots,1$
\begin{equation} \label{eq:newton_update_block}
    s_{i} \leftarrow s_{i} - H_{i,i+1}s_{i+1}.
\end{equation}
Note that we efficiently solve the equations in Eq.~\ref{eq:newton_solve_block} as
\begin{equation} \label{eq:newton_inversion}
    \begin{split}
        H_{ii}y_{i} &= \left(H_{ii+1} | s_{i} \right), \quad i=1,\dots,T-2, \\
        H_{T-1} y_{T-1} &= s_{T-1}, \\
    \end{split}
\end{equation}
such that only one linear system of equations has to be solved in each iteration. The Newton method does not have global convergence. In Algorithm~\ref{al:sparse_newton} we write the Sparse Newton algorithm in pseudo-code switching between the Newton method and gradient descent with line-search. Note that the above is easily extended to the Finslerian case, where $G$ denotes the fundamental tensor and depends on both $x$ and $u$.
\begin{algorithm}[hbt]
    \caption{Sparse Newton Method}
    \label{al:sparse_newton}
    \begin{algorithmic}[1]
        \State \textbf{Input}: $\mathrm{tol}$, $T$
        \State \textbf{Output}: Geodesic estimate $x_{0:T}$
        \State Set $x_{t}^{(0)}\leftarrow a+\frac{b-a}{T}t$ for $t=0.,\dots,T-1$  and $j \leftarrow 0$
        \While{$\norm{\restr{\nabla_{y}E(y)}{y=x_{t}^{(i)}}}_{2} > \mathrm{tol}$}
        \State Compute $\left\{s_{t}^{\mathrm{grad}}, s_{t}^{\mathrm{newton}}\right\}_{t=1,\dots,T-2}$ using Eq.~\ref{eq:newton_solve_block} and Eq.~\ref{eq:newton_update_block}.
        \If $\sum_{t=1}^{T-1}s_{t}^{\mathrm{grad}}\left(s_{t}^{\mathrm{newton}}\right)^{\top} \geq 0$
        \State $s_{t} = s_{t}^{\mathrm{newton}}$ for $t=1,\dots,T-1$ \\
        \Else
        \State $s_{t} = s_{t}^{\mathrm{grad}}$ for $t=1,\dots,T-1$ 
        \EndIf
        \State Perform line-search for $\alpha^{*} = \argmin_{0 \leq \alpha \leq 1} E\left(x_{1:(T-1)}+\alpha s_{1:(T-1)}\right)$ with 
        \begin{equation*}
            x_{t}^{(j+1)} \leftarrow x_{t}^{(j)} + \alpha^{*} s_{t}, \quad t=1,\dots,T-1.
        \end{equation*}
        \State $j \leftarrow j+1$
        \EndWhile
        \State return $x_{t}^{(j)}$ for $t=0,\dots,T-1$
    \end{algorithmic}
\end{algorithm}
Note that the sparse Newton algorithm in Algorithm~\ref{al:sparse_newton} has complexity $\mathcal{O}\left(Td^{3}\right)$ completely similar to GEORCE.

Since the Hessian of the energy is not necessarily positive definite, the Newton scheme can be numerically unstable. To avoid this, we represent in Algorithm~\ref{al:sparse_reg_newton} a regularized Newton method by adding suitable scaling in the diagonal.
\begin{algorithm}[hbt]
    \caption{Sparse Regularized Newton Method}
    \label{al:sparse_reg_newton}
    \begin{algorithmic}[1]
        \State \textbf{Input}: $\mathrm{tol}$, $T$, $\lambda$, $\kappa$ \\
        \State \textbf{Output}: Geodesic estimate $x_{0:T}$ \\
        \State Set $x_{t}^{(0)}\leftarrow a+\frac{b-a}{T}t$ for $t=0.,\dots,T-1$ and $j \leftarrow 0$ \\
        \While{$\norm{\restr{\nabla_{y}E(y)}{y=x_{t}^{(i)}}}_{2} > \mathrm{tol}$}
        \State $H_{ii} \leftarrow \lambda I + \partial^{2}_{x_{i}x_{i}}E(x_{0:T})$ \\
        \State $H_{ij} \leftarrow \partial^{2}_{x_{i}x_{j}}E(x_{0:T})$ \\
        \State Compute $\left\{s_{t}^{\mathrm{grad}}, s_{t}^{\mathrm{newton}}\right\}_{t=1,\dots,T-2}$ using Eq.~\ref{eq:newton_solve_block} and Eq.~\ref{eq:newton_update_block}. \\
        \If $\sum_{t=1}^{T-1}s_{t}^{\mathrm{grad}}\left(s_{t}^{\mathrm{newton}}\right)^{\top} \geq 0$ \\
        \State $s_{t} = s_{t}^{\mathrm{newton}}$ for $t=1,\dots,T-1$ \\
        \State $\lambda \leftarrow \kappa \lambda$
        \Else
        \State $s_{t} = s_{t}^{\mathrm{grad}}$ for $t=1,\dots,T-1$ \\
        \State $\lambda \leftarrow \sfrac{\lambda}{\kappa}$
        \EndIf
        \State Perform line-search for $\alpha^{*} = \argmin_{0 \leq \alpha \leq 1} E\left(x_{1:(T-1)}+\alpha s_{1:(T-1)}\right)$ with 
        \begin{equation*}
            x_{t}^{(j+1)} \leftarrow x_{t}^{(j)} + \alpha^{*} s_{t}, \quad t=1,\dots,T-1.
        \end{equation*}
        \State $j \leftarrow j+1$
        \EndWhile
        \State return $x_{t}^{(j)}$ for $t=0,\dots,T-1$
    \end{algorithmic}
\end{algorithm}
Note that the Newton method for the discretized energy was also proposed in \citep{noakes_join_geodesics}, which also formulates the Hessian as a block matrix similar to Eq.~\ref{eq:newton_block_structure}. However, they do not derive an adaptive scheme to solve the linear equation in the Newton step or propose the regularized version in Algorithm~\ref{al:sparse_reg_newton}.
    \clearpage
    \section{Choice of initialization curve} \label{ap:init_curve}
\begin{wrapfigure}[19]{r}{0.50\textwidth}
    \vspace{-5.0mm}
    \centering
    \includegraphics[width=0.50\textwidth]{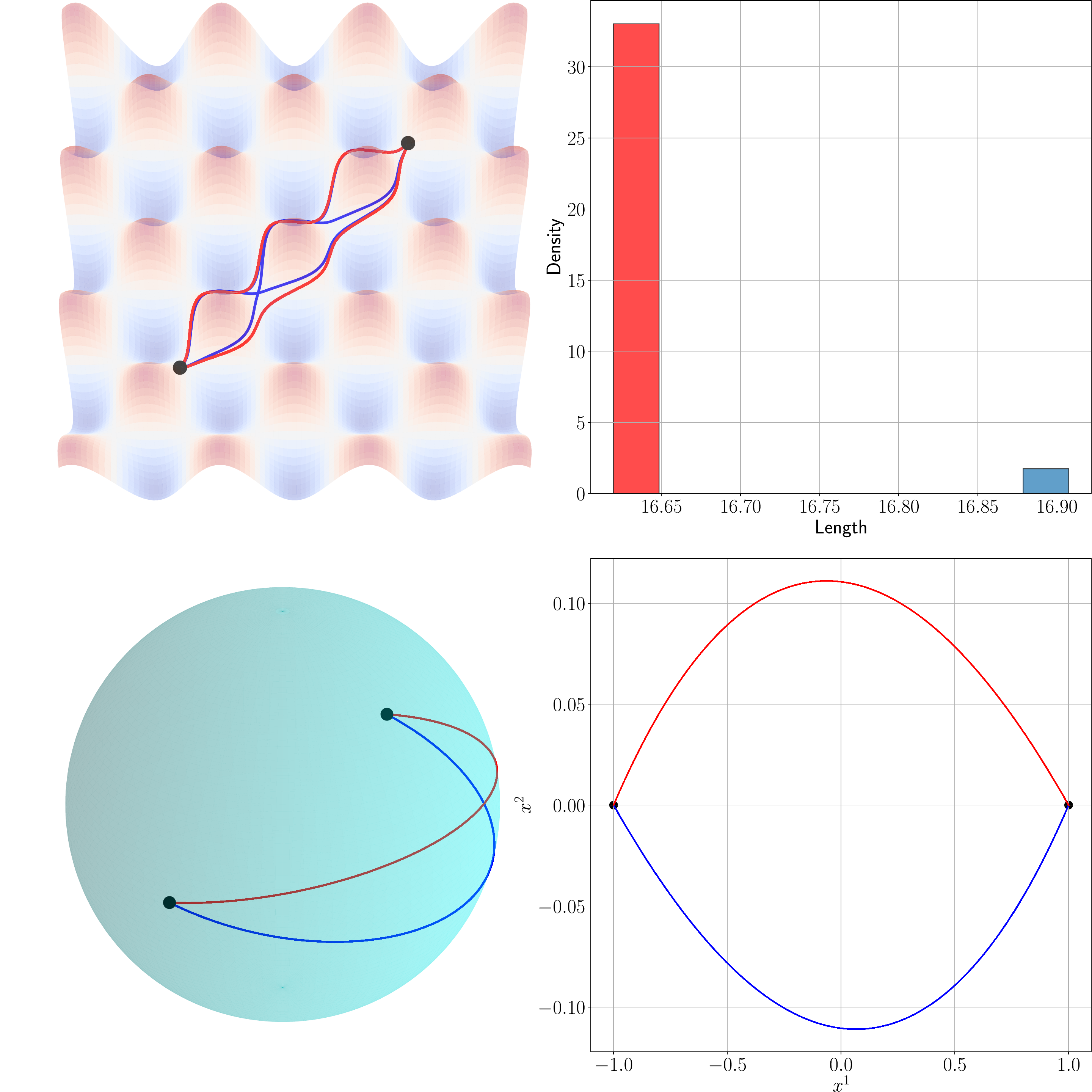}
    \vspace{-3.0mm}
    \caption{The first row shows the estimated curves using \textit{GEORCE} for the egg-tray, i.e. $f(x,y)=\left(x,y,2\cos x \cos y\right)$, while the second row shows the estimated curves for the sphere for different initializations of the initial curve.}
    \vspace{5.0mm}
    \label{fig:initilization_curve}
\end{wrapfigure}
In this section, we show the effect of the choice of initialization curve in \textit{GEORCE}. In Algorithm~\ref{al:georce} and the experiments, the initial curve is a straight line between the start and end point. However, the initial curve is not restricted to this and can in principle be any curve assuming that it is defined in the local chart. To illustrate the effect of the initial curve, we will consider the egg-tray with the parametrization $f(x,y)=\left(x,y,2\cos x \cos y\right)$, and the sphere for two antipodal points on the equator. For the egg-tray we consider setting the initial curve randomly as
\begin{equation*}
    x_{t}^{(0)} = t\frac{b-a}{T}+a + \epsilon_{t}, \quad \epsilon_{t} \sim \mathcal{N}\left(0,I\right),
\end{equation*}
while we consider for the sphere the random initial curve.
\begin{equation*}
    x_{t}^{(0)} = t\frac{b-a}{T}+a + \pmb{1}\mathrm{Sign}\left(\epsilon\right), \quad \epsilon \sim \mathcal{N}\left(0,1\right),
\end{equation*}
For each random initial curve, we estimate the energy-minimizing curve using \textit{GEORCE} and show the result in Fig.~\ref{fig:initilization_curve} for $100$ random initializations. The first row shows the results for the egg-tray, where the rightmost figure is the distribution of length for the computed curves. We display the corresponding curves in the same color as in the histogram. In general, we see that for a small set of curves \textit{GEORCE} has found a local minimum, where the curves move on the inside of the egg-tray rather than on the outside of the egg-tray, which appears to be the global minimum. This illustrates how \textit{GEORCE} is guaranteed to converge to a local minimum and not a global minimum, and that this can be checked by expecting different initialization curves. For the sphere, there are infinitely many length-minimizing curves, and since we add the sign of a standard normal distributed variable for the random curves, we see that \textit{GEORCE} tends to move in the direction of the closest local minimum.

    \section{Manifolds} \label{ap:manifold_description}
    Table~\ref{tab:manifold_description} contains a description of the Riemannian manifolds shown in the paper.

\begin{table}[ht]
    \centering
    \scriptsize
    \begin{tabular}{p{2cm} | p{5cm} | p{2cm} | p{4cm}}
    \textbf{Manifold} & \textbf{Description} & \textbf{Parameters} & \textbf{Local Coordinates} \\
    \hline
    $\mathbb{S}^{n}$ & The n-sphere, $\{x \in \mathbb{R}^{n+1} \;|\; ||x||_{2}=1\}$ & - & Stereographic coordinates \\
    \hline
    $E(n)$ & The n-Ellipsoid, $\{x \in \mathbb{R}^{n+1} \;|\; ||p\odot x||_{2}=1\}$ & Half-axes: $p=\left(0.5,\dots,1\right)$ equally distributed points. & Stereographic coordinates \\
    \hline
    $\mathbb{T}^{2}$ & Torus & Major radius $R=3.0$ and minor radius $r=1.0$ & Parameterized by $\left(\theta,\phi\right)$ with $x=(R+r\cos \theta)\cos\phi, y=(R+r\cos \theta) \sin \phi, z=r \sin \theta$. \\
    \hline
    $\mathbb{H}^{2}$ & 2 dimensional Hyperbolic Space embedded into Minkowski pace $\mathbb{R}^{3}$ & - & Parameterized by $\left(\alpha,\beta\right)$ with $x=\cosh\alpha, y=\sinh\alpha\cos\beta, z=\sinh\alpha\sin\beta$. \\
    \hline
    $P(n)$ & Paraboloid of dimension $n$ & - & Parameterized by standard coordinates by $\left(x^{1},\dots,x^{n},\sum_{i=1}^{n}x_{i}^{2}\right)$ \\
    \hline
    $\mathcal{P}(n)$ & Symmetric positive definite matrices of size $n^{2}$ & - & Embedded into $\mathbb{R}^{n^{2}}$ by the mapping $f(x) = l(x)l(x)^{T}$, where $l: \mathbb{R}^{n(n+1)/2} \rightarrow \mathbb{R}^{n \times n}$ maps $x$ into a lower triangle matrix consisting of the elements in $x$. \\
    \hline
    'Egg-tray' & A parameterized surface in shape of an egg tray & - & Parameterized by $(x,y)$ with $x=x, y=y, z=2\cos x \cos y$. \\
    \hline
    Gaussian Distribution & Parameters of the Gaussian distribution equipped with the Fischer-Rao metric & - & Parametrized by $(\mu,\sigma) \in \mathbb{R} \times \mathbb{R}_{+}$. \\
    \hline
    Fr\'echet Distribution & Parameters of the Fr\'echet distribution equipped with the Fischer-Rao metric & - & Parametrized by $(\beta,\lambda) \in \mathbb{R}_{+} \times \mathbb{R}_{+}$. \\
    \hline
    Cauchy Distribution & Parameters of the Cauchy distribution equipped with the Fischer-Rao metric & - & Parametrized by $(\mu,\sigma) \in \mathbb{R} \times \mathbb{R}_{+}$. \\
    \hline
    Pareto Distribution & Parameters of the Pareto distribution equipped with the Fischer-Rao metric & - & Parametrized by $(\theta,\alpha) \in \mathbb{R}_{+} \times \mathbb{R}_{+}$. \\
    \hline
    VAE MNIST & Variational-Autoencoder equipped with the pull-back metric for MNIST-data & - & Parametrized by the encoded latent space. \\
    \hline
    VAE SVHN & Variational-Autoencoder equipped with the pull-back metric for SVHN-data & - & Parametrized by the encoded latent space. \\
    \hline
    VAE CelebA & Variational-Autoencoder equipped with the pull-back metric for CelebA-data & - & Parametrized by the encoded latent space. \\
    \hline
    \end{tabular}
    \caption{Description of the manifolds used in the paper.}
    \label{tab:manifold_description}
\end{table}

In the Finsler cases we consider a Riemannian manifold with a background metric, $G$, equipped with a force-field.

\begin{equation*}
    f(x) = \frac{\sin x \odot \cos x}{(\cos x)^{\top} G(x) \cos x},
\end{equation*}

where $x$ denotes the local coordinates of the manifold. In this case the corresponding metric is a Randers metric, which is a special case of Riemannian metrics, of the form \citep{Piro_2021}

\begin{equation}
    \begin{split}
        F &= \sqrt{a_{ij}\dot{r}^{i}\dot{r}^{j}}+b_{i}\dot{r}^{i} \\
        a_{ij} &= g_{ij} \lambda + f_{i}f_{j}\lambda^{2} \\
        b_{i} &= -f_{i}\lambda \\
        f_{i} &= g_{ij}f^{j} \\
        \lambda^{-1} &= v_{0}^{2}-g_{ij}f^{i}f^{j},
    \end{split}
\end{equation}

where $g_{ij}$ denotes the elements of the Riemannian background metric, $f^{j}$ denotes the $j$th element of the force field, while $v_{0}$ denotes the initial velocity Riemannian length of an object moving with constant velocity $v$ on the surface. For the manifold to be a proper Finsler manifold it is required that $||f||_{g} < v_{0}$.
    \section{Variational autoencoders} \label{ap:vae}
    \subsection{Architecture} \label{ap:vae_architecture}
    We train a Variational-Autoencoder (VAE) \citep{kingma2022autoencodingvariationalbayes}, where we use the normal distribution as a variational family to approximate the data distribution. We train a VAE for each of the following datasets

\begin{itemize}
    \item MNIST data \citep{deng2012mnist}: The data consists of $28 \times 28$ images of handwritten digits between 0 and 9. The data consists of $60,000$ training images and $10,000$ testing images.
    \item SVHN data \citep{svhn}: The data consists of $32 \times 32 \times 3$ images of house numbers obtained through Google street view. The data consists of $73,257$ digits for training and $26,032$ digits for testing. We use approximately $80\%$ of the training data for training similar to \citep{shao2017riemannian}.
    \item CelebA data \citep{liu2015faceattributes}: The data consists of $202,599$ images of famous people of size $64 \times 64 \times 3$, where we use approximately $80\%$ for training similar to \citep{shao2017riemannian}.
\end{itemize}

The architectures for the VAE for MNIST data, SVHN data and CelebA data are shown in Table~\ref{tab:mnist_vae_architecture}, Table~\ref{tab:svhn_vae_architecture} and Table~\ref{tab:celeba_vae_architecture}, respectively. The architecture is similar to \citep{shao2017riemannian} with some modifications. All models are trained for $50,000$ epochs on a GPU described in Appendix~\ref{ap:hardware}.

\begin{table}[!ht]
    \centering
    \begin{tabular}{c}
        \hline
        \multicolumn{1}{c}{\textbf{Encoder}} \\
        \hline
        $\mathrm{Conv}\left(\mathrm{output\_channels}=64, \mathrm{kernel\_shape}=4 \times 4, \mathrm{stride}=2, \mathrm{bias}=\mathrm{False}\right)$ \\ 
        $\mathrm{GELU}$ \\
        $\mathrm{Conv}\left(\mathrm{output\_channels}=64, \mathrm{kernel\_shape}=4 \times 4, \mathrm{stride}=2, \mathrm{bias}=\mathrm{False}\right)$ \\ 
        $\mathrm{GELU}$ \\
        $\mathrm{Conv}\left(\mathrm{output\_channels}=64, \mathrm{kernel\_shape}=4 \times 4, \mathrm{stride}=1, \mathrm{bias}=\mathrm{False}\right)$ \\
        $\mathrm{GELU}$ \\
        $\mathrm{Conv}\left(\mathrm{output\_channels}=64, \mathrm{kernel\_shape}=4 \times 4, \mathrm{stride}=2, \mathrm{bias}=\mathrm{False}\right)$ \\
        $\mathrm{GELU}$ \\
        \hline
        \multicolumn{1}{c}{\textbf{Latent Parameters}} \\
        \hline
        $\mu$: $\mathrm{Linear}(\mathrm{out\_feature}=8, \mathrm{bias}=\mathrm{True})$, $\mathrm{Identity}$ \\
        $\sigma$: $\mathrm{Linear}(\mathrm{out\_feature}=8, \mathrm{bias}=\mathrm{True})$, $\mathrm{Sigmoid}$ \\
        \hline
        \multicolumn{1}{c}{\textbf{Decoder}} \\
        \hline
        $\mathrm{Linear}\left(\mathrm{out\_feature}=50, \mathrm{bias}=\mathrm{True}\right)$ \\
        $\mathrm{GELU}$ \\
        $\mathrm{Conv2dTransposed}\left(\mathrm{output\_channels}=64, \mathrm{kernel\_shape}=4 \times 4, \mathrm{stride}=2, \mathrm{bias}=\mathrm{False}\right)$ \\
        $\mathrm{GELU}$ \\
        $\mathrm{Conv2dTransposed}\left(\mathrm{output\_channels}=32, \mathrm{kernel\_shape}=4 \times 4, \mathrm{stride}=1, \mathrm{bias}=\mathrm{False}\right)$ \\
        $\mathrm{GELU}$ \\
        $\mathrm{Conv2dTransposed}\left(\mathrm{output\_channels}=16, \mathrm{kernel\_shape}=4 \times 4, \mathrm{stride}=1, \mathrm{bias}=\mathrm{False}\right)$ \\
        $\mathrm{GELU}$ \\
        $\mathrm{Linear}\left(\mathrm{out\_feature}=784, \mathrm{bias}=\mathrm{True}\right)$ \\
        \hline
    \end{tabular}
    \caption{Architecture for VAE for the MNIST-dataset}
    \label{tab:mnist_vae_architecture}
\end{table}

\begin{table}[!ht]
    \centering
    \begin{tabular}{c}
        \hline
        \multicolumn{1}{c}{\textbf{Encoder}} \\
        \hline
        $\mathrm{Conv}\left(\mathrm{output\_channels}=32, \mathrm{kernel\_shape}=2 \times 2, \mathrm{stride}=2, \mathrm{bias}=\mathrm{False}\right)$ \\ 
        $\mathrm{GELU}$ \\
        $\mathrm{Conv}\left(\mathrm{output\_channels}=32, \mathrm{kernel\_shape}=2 \times 2, \mathrm{stride}=2, \mathrm{bias}=\mathrm{False}\right)$ \\ 
        $\mathrm{GELU}$ \\
        $\mathrm{Conv}\left(\mathrm{output\_channels}=64, \mathrm{kernel\_shape}=2 \times 2, \mathrm{stride}=2, \mathrm{bias}=\mathrm{False}\right)$ \\
        $\mathrm{GELU}$ \\
        $\mathrm{Conv}\left(\mathrm{output\_channels}=64, \mathrm{kernel\_shape}=2 \times 2, \mathrm{stride}=2, \mathrm{bias}=\mathrm{False}\right)$ \\
        $\mathrm{GELU}$ \\
        \hline
        \multicolumn{1}{c}{\textbf{Latent Parameters}} \\
        \hline
        $\mu$: $\mathrm{Linear}(\mathrm{out\_feature}=32, \mathrm{bias}=\mathrm{True})$, $\mathrm{Identity}$ \\
        $\sigma$: $\mathrm{Linear}(\mathrm{out\_feature}=32, \mathrm{bias}=\mathrm{True})$, $\mathrm{Sigmoid}$ \\
        \hline
        \multicolumn{1}{c}{\textbf{Decoder}} \\
        \hline
        $\mathrm{Conv2dTransposed}\left(\mathrm{output\_channels}=64, \mathrm{kernel\_shape}=2 \times 2, \mathrm{stride}=2, \mathrm{bias}=\mathrm{False}\right)$ \\
        $\mathrm{GELU}$ \\
        $\mathrm{Conv2dTransposed}\left(\mathrm{output\_channels}=64, \mathrm{kernel\_shape}=2 \times 2, \mathrm{stride}=2, \mathrm{bias}=\mathrm{False}\right)$ \\
        $\mathrm{GELU}$ \\
        $\mathrm{Conv2dTransposed}\left(\mathrm{output\_channels}=32, \mathrm{kernel\_shape}=2 \times 2, \mathrm{stride}=2, \mathrm{bias}=\mathrm{False}\right)$ \\
        $\mathrm{GELU}$ \\
        $\mathrm{Conv2dTransposed}\left(\mathrm{output\_channels}=32, \mathrm{kernel\_shape}=2 \times 2, \mathrm{stride}=2, \mathrm{bias}=\mathrm{False}\right)$ \\
        $\mathrm{GELU}$ \\
        $\mathrm{Conv2dTransposed}\left(\mathrm{output\_channels}=3, \mathrm{kernel\_shape}=2 \times 2, \mathrm{stride}=2, \mathrm{bias}=\mathrm{False}\right)$ \\
        \hline
    \end{tabular}
    \caption{Architecture for VAE for the SVHN-dataset}
    \label{tab:svhn_vae_architecture}
\end{table}

\begin{table}[!ht]
    \centering
    \begin{tabular}{c}
        \hline
        \multicolumn{1}{c}{\textbf{Encoder}} \\
        \hline
        $\mathrm{Conv}\left(\mathrm{output\_channels}=32, \mathrm{kernel\_shape}=4 \times 4, \mathrm{stride}=2, \mathrm{bias}=\mathrm{False}\right)$ \\ 
        $\mathrm{GELU}$ \\
        $\mathrm{Conv}\left(\mathrm{output\_channels}=32, \mathrm{kernel\_shape}=4 \times 4, \mathrm{stride}=2, \mathrm{bias}=\mathrm{False}\right)$ \\ 
        $\mathrm{GELU}$ \\
        $\mathrm{Conv}\left(\mathrm{output\_channels}=64, \mathrm{kernel\_shape}=4 \times 4, \mathrm{stride}=2, \mathrm{bias}=\mathrm{False}\right)$ \\
        $\mathrm{GELU}$ \\
        $\mathrm{Conv}\left(\mathrm{output\_channels}=64, \mathrm{kernel\_shape}=4 \times 4, \mathrm{stride}=2, \mathrm{bias}=\mathrm{False}\right)$ \\
        $\mathrm{GELU}$ \\
        \hline
        \multicolumn{1}{c}{\textbf{Latent Parameters}} \\
        \hline
        $\mu$: $\mathrm{Linear}(\mathrm{out\_feature}=32, \mathrm{bias}=\mathrm{True})$, $\mathrm{Identity}$ \\
        $\sigma$: $\mathrm{Linear}(\mathrm{out\_feature}=32, \mathrm{bias}=\mathrm{True})$, $\mathrm{Sigmoid}$ \\
        \hline
        \multicolumn{1}{c}{\textbf{Decoder}} \\
        \hline
        $\mathrm{Conv2dTransposed}\left(\mathrm{output\_channels}=64, \mathrm{kernel\_shape}=4 \times 4, \mathrm{stride}=2, \mathrm{bias}=\mathrm{False}\right)$ \\
        $\mathrm{GELU}$ \\
        $\mathrm{Conv2dTransposed}\left(\mathrm{output\_channels}=64, \mathrm{kernel\_shape}=4 \times 4, \mathrm{stride}=2, \mathrm{bias}=\mathrm{False}\right)$ \\
        $\mathrm{GELU}$ \\
        $\mathrm{Conv2dTransposed}\left(\mathrm{output\_channels}=32, \mathrm{kernel\_shape}=4 \times 4, \mathrm{stride}=2, \mathrm{bias}=\mathrm{False}\right)$ \\
        $\mathrm{GELU}$ \\
        $\mathrm{Conv2dTransposed}\left(\mathrm{output\_channels}=32, \mathrm{kernel\_shape}=4 \times 4, \mathrm{stride}=2, \mathrm{bias}=\mathrm{False}\right)$ \\
        $\mathrm{GELU}$ \\
        $\mathrm{Conv2dTransposed}\left(\mathrm{output\_channels}=3, \mathrm{kernel\_shape}=4 \times 4, \mathrm{stride}=4, \mathrm{bias}=\mathrm{False}\right)$ \\
        \hline
    \end{tabular}
    \caption{Architecture for VAE for the CelebA-dataset}
    \label{tab:celeba_vae_architecture}
\end{table}
    \subsection{Additional experiments} \label{ap:vae_experiments}
    In Fig.~\ref{fig:mnist_8} and Fig.~\ref{fig:svhn_32} we show the geodesics between two reconstructed images using the learned VAE for the MNIST VAE and SVHN VAE, respectively, where we show the results using \textit{GEORCE}, the initialization and alternative methods.

\begin{figure}[t!]
    \centering
    \includegraphics[width=1.0\textwidth]{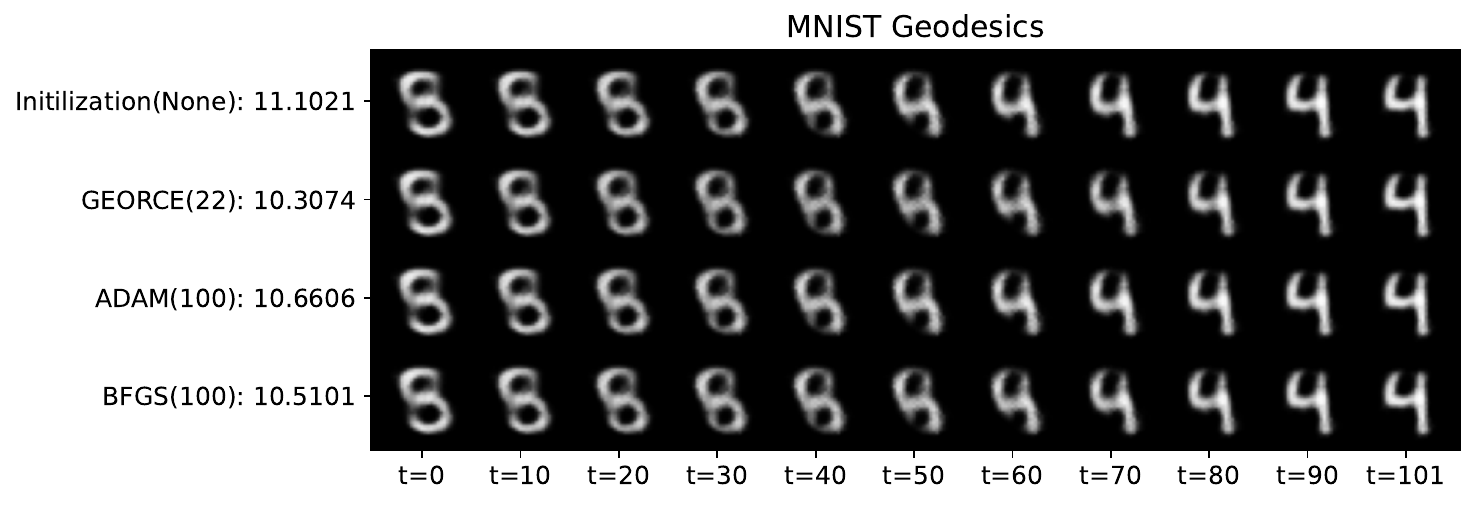}
    \caption{The figure shows the results of using \textit{GEORCE} and alternative methods for constructing geodesics for a Variational-Autoencoder \citep{kingma2022autoencodingvariationalbayes} (VAE) for the MNIST dataset \citep{deng2012mnist} All algorithms are terminated if the $\ell^{2}$-norm of the gradient of the discretized energy functional in Eq.~\ref{eq:disc_const_energy} is less than $10^{-3}$, or if the number of iterations exceeds $100$. We see that \textit{GEORCE} uses less iterations and obtains the smallest length compared to the other algorithms. The details of the manifolds and experiments can be found in the supplementary material.}
    \label{fig:mnist_8}
    \vspace{-1.0em}
\end{figure}

\begin{figure}[t!]
    \centering
    \includegraphics[width=1.0\textwidth]{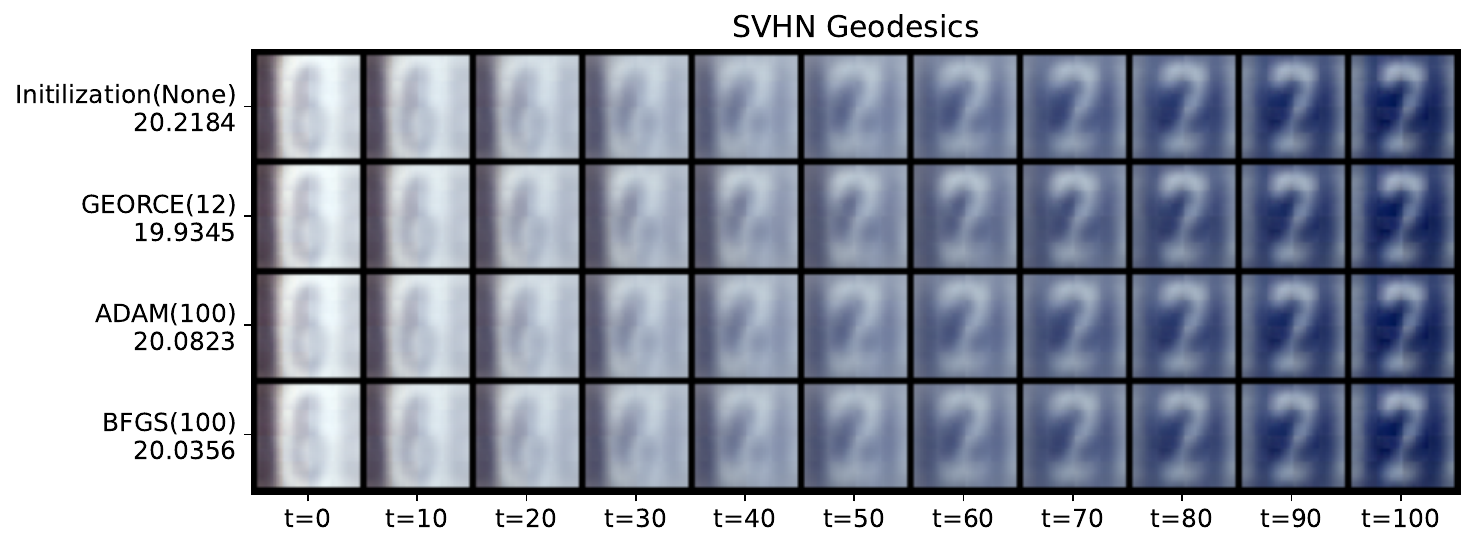}
    \caption{The figure shows the results of using \textit{GEORCE} and alternative methods for constructing geodesics for a Variational-Autoencoder \citep{kingma2022autoencodingvariationalbayes} (VAE) for the SVHN dataset \citep{svhn}. All algorithms are terminated if the $\ell^{2}$-norm of the gradient of the discretized energy functional in Eq.~\ref{eq:disc_const_energy} is less than $10^{-3}$, or if the number of iterations exceeds $100$. We see that \textit{GEORCE} uses less iterations and obtains the smallest length compared to the other algorithms. The details of the manifolds and experiments can be found in the supplementary material.}
    \label{fig:svhn_32}
    \vspace{-1.0em}
\end{figure}
    \section{Experiments} \label{ap:experiments}
    The following contains details of hyper-parameters, hardware and additional experiments. The implementation and code can be found at \\ \url{https://github.com/FrederikMR/georce}.
    \subsection{Hyper-parameters and methods} \label{ap:hyper_parameters}
    Table~\ref{tab:hyper_parameters} shows the hyper parameters used for the different methods. It is the same hyper-parameters that have been used for all manifolds across dimension and grid size in the paper.

\begin{table}[!ht]
    \centering
    \scriptsize
    \begin{tabular}{p{3.8cm} |p{4cm} | p{4cm}}
    \textbf{Method} & \textbf{Description} & \textbf{Parameters} \\
    \hline
    \textit{GEORCE} & See algorthm in main paper & Backtracking: $\rho=0.5$. \\
    \hline
    \textit{ADAM} \citep{kingma2017adam} & Moments based gradient descent method & $\gamma=0.01$ (step size), $\beta_{1}=0.9$, $\beta_{2}=0.999$, $\epsilon=10^{-8}$ \\
    \hline
    \textit{SGD} \citep{ruder2017overviewgradientdescentoptimization} & Stochastic gradient desent & $\gamma=0.01$ (step size), momentum=0, dampening=0, weight\_decay=0. \\
    \hline 
    \textit{BFGS} \citep{broyden_bfgs, fletcher_bfgs, Goldfarb1970AFO, shanno_bfgs} & Quasi-Newton method & - \\
    \hline 
    \textit{CG} \citep{nocedal2006numerical} & Conjugate-Gradient method & - \\
    \hline 
    \textit{dogleg} \citep{nocedal2006numerical} & trust region algorithm & - \\
    \hline 
    \textit{trust\_ncg} \citep{nocedal2006numerical} & Newton conjugate trust-region algorithm & - \\
    \hline 
    \textit{trust\_exact} \citep{nocedal2006numerical} & nearly exact trust region algorithm & - \\
    \hline
    \textit{Sparse Newton} & See Appendix~\ref{ap:newton_method} & Backtracking: $\rho=0.5$. \\
    \hline
    \textit{Sparse Regularized Newton} & See Appendix~\ref{ap:newton_method} & Backtracking: $\rho=0.5$ with $\lambda = 1.0$ and $\kappa=0.5$. \\
    \hline 
    \end{tabular}
    \caption{The Hyper-parameters for estimating of the geodesics. We refer to \text{Scipy} \citep{2020SciPy-NMeth} for the description and implementation of \textit{BFGS}, \textit{CG}, \textit{dogleg}, \textit{trust\_ncg} and \textit{trust\_exact}.}
    \label{tab:hyper_parameters}
\end{table}

Table~\ref{tab:ode_methods} shows the integration methods used for solving the shooting problem.

\begin{table}[!ht]
    \centering
    \scriptsize
    \begin{tabular}{p{3.8cm} |p{4cm} | p{4cm}}
    \textbf{Method} & \textbf{Description} & \textbf{Parameters} \\
    \hline
    \textit{RK45} \citep{rk45} & The explicit Runge-Kutta of order four to five & - \\
    \hline
    \textit{RK23} \citep{rk32} & The explicit Runge-Kutta of order two to three & - \\
    \hline
    \textit{DOP853} \citep{dop853} & The explicit Runge-Kutta of order eight & - \\
    \hline 
    \textit{Radau} \citep{radau} & Implicit Runge-Kutta method using Radau method & - \\
    \hline 
    \textit{BDF} \citep{bdf} & Implicit multi-step order using backward differentiation & - \\
    \hline 
    \textit{LSODA} \citep{odepack} & Method combining BDF and ADAM with stifness detection & - \\
    \hline 
    \end{tabular}
    \caption{The \textsc{ode} methods. The descriptions are from the \textit{Scipy}-library \citep{2020SciPy-NMeth}.}
    \label{tab:ode_methods}
\end{table}

All the timing estimates are computed by compiling the algorithms first, and then running the algorithm five times to compute the mean and standard deviation of the runtime. Table~\ref{tab:manifold_points} states the start and end point of the geodesics constructed on the different manifolds in a local coordinate system as described in Table~\ref{tab:manifold_description}.

\begin{table}[!ht]
    \centering
    \scriptsize
    \begin{tabular}{p{3cm} | p{5cm} | p{3cm}}
    \textbf{Manifold} & \textbf{Start point} & \textbf{End point} \\
    \hline
    $\mathbb{S}^{n}$ & $n$ equally spaced points between $0$ and $1$ (end point is excluded) & $\left(0.5,0.5,\dots,0.5\right)$ \\
    \hline
    $E(n)$ & $n$ equally spaced points between $0$ and $1$ (end point is excluded) & $\left(0.5,0.5,\dots,0.5\right)$ \\
    \hline
    $\mathbb{T}^{2}$ & $(0.0,0.0)$ & $(5\pi/4,5\pi/4)$ \\
    \hline
    $\mathbb{H}^{2}$ & $(1.0,1.0)$ & $(0.1,0.1)$ \\
    \hline
    Paraboloid & $(1,1)$ & $(0,0.5)$  \\
    \hline
    $\mathcal{P}(n)$ & The identity matrix in local coordinates & $n(n+1)/2$ equally spaced points between $0.5$ and $1.0$. \\
    \hline
    'Egg-tray' & $(-5.0,5.0)$ & $(5.0,5.0)$  \\
    \hline
    Gaussian Distribution & $\left(\mu_{0},\sigma_{0}\right) = \left(-1.0,0.5\right)$ & $\left(\mu_{T},\sigma_{T}\right) = \left(1.0,1.0\right)$  \\
    \hline
    Fréchet Distribution & $\left(\beta_{0},\lambda_{0}\right) = \left(0.5,0.5\right)$ & $\left(\beta_{T},\lambda_{T}\right) = \left(1.0,1.0\right)$  \\
    \hline
    Cauchy Distribution & $\left(\mu_{0},\sigma_{0}\right) = \left(-1.0,0.5\right)$ & $\left(\mu_{T},\sigma_{T}\right) = \left(1.0,1.0\right)$  \\
    \hline
    Pareto Distribution & $\left(\theta_{0},\alpha_{0}\right) = \left(0.5,0.5\right)$ & $\left(\theta_{T},\alpha_{T}\right) = \left(1.0,1.0\right)$  \\
    \end{tabular}
    \caption{The start and end point for the geodesics estimated for each manifold.}
    \label{tab:manifold_points}
\end{table}
    \subsection{Hardware} \label{ap:hardware}
    All figures have been computed on a \textit{HP} computer with Intel Core i9-11950H 2.6 GHz 8C, 15.6'' FHD, 720P CAM, 32 GB (2$\times$16GB) DDR4 3200 So-Dimm, Nvidia Quadro TI2000 4GB Descrete Graphics, 1TB PCle NVMe SSD, 150W PSU, 8cell, W11Home 64 Advanced.

The runtime experiments on a CPU have been run on a CPU model \textit{XeonE5\_2660v3} with $1$ nodes for at most $24$ hours with a maximum memory of $10$ GB.

The runtime estimates on a GPU and training of the VAE have been computed on a GPU for at most 24 hours with a maximum memory of $16$ GB. The $GPU$ consists of $1$ node on a \textit{Tesla V100}.
\end{appendices}


\end{document}